\documentclass[11pt,oneside]{article}
\usepackage[letterpaper]{geometry}
\usepackage{amsmath}
\usepackage{amsthm}
\usepackage{amssymb}
\usepackage{mathrsfs}
\usepackage{dsfont}
\usepackage{bbm}
\usepackage{lineno}
\usepackage{xcolor}
\usepackage[normalem]{ulem}
\usepackage{charter}
\usepackage{amsfonts}
\usepackage{graphicx}

\usepackage{caption}
\usepackage{subcaption}

\makeatletter
\theoremstyle{plain}
\newtheorem*{question*}{\protect\questionname}
\theoremstyle{plain}
\newtheorem{example}{\protect\examplename}
\theoremstyle{plain}
\newtheorem{thm}{\protect\theoremname}
\theoremstyle{plain}

\theoremstyle{plain}
\newtheorem{assumption}[thm]{\protect\assumptionname}
\theoremstyle{definition}
\newtheorem{defn}[thm]{\protect\definitionname}
\theoremstyle{definition}

\theoremstyle{remark}
\newtheorem{rem}[thm]{\protect\remarkname}
\theoremstyle{plain}
\newtheorem{lem}[thm]{\protect\lemmaname}

\DeclareMathOperator*{\argmin}{arg\,min}
\DeclareMathOperator*{\sqmin}{\square\min}
\DeclareMathOperator*{\sqinf}{\square\inf}
\DeclareMathOperator*{\sqsup}{\square\sup}

\providecommand{\assumptionname}{Assumption}
\providecommand{\examplename}{Example}
\providecommand{\definitionname}{Definition}
\providecommand{\corollaryname}{Corollary}
\providecommand{\lemmaname}{Lemma}
\providecommand{\questionname}{Question}
\providecommand{\remarkname}{Remark}
\providecommand{\theoremname}{Theorem}

\renewcommand{\Re}{\mathbb{R}}

\renewcommand{\Re}{\mathbb{R}}

\newcommand{\pr}[1]{\mathds{P}\left\{#1\right\}}

\title{Dynamic Capital Requirements for Markov Decision Processes}
\author{William B. Haskell, Abhishek Gupta, and Shiping Shao}
\date{\today}

\begin{document}

\maketitle
\begin{abstract}
We build on the theory of capital requirements (CRs) to create a new framework for modeling dynamic risk preferences.
The key question is how to evaluate the risk of a payoff stream sequentially as new information is revealed.
In our model, we associate each payoff stream with a disbursement strategy and a premium schedule to form a triple of stochastic processes.
We characterize risk preferences in terms of a single set that we call the {\it risk frontier} which characterizes acceptable triples.
We then propose the generalized capital requirement (GCR) which evaluates the risk of a payoff stream by minimizing the premium schedule over acceptable triples.
We apply this model to a risk-aware decision maker (DM) who controls a Markov decision process (MDP) and wants to find a policy to minimize the GCR of its payoff stream.
The resulting GCR-MDP recovers many well-known risk-aware MDPs as special cases.
To make this approach computationally viable, we obtain the temporal decomposition of the GCR in terms of the risk frontier. Then, we connect the temporal decomposition with the notion of an information state to compactly capture the dependence of DM's risk preferences on the problem history, where augmented dynamic programming can be used to compute an optimal policy.
We report numerical experiments for the GCR-minimizing newsvendor.
\end{abstract}



\section{Introduction}
Markov decision processes (MDPs) are a ubiquitous model for dynamic optimization \cite{bertsekas2012dynamic,puterman2014markov}. We consider a decision maker (DM) who controls an underlying discrete time finite horizon MDP which generates a stochastic payoff stream.
The classical MDP model is risk-neutral and maximizes expected total payoff. However, in many applications DM may have performance targets or reliability requirements that go beyond expected performance. It is essential to incorporate these additional requirements into DM's dynamic optimization model.

There are two major challenges to doing risk-aware dynamic optimization. First, characterizing DM's risk preferences in the dynamic setting is a much more complex task than in the static setting. We only need to model the risk of a single payoff in the static setting. In the dynamic setting, the payoffs are sequential so both the timing and the amounts matter. We need to model the evolution of DM's risk preferences and their dependence on the problem history, while also ensuring that they are consistent over time. Second, it is more difficult numerically to solve risk-aware MDPs compared to their static counterparts. The well known curse of dimensionality enters for MDPs with large state spaces, and risk-aware MDPs with history dependence only exacerbate this difficulty.

Much work on risk-aware MDPs focuses on nested risk measures (also called Markov/iterated risk measures), which are originally developed in \cite{ruszczynski2010risk,ruszczynski2014erratum}. This class of dynamic risk measures leads to dynamic programming (DP) decompositions on the same state space as the original risk-neutral problem. To use nested risk measures, DM's preferences must be elicited to identify an appropriate one-step risk measure for all time periods and states. These individual one-step risk measures are then composed to form a dynamic risk measure on the payoff stream.
However, in \cite{pflug2016time,pflug2016timea} the authors argue that nested risk measures are hard to elicit/interpret/justify and that they are not what decision makers understand as dynamic risk or acceptability. In addition, in \cite{li2017quantile} the authors demonstrate that nested risk measures can lead to overly conservative risk assessments.

Several alternative risk-aware MDP models have been proposed that optimize a risk measure of the cumulative payoff (e.g., expected utility \cite{bauerle2013more}, conditional value-at-risk \cite{chow2015risk,pflug2016time}, quantile risk measures \cite{li2017quantile}, etc.). These works all develop a DP decomposition on an augmented state space to compute the optimal policy. Yet, these models reflect different decision-theoretic paradigms, and there is no universally established way for DM to choose among them.

The theory of capital requirements (CRs) offers another way to model dynamic risk preferences, and it unifies the approaches just mentioned.
CRs (see \cite{frittelli2006risk}) build on monetary risk measures by allowing more general financial apparatus and financial premiums (than fixed cash payments).
The CR framework also facilitates preference elicitation for dynamic optimization, where DM can directly articulating his performance targets for the payoff stream. Then, he can evaluate its risk by finding the lowest cost disbursement strategy for which the payoff stream meets the targets.

Our goal in the present paper is to extend the existing CR framework to `generalized capital requirements' (GCRs), which cover a broader expression of dynamic risk preferences.
We additionally want to deploy GCRs in MDPs to create the class of GCR-MDPs and to obtain their DP decomposition.

\subsection{Motivating Examples}

It is natural to combine MDPs and CRs to express DM's risk preferences in many dynamic optimization problems.

\textbf{Manufacturing.} A plant manager must allocate manpower and resources to maximize production. The plant earns revenue from daily sales, and the manager is given cash-flow targets.
He must determine an initial investment in machinery, production capacity, and raw materials.
Given a fixed payoff stream, the financial premium is the amount of initial capital needed at the beginning of the planning horizon to achieve the cash-flow targets. The manager's problem is to dynamically optimize the plant's activities to minimize the amount of initial capital needed. 

\textbf{Insurance.} An insurance firm's daily profit depends on which claims are settled and how many new customers are acquired.
The firm must keep sufficient working capital in reserve to cover settled claims, where the reserve threshold is determined by regulators. It can sell assets from its portfolio of long-term investments to supplement on-hand cash.
The financial premium consists of the capital gains taxes and other opportunity costs from sale of long-term investments. 

\textbf{Consumption.} An individual makes consumption and investment decisions to meet their savings targets. At the beginning of every week, the individual decides which work contracts to accept and how many hours to spend to generate (possibly stochastic) income.
Then, the individual must determine how much to consume and withdraw/invest in savings.
The financial premium consists of income taxes and insurance against consumption shortfall.

\textbf{Renewable energy.} A renewable energy firm manages a portfolio with multiple types of storage assets.
The company's income is stochastic since it depends on the market price of electricity, output from renewables, and demand.
Each day, these uncertainties are realized and the firm decides how much to deposit/extract from storage assets.
The firm also has reserve requirements for meeting demand surges. 
The financial premium is the amount of initial capital injected at the beginning of the planning horizon for investment in new generators and storage assets.

\subsection{Contributions}
We make the following contributions in the present paper.

\textbf{New family of GCRs.} This work is built on CRs, which are developed in \cite{scandolo2004models,frittelli2006risk} as a tool for evaluating the risk of payoff streams.
We generalize the CRs developed in \cite{frittelli2006risk} by combining the three CR ingredients (acceptance set, set of financial strategies, and financial cost function) into a single set that we call the {\it risk frontier}.
The risk frontier fully characterizes acceptable triples of payoffs, disbursements, and premiums.
The GCR then evaluates risk by minimizing the premium over acceptable triples.
This setup generalizes the model in \cite{frittelli2006risk} while preserving its nice features.

\textbf{Time consistency.} We identify sufficient conditions for the time consistency of GCRs.
To obtain a temporal decomposition, we apply lattice theory to justify the interchange of the order of minimization in GCRs to obtain their recursive form.

\textbf{Risk-aware information states.} Dynamic risk preferences can potentially depend on the full problem history, leading to intractable models.
The information state framework in \cite{subramanian2022approximate} (for risk-neutral MDPs) can be adapted to compactly model history-dependent risk preferences in our setting.
The idea of an information state also naturally emerges in risk-aware MDPs. In \cite{bauerle2013more}, the information state is the cumulative wealth; in \cite{chow2015risk,li2017quantile}, it is the risk level.
We connect the information state concept with our new family of GCRs, and use it to compactly represent DM's dynamic risk preferences and their evolution.

\textbf{New family of GCR-MDPs.} We create a new class of risk-aware MDPs that we call {\it GCR-MDPs}, where the objective is to find a policy to minimize a GCR evaluated on the payoff stream.
We obtain the DP decomposition of GCR-MDPs which leads to a more general form of the classical DP recursion.
In GCR-MDPs, each cost-to-go problem appears in epigraph form where the objective is to minimize the current period premium subject to constraints on the current triple of income, disbursement, and premium.
The dependence on future periods is fully captured by the form of these constraints.

\textbf{Unifying framework.} We recast several well known risk-aware MDPs as GCR-MDPs (e.g., nested risk measures, expected utility, conditional value-at-risk, etc.).
The GCR-MDP framework highlights their essential common structure by demonstrating that they all follow from appropriate construction of the risk frontier and information state.
Our framework also streamlines the design of risk-aware MDPs, and we develop new GCR-MDPs based on wealth-dependent preferences and target shortfall.

\subsection{Previous Work}
The evaluation of risk starts with an axiomatic foundation. For instance, \cite{morgenstern1953theory} develop the axiomatic foundations for VNM expected utility theory based on: (i) independence; and (ii) diminishing marginal utility (concavity).
\cite{artzner1999coherent} provides an axiomatic foundation for coherent risk measures, based on: (i) monotonicity; (ii) translation invariance; (iii) convexity; and (iv) positive homogeneity.

\textbf{Dynamic risk measures.} In the dynamic setting, information is revealed sequentially which gives opportunities to update risk evaluations.
\cite{ruszczynski2006conditional} develops the theory of conditional risk mappings.
Dynamic risk measures have been developed for terminal payoffs and for various value processes \cite{artzner2007coherent,cheridito2006dynamic,jobert2008valuations}. \cite{bielecki2014dynamic} develops the class of dynamic acceptability indices which account for both the timing and the amount of the payments.
\cite{detlefsen2005conditional} develops dynamic risk measures for terminal payoffs.

The property of time consistency is a central issue in dynamic risk evaluation. A dynamic risk measure is time-inconsistent if, at some time, it says one prospect is less risky than another, but gives the opposite prediction based on the conditional distributions at a later time.
For example, re-applying the same static risk measure to the tail cumulative payoff in each period may give time-inconsistent risk evaluations.
The concept of time consistency has been formalized in several ways (see, e.g., \cite{artzner2007coherent,Shapiro:2009fk,carpentier2012dynamic}).
\cite{tutsch2008update} examines update rules for convex risk measures and studies time consistency with respect to test sets (e.g., the set of all fixed cash payments or the set of all financial instruments).
\cite{Shapiro:2009fk} develops a time consistency concept in terms of the overall optimal policy (determined in the initial period) for each of the tail problems.
\cite{bielecki2018unified} gives a unified study of many different notions of time consistency for both terminal payoffs and processes. This work is based on the idea of an update rule which recovers different versions of time consistency found in the literature.

The existing risk-aware MDP models can be roughly categorized as: time-consistent or time-inconsistent.
Time-consistent formulations emphasize a temporal decomposition, where the optimal policy is computed by DP.
Time-inconsistent formulations emphasize `pre-committed' policies, where it is not possible to deviate from the policy once it is engaged.
In \cite{pflug2016timea}, it is argued that time-inconsistent formulations are easier to specify. Further, they more naturally align with a decision maker's sense of the acceptability of the final outcome and enjoy greater interpretability.

\textbf{Nested risk measures.} \cite{ruszczynski2006conditional} develops a foundation for conditional risk functions, and give a corresponding DP decomposition.
\cite{ruszczynski2010risk} introduces the family of nested risk measures and defines a new formulation for risk-averse MDPs. In addition, \cite{ruszczynski2010risk} derives a DP decomposition and value iteration method for both finite and infinite horizon problems.
There is a close connection between nested risk measures and robust DP that is elaborated in \cite{osogami2011iterated}.
Follow up work develops several numerical methods for MDPs with Markov risk measures.
\cite{tamar2016sequential} considers simulation-based algorithms for risk-aware MDPs, they use an actor-critic type algorithm.
\cite{jiang2018risk} develops additional simulation-based numerical methods for approximate DP for MDPs with nested risk measures.

\textbf{Expected utility.} Expected utility models \cite{howard1972risk,kreps1977decisionb,bauerle2013more} seek to maximize the expected utility of terminal wealth and naturally depend on DM's wealth level.
Early work on risk aversion in MDPs focused on separable utility functions (see, e.g., \cite{chung1987discounted}) where there is a separate utility function for each period.
\cite{bauerle2013more} develops a DP decomposition for expected utility maximizing MDPs, where wealth becomes the augmenting state variable.
\cite{li2022time} solves the dynamic portfolio problem with expected utility objective, where the risk aversion is allowed to be regime dependent (e.g., whether we are in a bull or bear market).

\textbf{Mean-risk models.} The mean-variance criteria is a natural generalization of the risk-neutral expectation and is applied to MDPs in \cite{Filar1989,Sobel1994,delage2010percentile,mannor2013algorithmic,shi2017better}.
\cite{basak2010dynamic} derives a time-consistent solution for the dynamic mean-variance portfolio problem in continuous time.
\cite{bjork2014mean,bjork2014theory} consider a class of mean-variance portfolio problems with wealth-dependent risk aversion.
This problem is solved through a game-theoretic formulation to obtain a time-consistent policy.
\cite{kovavcova2021time} develop a set-valued Bellman principle for a dynamic mean-risk portfolio problem.
\cite{he2022mean} studies the mean-variance portfolio in continuous time with dynamic targets for the terminal wealth, and finds the optimal policy through an equilibrium strategy.

\textbf{Conditional value-at-risk.} Minimization of conditional value-at-risk (CVaR) of total costs/rewards is considered in \cite{bauerle2011markov,chow2015risk}. \cite{bauerle2011markov} is based on augmenting the state space with the cumulative wealth, and then solving a sequence of DP problems based on optimizing the threshold parameter in the variational representation of CVaR.
The notion of extended conditional risk functional is introduced in \cite{pflug2016time} which allows for a temporal decomposition of an initial (static) risk functional.
In particular, \cite{pflug2016time} derives a time-consistent temporal decomposition for conditional value-at-risk (CVaR) which depends on the random risk level.
\cite{chow2015risk} is based on this result from \cite{pflug2016time}, and gives a DP decomposition for CVaR where the augmenting state variable is the risk level.

\textbf{General risk measures.} \cite{pflug2016timea} considers the temporal decomposition of distortion risk measures and shows that the DP principle applies.
\cite{bauerle2020minimizing} studies minimization of a spectral risk measure of the total discounted cost. Similar to \cite{bauerle2011markov}, the solution method in \cite{bauerle2020minimizing} is based on optimizing over the risk spectra, and then solving a corresponding DP problem.
In \cite{haskell2015convex}, the convex analytic approach is developed to solve time-inconsistent MDPs with a general risk objective. In line with related works, \cite{haskell2015convex} is based on augmenting the state space with the cumulative wealth. Then, occupation measures are defined on the augmented state space which account for the distribution of terminal wealth.
\cite{tamar2016sequential} consider a policy gradient algorithm for general time-inconsistent static risk measures, based on the expression of the risk envelope.

\subsection{Organization}
This paper is organized as follows.
In Section~\ref{sec:preliminaries}, we review preliminaries on dynamic risk measures.
In Section~\ref{sec:gcr}, we present GCRs and then Section~\ref{sec:tempCR} investigates their temporal decomposition and the role of an information state.
In Section~\ref{sec:problem} we discuss GCR-MDPs, derive their DP decomposition, and connect it with the information state.
Section~\ref{sec:applications} casts several dynamic risk measures as GCRs.
We explore GCR-MDP models for the dynamic newsvendor in simulations in Section~\ref{sec:simulations}, and we conclude the paper in Section~\ref{sec:conclusion}. All proofs are in the Appendices.

\section{Preliminaries}
\label{sec:preliminaries}

\subsection{Notation}
For any integer $n \geq 0$, we let $[n] \triangleq \{0, 1, 2, \ldots, n\}$ denote the running index starting from zero.
For integers $n_1 \leq n_2$, we let $[n_1, n_2] \triangleq \{n_1, n_1 + 1, \ldots, n_2\}$ denote the running index from $n_1$ to $n_2$.
Let $\mathcal U$ be an arbitrary set in a vector space equipped with a partial order $\succeq$. Then $\mathcal U$ is monotone if $X \in \mathcal U$ and $Y \succeq X$ imply $Y \in \mathcal U$.
We take the optimal value of an infeasible minimization problem to be $+\infty$ and the optimal value of an infeasible maximization problem to be $-\infty$.


For a finite set $\mathbb D$, we let $\mathcal P(\mathbb D)$ denote the set of probability measures over $\mathbb D$.
For real-valued random variables $X$ and $Y$ defined on the same probability space $(\Omega, \mathscr{F}, \mathbb{P})$, $X \geq Y$ means $X(\omega) \geq Y(\omega)$ for $\mathbb P-$almost all $\omega \in \Omega$.
For two jointly distributed random variables $X$ and $Y$, $[X|y]$ denotes the conditional random variable $X$ given $Y = y$. 





\subsection{Payoff Streams}

We consider the risk of payoff streams over a finite planning horizon $[T] \triangleq \{0, 1, \ldots, T\}$ for $0 \leq T < \infty$.
We let $\Xi$ be the space of uncertain data, and we let $\{\xi_t\}_{t \in [T]}$ be the uncertain data process where $\xi_t \in \Xi$ is the realization in period $t \in [T]$.

We suppose the system has an evolving state which is defined on the state space $\mathbb{S}$, where $s_t \in \mathbb{S}$ is the state at the beginning of period $t$.
The system dynamic is given by $f_t : \mathbb{S} \times \Xi \rightarrow \mathbb{S}$ for $t \in [T-1]$, and state transitions follow $s_{t+1} = f_t(s_t,\xi_t)$ for all $t \in [T-1]$ (where the terminal state is $s_T$).

\begin{assumption}
\label{assu:preliminaries}
(i) $\Xi$ is finite.

(ii) $\mathbb{S}$ is finite.

(iii) The initial state $s_{0}\in\mathbb{S}$ is fixed.
\end{assumption}
\noindent
Under Assumption~\ref{assu:preliminaries}, we take the underlying sample space $\Omega = \mathbb S \times \Xi^{T+1}$ to be the set of all trajectories $\omega = (s_0, \xi_0, \ldots, \xi_{T})$ of initial state and data realizations.
Since we have a system dynamic $\{f_t\}_{t \in [T]}$, $(s_1, \ldots, s_T)$ is completely determined on $\omega$ given $s_0$.
Let $h_t = (s_0, \xi_0, \ldots, \xi_{t-1})$ be the history up to the beginning of period $t \in [T]$, and let $h_{T+1} = (s_0, \xi_0, \ldots, \xi_{T})$ be the full problem history (which is observable at the beginning of period $T+1$).
Then, let $\mathbb{H}_t$ be the set of all histories up to period $t \in [T+1]$.

Under Assumption~\ref{assu:preliminaries}, $\Omega$ is finite and we may let $\mathscr{F}$ be the standard Borel $\sigma$-algebra on $\Omega$.
We then let $\mathbb P$ be a probability distribution on $(\Omega, \mathscr{F})$ to obtain the probability space $(\Omega, \mathscr{F}, \mathbb{P})$.
Let $\mathscr{F}_t = \sigma(h_t)$ be the information available at the beginning of period $t \in [T]$, and let $\mathscr{F}_{T+1} = \sigma(h_{T+1}) = \mathscr{F}$ for the terminal period $T+1$.
We define $\mathcal L_t \triangleq \mathcal L^{\infty}(\Omega, \mathscr{F}_t, \mathbb P)$ for all $t \in [T+1]$ to be the space of essentially bounded $\mathscr{F}_t-$measurable random variables on $(\Omega, \mathscr{F}, \mathbb{P})$.
Then we define $\mathcal L_{1:T+1} \triangleq \prod_{t=1}^{T+1} \mathcal L_t$ to be the set of essentially bounded stochastic payoff streams.
Let $X_{0:T} = (X_t)_{t = 0}^T \in \mathcal L_{1:T+1}$ be a particular payoff stream where each $X_t \in \mathcal L_{t+1}$ is the payoff in period $t \in [T]$ (which depends on $\xi_t$ and is observed at the end of period $t$).

Since $X_{0:T}$ is dynamic, we face a sequence of risk evaluations of the tail payoff streams in each period after new information becomes available.
For any $t \in [1, T+1]$, we define $\mathcal L_{t+1:T+1} \triangleq \prod_{k=t+1}^{T+1} \mathcal L_t$ to be the set of essentially bounded tail payoff streams on $[t, T]$.
We then let $X_{t:T} = (X_t, \ldots, X_T) \in \mathcal L_{t+1:T+1}$ denote a tail payoff stream starting from period $t \in [T]$.
Dynamic risk measures formalize the idea of a sequence of risk evaluations of the tails of a payoff stream.

\begin{defn}
\label{defn:dynamic_risk}
A dynamic risk measure on $\mathcal L_{1:T+1}$ is a sequence of mappings $\{\varrho_{t:T}\}_{t \in [T]}$, where $\varrho_{t:T} : \mathcal L_{t+1:T+1} \rightarrow \mathcal L_t$ evaluates the risk of the tail payoff stream $X_{t:T}$ for all $t \in [T]$.
\end{defn}
\noindent
We note that $\varrho_{t:T}(X_{t:T})$ is an $\mathscr{F}_t-$measurable random variable, where $[\varrho_{t:T}(X_{t:T})](h_t)$ denotes the period $t$ risk evaluation of $X_{t:T}$ on history $h_t \in \mathbb{H}_t$.
Next we recall some key properties of dynamic risk measures.

\begin{defn}
A dynamic risk measure $\{\varrho_{t:T}\}_{t \in [T]}$ is:

(i) Monotone if $X_{t:T} \geq Y_{t:T}$ implies $\varrho_{t:T}(X_{t:T}) \leq \varrho_{t:T}(Y_{t:T})$ for all $X_{t:T},\, Y_{t:T} \in \mathcal L_{t+1:T+1}$ and $t \in [T]$.

(ii) Convex if $\varrho_{t:T}(\lambda\, X_{t:T} + (1 - \lambda)Y_{t:T}) \geq \lambda\,\varrho_{t:T}(X_{t:T}) + (1 - \lambda)\varrho_{t:T}(Y_{t:T})$ for all $X_{t:T},\, Y_{t:T} \in \mathcal L_{t+1:T+1}$, $\lambda \in [0, 1]$, and $t \in [T]$.

(iii) Positively homogeneous if $\varrho_{t:T}(\alpha_t X_{t:T}) = \alpha_t \varrho_{t:T}(X_{t:T})$ for all $X_{t:T} \in \mathcal L_{t+1:T+1}$ and $\alpha_t \in \mathcal L_t$ with $\alpha_t \geq 0$, and $t \in [T]$.
\end{defn}
\noindent
Several definitions of time consistency of $\{\varrho_{t:T}\}_{t \in [T]}$ have been proposed, the following definition from \cite{ruszczynski2010risk} is referred to as `strong time consistency'.
\begin{defn}
\label{defn:time_consistency}
A dynamic risk measure $\{\varrho_{t:T}\}_{t \in [T]}$ is time-consistent if for all $0 \leq t_1 < t_2 \leq T$ and all sequences $X_{t_1:T}, Y_{t_1:T} \in \mathcal L_{t_1:T}$, the conditions $X_k = Y_k$ for $k \in [t_1, t_2 - 1]$ and $\varrho_{t_2:T}(X_{t_2:T}) \leq \varrho_{t_2:T}(Y_{t_2:T})$ imply $\varrho_{t_1:T}(X_{t_1:T}) \leq \varrho_{t_1:T}(Y_{t_1:T})$.
\end{defn}
\noindent
In other words, if two processes $X_{t_1:T}$ and $Y_{t_1:T}$ are the same on periods $[t_1, t_2 - 1]$, and the tail risk evaluations for $X_{t_2:T}$ and $Y_{t_2:T}$ satisfy $\varrho_{t_2:T}(X_{t_2:T}) \leq \varrho_{t_2:T}(Y_{t_2:T})$, then $X_{t_1:T}$ is preferred to $Y_{t_1:T}$.

\section{Generalized Capital Requirements}\label{sec:gcr}

In this section, we review the existing CR setup and then present our generalization. 
To begin, let $\mathcal A \subset \mathcal L_{1:T+1}$ be an acceptance set for payoff streams.
The acceptance set may reflect DM's preferences or be fixed by an external regulator.

\begin{example}
(i) Let $\mathcal A = \{X_{0:T} : \sum_{t=0}^T X_t \geq 0\}$ be the set of all payoff streams with non-negative cumulative payoff.

(ii) Let $u : \Re \rightarrow \Re$ be an increasing utility function, and let $\mathcal A = \{X_{0:T} : \mathbb E[u(\sum_{t=0}^T X_t)] \geq 0\}$ be the set of all payoff streams with non-negative expected utility of cumulative payoff.

(iii) For each period $t \in [T]$, let $\mathcal A_t \subset \mathcal L_{t+1}$ be an acceptance set. Then, the product acceptance set is $\mathcal A_{\times} = \mathcal A_0 \times \cdots \times \mathcal A_T$.
\end{example}

Next we introduce the disbursement strategy $Z_{0:T} = (Z_t)_{t=0}^T \in \mathcal L_{0:T}$, where $Z_t \in \mathcal L_{t}$ is the period $t \in [T]$ disbursement.
We let $z_t \in \Re$ denote the realizations of $Z_t$.
The disbursements are financial actions which augment the payoff, e.g., they can represent withdrawal from a savings account, sale of assets, or exercise of an option.
Positive $Z_t$ is considered as a disbursement to supplement the current payoff, negative $Z_t$ is considered as a deposit to savings or a loan repayment, so the total income is $X_t + Z_t$.
We refer to $Z_{0:T}$ generally as `disbursements' but this process takes on other meanings (e.g., consumption, target shortfall, etc.) in our applications.
We emphasize that $Z_t \in \mathcal L_{t}$ while $X_t \in \mathcal L_{t+1}$ for all $t \in [T]$, so disbursements are pre-committed in each period before observing the current payoff.

We suppose DM has no access to an underlying financial market. Rather, DM's financial position is captured by a risk-free savings account that we refer to as the wealth.
In this setting, DM first determines an initial endowment $w_0 \in \Re$ to be put into the risk-free savings account in period $t = 0$.
Then, he determines the timing and the amounts of the disbursements over the planning horizon.
The wealth process is $W_{0:T+1} \triangleq (W_t)_{t=0}^{T+1}$ where $W_t \in \mathcal L_t$ is the wealth at the beginning of period $t \in [T+1]$. 
Wealth evolves according to:
\begin{equation*}
W_{t} = W_{t-1} - Z_{t-1} = w_0 - \sum_{i=0}^{t-1} Z_i,\, \forall t \in [T].
\end{equation*}
In particular, the terminal wealth is $W_{T+1} = w_0 - \sum_{t=0}^T Z_t$.
Wealth can become negative, e.g., when the firm takes out a loan to supplement its payoff. We let $\mathcal{C}_{w} \triangleq \{Z_{0:T} : W_{T+1} \geq 0\}$ be the set of disbursement strategies where all loans are repaid by the end of the planning horizon.

The standard CR finds the minimal initial endowment required for the total income to reach $\mathcal A$, subject to the requirement that all loans are repaid by the end of the planning horizon.

\begin{defn}
\label{defn:CR_standard}
\cite[Definition 4.3]{frittelli2006risk}
The standard CR associated to $\mathcal A$ is the map $\varrho_{\mathcal A} : \mathcal L_{1:T+1} \rightarrow \Re$ defined by:
\[
\varrho_{\mathcal A}(X_{0:T}) \triangleq \min_{w_0 \in \Re,\, Z_{0:T}} \left\{w_0 : X_{0:T} + Z_{0:T} \in \mathcal A,\, Z_{0:T} \in \mathcal{C}_{w} \right\},\, \forall X_{0:T} \in \mathcal L_{1:T+1}.
\]
\end{defn}

There may be additional constraints on feasible disbursement strategies, represented by a set $\mathcal C \subset \mathcal L_{0:T}$.

\begin{example}
(i) $\mathcal{C}_{\geq 0} \triangleq \{Z_{0:T} \geq 0\}$ requires non-negative disbursements.

(ii) $\mathcal C_{\beta} \triangleq \{W_t \geq - \beta,\, \forall t \in [T]\}$ imposes a credit limit of $\beta \geq 0$ in all periods.
\end{example}

Next, we let $\phi : \mathcal C \rightarrow \Re$ be a cost function for the disbursement strategy, where $\phi(Z_{0:T})$ is the period $t = 0$ cost for DM to access $Z_{0:T}$.

\begin{example}
    (i) Expected total disbursement $\phi(Z_{0:T}) = \mathbb E[\sum_{t=0}^T Z_t]$.

    (ii) Risk of total disbursement $\phi(Z_{0:T}) = \rho(\sum_{t=0}^T Z_t)$.

    (iii) Worst-case disbursement $\phi(Z_{0:T}) = \max\{Z_0, \ldots, Z_T\}$.

    (iv) Essential supremum of total disbursement $\phi(Z_{0:T}) = {\rm ess} \sup(\sum_{t=0}^T Z_t)$ (which recovers the standard CR).
\end{example}

The ingredients $(\mathcal A, \mathcal C, \phi)$ together give a CR.

\begin{defn}
\label{defn:CR_deluxe}
The CR associated to $(\mathcal A, \mathcal C, \phi)$ is the map $\varrho_{\mathcal A, \mathcal C, \phi} : \mathcal L_{1:T+1} \rightarrow \Re$ defined by:
\[
\varrho_{\mathcal A, \mathcal C, \phi}(X_{0:T}) \triangleq \min_{Z_{0:T}}\left\{\phi(Z_{0:T}) : X_{0:T} + Z_{0:T} \in \mathcal A,\, Z_{0:T} \in \mathcal C \right\},\, \forall X_{0:T} \in \mathcal L_{1:T+1}.
\]
\end{defn}

We generalize $\varrho_{\mathcal A, \mathcal C, \phi}$ by unifying the ingredients $(\mathcal A, \mathcal C, \phi)$ into a single {\it risk frontier}.
We introduce a {\it premium schedule} $\Phi_{0:T} = (\Phi_t)_{t=0}^T \in \mathcal L_{0:T}$, where $\Phi_t \in \mathcal L_t$ is the period $t \in [T]$ premium. For each $t \in [T]$, $\Phi_t(h_t)$ is the premium for the tail sequence $(X_{t:T}, Z_{t:T}, \Phi_{t:T})$ on the remaining periods $[t, T]$ conditional on history $h_t$.
We let $\varphi_t \in \Re$ denote the realizations of $\Phi_t$.
We motivate each $\Phi_t$ as an epigraphical variable for the corresponding tail problem, i.e., $\min_{x \in \mathbb X} f(x) \equiv \min_{x \in \mathbb X, s \in \Re}\{s : s \geq f(x)\}$.
The initial $\Phi_0$ is the premium for the entire triple $(X_{0:T}, Z_{0:T}, \Phi_{0:T})$.
It is $\mathscr{F}_0-$measurable and so is a constant, we just denote it as $\varphi_0$ when it appears alone.
For later periods $t \in [1, T]$, $\Phi_t$ is $\mathscr{F}_t-$measurable and may depend on the information realized up to the beginning of period $t$.

In the GCR framework we evaluate the risk of $X_{0:T}$ in terms of the entire triple $(X_{0:T}, Z_{0:T}, \Phi_{0:T}) \in \mathcal L_{1:T+1} \times \mathcal L_{0:T} \times \mathcal L_{0:T}$, and next we formally define the set of acceptable triples.
This definition includes some intuitively reasonable monotonicity properties to describe the acceptability of a triple after a change in one of its components.

\begin{defn}[Risk Frontier]
\label{defn:risk_frontier}
A set $\mathcal U \subset \mathcal L_{1:T+1} \times \mathcal L_{0:T} \times \mathcal L_{0:T}$ is a {\em risk frontier} if:

(i) $(X_{0:T}, Z_{0:T}, \Phi_{0:T}) \in \mathcal U$ implies $(X_{0:T}', Z_{0:T}, \Phi_{0:T}) \in \mathcal U$ for all $X_{0:T}' \geq X_{0:T}$.

(ii) $(X_{0:T}, Z_{0:T}, \Phi_{0:T}) \in \mathcal U$ implies $(X_{0:T}, Z_{0:T}, \Phi_{0:T} + \Phi_{0:T}') \in \mathcal U$ for all non-increasing processes $\Phi_{0:T}' \geq 0$ (i.e., $\Phi_t'(h_t) \geq \mathbb E[\Phi_{t+1}'|h_t]$ for all $h_t \in \mathbb{H}_t$ and $t \in [T]$).

(iii) $(X_{0:T}, Z_{0:T}, \Phi_{0:T}) \in \mathcal U$ and $Z_{0:T}' \geq Z_{0:T}$ implies $(X_{0:T}, Z_{0:T}', \Phi_{0:T}') \in \mathcal U$ for some $\Phi_{0:T}' \geq \Phi_{0:T}$.

(iv) $\mathcal U$ is closed.

\noindent
We say a triple is acceptable when $(X_{0:T}, Z_{0:T}, \Phi_{0:T}) \in \mathcal U$.
\end{defn}
\noindent
Definition~\ref{defn:risk_frontier} and $\mathcal U$ completely characterize DM's risk preferences in the GCR framework.
Definition~\ref{defn:risk_frontier}(i) addresses monotonicity with respect to payoffs. If a triple is acceptable, then the same triple with a higher payoff at all times is also acceptable (for the same disbursements and premiums). 
Definition~\ref{defn:risk_frontier}(ii) addresses monotonicity with respect to premiums. If a triple is acceptable, then the same triple with higher premiums (that remain non-increasing in time) is also acceptable (to the financial institution).
Definition~\ref{defn:risk_frontier}(iii) addresses monotonicity with respect to disbursements. If a triple is acceptable, then the same triple with a higher disbursement is acceptable for a possibly higher premium (i.e., greater disbursement incurs greater costs).
Finally, Definition~\ref{defn:risk_frontier}(iv) states that if a sequence of triples is acceptable, then its limit should also be acceptable.

Next we define the risk evaluation given a risk frontier.
\begin{defn}[Generalized CR]
\label{defn:GCR}
Let $\mathcal U$ be a risk frontier.
The generalized CR (GCR) associated to $\mathcal U$ is the map $\varrho_{\mathcal U} : \mathcal L_{1:T+1} \rightarrow \Re$ defined by:
\begin{equation}
\label{eq:GCR_minimization}
\varrho_{\mathcal U}(X_{0:T}) \triangleq \min_{Z_{0:T}, \Phi_{0:T}} \{ \varphi_0 : (X_{0:T}, Z_{0:T}, \Phi_{0:T}) \in \mathcal U\},\, \forall X_{0:T} \in \mathcal L_{1:T+1}.
\end{equation}
\end{defn}
\noindent
In Eq.~\eqref{eq:GCR_minimization}, the objective is the period $t = 0$ premium $\varphi_0$.
Given a payoff stream $X_{0:T}$, $\varrho_{\mathcal U}(X_{0:T})$ optimizes $(Z_{0:T}, \Phi_{0:T})$ over acceptable triples $(X_{0:T}, Z_{0:T}, \Phi_{0:T}) \in \mathcal U$ to minimize $\varphi_0$.

\begin{rem}
We canonically define a GCR as a minimization problem, but we can similarly define it as a maximization problem:
\begin{equation}
\label{eq:GCR_maximization}
\varrho_{\mathcal U}(X_{0:T}) \triangleq \max_{Z_{0:T}, \Phi_{0:T}} \{ \varphi_0 : (X_{0:T}, Z_{0:T}, \Phi_{0:T}) \in \mathcal U\},\, \forall X_{0:T} \in \mathcal L_{1:T+1}.
\end{equation}
For $\varrho_{\mathcal U}$ in maximization form Eq.~\eqref{eq:GCR_maximization}, we re-interpret $\Phi_{0:T} = (\Phi_0, \ldots, \Phi_T) \in \mathcal L_{0:T}$ as an acceptability schedule, where $\Phi_t \in \mathcal L_t$ corresponds to the degree of acceptability of the tail payoff stream $X_{t:T}$ for all $t \in [T]$.
\end{rem}

The next assumption gives conditions for $\varrho_{\mathcal U}(X_{0:T})$ to be well-defined (where the optimal value of the inner minimization is attained).
The set $\Omega$ is finite under Assumption~\ref{assu:preliminaries}, so we can identify stochastic processes on $\Omega$ with vectors in a Euclidean space.
For fixed $X_{0:T} \in \mathcal L_{1:T+1}$ and a level $\tau \in \Re$, define (the possibly empty set)
$$
\mathcal{U}_{\tau}(X_{0:T}) \triangleq \{ (Z_{0:T}, \Phi_{0:T}) : \varphi_0 \leq \tau,\, (X_{0:T}, Z_{0:T}, \Phi_{0:T}) \in \mathcal U \}.
$$

\begin{assumption}
\label{assu:compact}
For any $X_{0:T} \in \mathcal L_{1:T+1}$, $\mathcal{U}_{\tau}(X_{0:T})$ is compact for all $\tau \in \mathbb R$.
\end{assumption}
\noindent
Assumption~\ref{assu:compact} acts as a growth condition on $\mathcal U$.
For a fixed level $\tau$, there is no disbursement $Z_t$ that we can increase arbitrarily while keeping $\varphi_0 \leq \tau$.
Under this assumption, we can establish compactness and apply the Weierstrass theorem to Eq.~\eqref{eq:GCR_minimization} to show that the minimum in $\varrho_{\mathcal U}(X_{0:T})$ is attained.

\begin{thm}
\label{thm:well-defined}
Suppose Assumption~\ref{assu:compact} holds, let $X_{0:T} \in \mathcal L_{1:T+1}$, and suppose $\mathcal{U}_{\tau}(X_{0:T})$ is non-empty for some $\tau \in \Re$. Then, $\varrho_{\mathcal U}(X_{0:T})$ is finite and the minimum is attained.
\end{thm}

Now we establish some properties of the GCR as a function of the properties of $\mathcal U$, this result extends \cite[Proposition 4.6]{frittelli2006risk} to our setting.
We note that we only need monotonicity of the GCR for our upcoming temporal decompositions.

\begin{thm}
\label{thm:properties}
Let $\mathcal U$ be a risk frontier.

(i) $\varrho_{\mathcal U}$ is monotone.

(ii) Suppose $\mathcal U$ is convex, then $\varrho_{\mathcal U}$ is convex.

(iii) Suppose $\mathcal U$ is a cone, then $\varrho_{\mathcal U}$ is positively homogeneous.
\end{thm}
\noindent
Monotonicity is expected for any sound dynamic risk measure from a decision-theoretic sense, Theorem~\ref{thm:properties}(i) shows that this property is automatic for any GCR by Definition~\ref{defn:risk_frontier}.

We conclude this section with some examples of GCRs and their risk frontiers.

\begin{example}
Let $\mathcal A \subset \mathcal L_{T+1}$ be an acceptance set for the cumulative payoff and $\varrho(X_{0:T}) = \min_{z \in \Re} \{ z : \sum_{t=0}^T X_t + z \in \mathcal A \}$.
Then $\varrho$ is recovered by $\mathcal U = \{ (X_{0:T}, Z_{0:T}, \Phi_{0:T}) : \varphi_0 \geq z_0,\, z_0 + \sum_{t=0}^T X_t \in \mathcal A\}$, where the disbursements $Z_{1:T}$ do not appear.
\end{example}

\begin{example}
The standard CR $\varrho_{\mathcal A}$ is recovered by $\mathcal U = \{(X_{0:T}, Z_{0:T}, \Phi_{0:T}) : X_{0:T} + Z_{0:T} \in \mathcal A,\, Z_{0:T} \in \mathcal C_{w},\, \varphi_0 \geq w_0, w_0 \in \Re\}$.
When $\mathcal A$ is convex, the inner optimization problem in $\varrho_{\mathcal A}(X_{0:T})$ is convex for any $X_{0:T}$ by linearity of the wealth update.
\end{example}

\begin{example}
The CR $\varrho_{\mathcal A, \mathcal C, \phi}$ is recovered by $\mathcal U = \{(X_{0:T}, Z_{0:T}, \Phi_{0:T}) : X_{0:T} + Z_{0:T} \in \mathcal A,\, Z_{0:T} \in \mathcal C,\, \varphi_0 \geq \phi(Z_{0:T})\}$.
In particular, $\varphi_0$ corresponds to $\phi(Z_{0:T})$ (the time $t = 0$ total cost).
When $\mathcal A$ and $\mathcal C$ are convex sets, and $\phi$ is a convex function, then $\mathcal U$ is convex and the inner optimization problem in $\varrho_{\mathcal A, \mathcal C, \phi}(X_{0:T})$ is convex for any $X_{0:T}$.
\end{example}

\begin{example}
Let $\nu : \mathcal L_{1:T+1} \rightarrow \mathbb R$ be a convex functional and let $\varrho_{\nu, \mathcal C, \phi} : \mathcal L_{1:T+1} \rightarrow \Re$ be the risk measure defined as:
\[
\varrho_{\nu, \mathcal C, \phi}(X_{0:T}) \triangleq \min_{Z_{0:T}} \{\phi(Z_{0:T}) + \nu(X_{0:T} + Z_{0:T}) : Z_{0:T} \in \mathcal C \},\, \forall X_{0:T} \in \mathcal L_{1:T+1}.
\]
Then $\varrho_{\nu, \mathcal C, \phi}$ is recovered by $\mathcal U = \{(X_{0:T}, Z_{0:T}, \Phi_{0:T}) : Z_{0:T} \in \mathcal C,\, \varphi_0 \geq \phi(Z_{0:T}) + \nu(X_{0:T} + Z_{0:T})\}$.
Notice that $\mathcal U$ is convex if $\mathcal C$ is a convex set, and $\phi$ and $\nu$ are convex functions.
\end{example}

\section{Temporal Decomposition of Risk Frontier}\label{sec:tempCR}

In the dynamic setting, we need to update our risk evaluations over time as new information is revealed.
In this section we discuss the temporal decomposition of the GCR.
In each period $t \in [T]$ after observing $h_t \in \mathbb{H}_t$, we want to evaluate the risk of the tail payoff stream $X_{t:T}$ conditional on $h_t$.
This is done in the GCR framework by optimizing over a tail disbursement strategy $Z_{t:T} \in \mathcal L_{t:T}$ and a tail premium schedule $\Phi_{t:T} \in \mathcal L_{t:T}$ (both of which are also conditional on $h_t$).

To continue, we make the following assumption on DM's available information.
\begin{assumption}
\label{assu:recall}
DM has perfect recall of $(X_{0:t-1}, Z_{0:t-1}, \Phi_{0:t-1})$ given $h_t$ for all $t \in [1, T+1]$.
\end{assumption}
\noindent
Assumption~\ref{assu:recall} states that knowledge of $h_t$ implicitly determines the historical payoffs $[X_{0:t-1} \vert h_t]$, disbursements $[Z_{0:t-1} \vert h_t]$, and premiums $[\Phi_{0:t-1} \vert h_t]$.

We need to disassemble the risk frontier to determine what membership in $\mathcal U$ implies with respect to $(X_{t:T}, Z_{t:T}, \Phi_{t:T})$.
The following definition gives a general notion of a temporal decomposition of $\mathcal U$ where the requirements on $(X_{t:T}, Z_{t:T}, \Phi_{t:T})$ only depend on what has happened so far on $h_t$.

\begin{defn}
\label{defn:U_decomposition}
We say that $\mathcal U$ has a temporal decomposition when there exist correspondences $\mathcal U_{t:T} : \mathbb{H}_t \rightrightarrows \mathcal L_{t+1:T+1} \times \mathcal L_{t:T} \times \mathcal L_{t:T}$ for all $t \in [T]$, such that:
\begin{align*}
    \mathcal U = \{(X_{0:T}, Z_{0:T}, \Phi_{0:T}) \text{ : } & [(X_{t:T}, Z_{t:T}, \Phi_{t:T}) \vert h_t] \in \mathcal U_{t:T}(h_t), \forall h_t \in \mathbb{H}_t, \forall t \in [T]\}.
\end{align*}
\end{defn}
\noindent
We refer to the correspondences $\{\mathcal U_{t:T}\}_{t \in [T]}$ in Definition~\ref{defn:U_decomposition} as the {\em tail risk frontiers}.
Under Definition~\ref{defn:U_decomposition}, we can define a tail risk measure $\varrho_{\mathcal U, t:T} : \mathcal L_{t+1:T+1} \rightarrow \mathcal L_t$ for all $t \in [T]$ by:
\[
    [\varrho_{\mathcal U, t:T}(X_{t:T})](h_t) \triangleq \min_{Z_{t:T}, \Phi_{t:T}} \{ \Phi_t(h_t) : [(X_{t:T}, Z_{t:T}, \Phi_{t:T}) \vert h_t] \in \mathcal U_{t:T}(h_t) \},\, \forall h_t \in \mathbb{H}_t,
\]
for all $X_{t:T} \in \mathcal L_{t:T}$.
Definition~\ref{defn:U_decomposition} then induces the dynamic risk measure $\{\varrho_{\mathcal U, t:T}\}_{t \in [T]}$ as a function of $\mathcal U$.
In each period $t \in [T]$ on history $h_t \in \mathbb{H}_t$, we evaluate the risk of $X_{t:T}$ by optimizing over feasible $(Z_{t:T}, \Phi_{t:T})$ for the inclusion $[(X_{t:T}, Z_{t:T}, \Phi_{t:T}) \vert h_t] \in \mathcal U_{t:T}(h_t)$ with respect to the tail risk frontier.
We see that $\varrho_{\mathcal U, t:T}$ matches the structure of the original $\varrho_{\mathcal U}$ and is in epigraphical form, where all the particulars of DM's risk preferences for $X_{t:T}$ are contained in $\mathcal U_{t:T}(h_t)$.

Definition~\ref{defn:U_decomposition} yields the dynamic risk measure $\{\varrho_{\mathcal U, t:T}\}_{t \in [T]}$, but it is not automatically time-consistent.
Time consistency of $\{\varrho_{\mathcal U, t:T}\}_{t \in [T]}$ requires additional structure on $\mathcal U$.
Next we give a sufficient condition for time consistency of $\{\varrho_{\mathcal U, t:T}\}_{t \in [T]}$ expressed in terms of membership in the tail risk frontiers $\{\mathcal U_{t:T}\}_{t \in [T]}$.

\begin{assumption}
\label{assu:risk_frontier_consistency}
Choose $0 \leq t_1 < t_2 \leq T$, $[(X_{t_1:T}, Z_{t_1:T}, \Phi_{t_1:T}) \vert h_{t_1}] \in \mathcal U_{t_1:T}(h_{t_1})$ for all $h_{t_1} \in \mathbb{H}_{t_1}$, and $[(X_{t_2:T}', Z_{t_2:T}', \Phi_{t_2:T}') \vert h_{t_2}] \in \mathcal U_{t_2:T}(h_{t_2})$ with $\Phi_{t_2}'(h_{t_2}) \leq \Phi_{t_2}(h_{t_2})$ for all $h_{t_2} \in \mathbb{H}_{t_2}$.
Define $(\tilde X_{t_1:T}, \tilde Z_{t_1:T}, \tilde \Phi_{t_1:T})$ as 
\begin{align*}
    (\tilde X_{k}, \tilde Z_{k}, \tilde \Phi_{k}) &= (X_{k}, Z_{k}, \Phi_{k}),\, \forall k \in [t_1, t_2 - 1],\\
    (\tilde X_{k}, \tilde Z_{k}, \tilde \Phi_{k}) &= (X_{k}', Z_{k}', \Phi_{k}'),\, \forall k \in [t_2, T], 
\end{align*}
then $[(\tilde X_{t_1:T}, \tilde Z_{t_1:T}, \tilde \Phi_{t_1:T}) \vert h_{t_1}] \in \mathcal U_{t_1:T}(h_{t_1})$ for all $h_{t_1} \in \mathbb{H}_{t_1}$.
\end{assumption}
\noindent
Assumption~\ref{assu:risk_frontier_consistency} states that when we replace a tail sequence with one with a lower premium, then the overall risk evaluation is the same or lower.

\begin{thm}
\label{thm:time_consistent}
Suppose Assumption~\ref{assu:risk_frontier_consistency} holds, then $\{\varrho_{\mathcal U, t:T}\}_{t \in [T]}$ is time-consistent.
\end{thm}

\subsection{Recursive Formulation}

Under Definition~\ref{defn:U_decomposition}, we need to construct the entire tail sequence $(X_{t:T}, Z_{t:T}, \Phi_{t:T})$ for the period $t \in [T]$ risk evaluation.
The following assumption is a special case of Definition~\ref{defn:U_decomposition} based on recursive structure of the risk frontier.
Under this structure, the period $t+1$ premium $\Phi_{t+1}$ is a sufficient statistic for the tail sequence $(X_{t+1:T}, Z_{t+1:T}, \Phi_{t+1:T})$.
We take $\Phi_{T+1} \in \mathcal L_{T+1}$ as an input that reflects any terminal conditions.

\begin{assumption}
\label{assu:one-step}
There exist correspondences $\mathcal U_t : \mathbb{H}_t \times \mathcal L_{t+1} \rightrightarrows \mathcal L_{t+1} \times \Re \times \Re$ for all $t \in [T]$ such that:

(i) The risk frontier has the decomposition
\begin{equation}
\label{eq:risk_frontier_recursive}
\mathcal U = \{(X_{0:T}, Z_{0:T}, \Phi_{0:T}) 
 \text{ : } [(X_{t}, Z_{t}, \Phi_{t}) \vert h_t]  \in \mathcal U_{t}(h_t, \Phi_{t+1}),\, \forall h_t \in \mathbb{H}_t,\, \forall t \in [T]\}.
\end{equation}

(ii) For any $t \in [T]$, $h_t \in \mathbb{H}_t$, $\Phi_{t+1}\in\mathcal L_{t+1}$, if $(X_t, z_t, \varphi_t) \in \mathcal U_t(h_t, \Phi_{t+1})$ and $X_t \leq X_t'$ then $(X_t', z_t, \varphi_t) \in \mathcal U_t(h_t, \Phi_{t+1})$.

(iii) For any $\Phi_{t+1} \leq \Phi_{t+1}'$, we have $\mathcal U_t(h_t, \Phi_{t+1}') \subset \mathcal U_t(h_t, \Phi_{t+1})$ for all $h_t \in \mathbb{H}_t$.

\end{assumption}
\noindent
Assumption~\ref{assu:one-step} parallels the classical DP principle, where in period $t$, dependence on the future only enters through the period $t+1$ value function (which gives the remaining costs on periods $[t+1,T]$).
The dependence across time here is captured by the presence of both $(\Phi_t(h_t), \Phi_{t+1})$ in each inclusion $[(X_{t}, Z_{t}, \Phi_{t}) \vert h_t] \in \mathcal U_{t}(h_t, \Phi_{t+1})$.
In the classical case, $\Phi_t(h_t)$ corresponds to the current expected cost-to-go on $h_t$, and $\Phi_{t+1}$ is the next period cost-to-go function.
Unlike the classical DP principle based on expected cost-to-go, we account for the entire distribution of $\Phi_{t+1}$ through the dependence of $\mathcal U_t(h_t, \Phi_{t+1})$ on $\Phi_{t+1}$.

In the next theorem, we show that the dynamic risk measure $\{\varrho_{\mathcal U, t:T}\}_{t \in [T]}$ has a recursive form under Assumption~\ref{assu:one-step}, where we initialize with $\varrho_{\mathcal U, T+1:T} \equiv \Phi_{T+1}$. Additionally, we demonstrate that this recursive structure leads to time consistency of $\{\varrho_{\mathcal U, t:T}\}_{t \in [T]}$.

\begin{thm}
\label{thm:recursive}
Suppose Assumption~\ref{assu:one-step} holds.

(i) For any $t \in [T]$, $X_{t:T} \in \mathcal L_{t+1:T+1}$, we have:
\begin{equation}
\label{eq:recursive}
[\varrho_{\mathcal U, t:T}(X_{t:T})](h_t) = \min_{z_t, \varphi_t} \{\varphi_t : ([X_t \vert h_t], z_t, \varphi_t) \in \mathcal U_t(h_t, \varrho_{\mathcal U, t+1 : T}(X_{t+1:T}))\},\, \forall h_t \in \mathbb{H}_t.
\end{equation}

(ii) $\{\varrho_{\mathcal U, t:T}\}_{t \in [T]}$ is time-consistent.
\end{thm}
\noindent
This result is related to the construction of nested risk measures from a composition of one-step risk measures (see, e.g., \cite{cheridito2006dynamic}).
It shows that the tail risk evaluation can be decomposed into a function of the current period payoff and tail payoffs.
However, the composition above allows general dependence on the next period risk evaluation through $\mathcal U_t(h_t, \cdot)$ (which depends on the entire distribution of $\varrho_{\mathcal U, t+1 : T}$).

\begin{example}
For the temporal decomposition of $\varrho_{\mathcal A, \mathcal C, \phi}$: let $\mathcal A_t:\mathbb{H}_t\rightrightarrows \mathcal L_{t+1}$ be correspondences for all $t \in [T]$ such that $\mathcal A = \{X_{0:T} : [X_t \vert h_t] \in \mathcal A_{t}(h_t),\, \forall t \in [T]\}$; and let $\mathcal C_t:\mathbb{H}_t\rightrightarrows \mathcal L_{t}$ be correspondences for all $t \in [T]$ such that $\mathcal C = \{Z_{0:T} : [Z_{t} \vert h_t] \in \mathcal C_{t}(h_t),\, \forall t \in [T]\}$.
Additionally, suppose there exist $\phi_t:\Re\to\Re$ for all $t \in [T]$ such that $\phi(Z_{0:T}) = \mathbb E [ \sum_{t=0}^T \phi_t(Z_t)]$ (the premium is additively separable).
We have $\Phi_{T+1} \equiv 0$, so the risk frontier for $\varrho_{\mathcal A, \mathcal C, \phi}$ is determined by the constraints: $[X_t + Z_t \vert h_t] \in \mathcal A_t(h_t)$; $[Z_t \vert h_t] \in \mathcal C_t(h_t)$; and $\Phi_{t}(h_t) \geq \mathbb E[\phi_t(Z_t) + \Phi_{t+1} \vert h_t]$ for all $h_t \in \mathbb{H}_t$ and $t \in [T]$.
The resulting $\mathcal U_{t}(h_t, \Phi_{t+1})$ are then determined by the constraints: $[X_t \vert h_t] + z_t \in \mathcal A_t(h_t)$, $z_t \in \mathcal C_t(h_t)$, and $\varphi_t \geq \mathbb E[\phi_t(z_t) + \Phi_{t+1} \vert h_t]$. The tail risk frontiers are automatically in the recursive form of Assumption~\ref{assu:one-step}.
\end{example}

\subsection{Information State}

If the payoff stream $X_{0:T}$ has general history dependence, even the risk-neutral DP recursion is generally intractable.
There is also history dependence inherent in $\mathcal U$ due to DM's risk preferences, so for any $(X_{0:T}, Z_{0:T}, \Phi_{0:T}) \in \mathcal U$ we expect the associated $(Z_{0:T}, \Phi_{0:T})$ to have history dependence as well.

We need a compact representation of the problem history to be used in augmented DP.
We start by making the following assumptions on the data process $\{\xi_t\}_{t \in [T]}$ and the payoff streams to simplify their history dependence.
\begin{assumption}
\label{assu:markov}
(i) $\{\xi_t\}_{t \in [T]}$ are independent, and $\xi_t$ is distributed according to $P_t \in \mathcal P(\Xi)$ for all $t \in [T]$.

(ii) $X_t$ only depends on $s_t$ for all $t \in [T]$.
\end{assumption}
\noindent
Under Assumption~\ref{assu:markov}(i), $\{s_t\}_{t \in [T]}$ is a Markov chain on $\mathbb{S}$.
We then let $\mathbb P = P_0 \times P_1 \times \cdots \times P_T$ to obtain the probability space $(\Omega,\mathscr{F},\mathbb{P})$.
Under Assumption~\ref{assu:markov}(ii), the payoff distribution in period $t \in [T]$ is $[X_t \vert s_t]$.

Even when Assumption~\ref{assu:markov} prevails and $X_{0:T}$ is Markov, the associated $(Z_{0:T}, \Phi_{0:T})$ for some $(X_{0:T}, Z_{0:T}, \Phi_{0:T}) \in \mathcal U$ may still have complex history dependence.
However, in many dynamic risk models, there is some essential information that succinctly captures the history dependence of DM's risk preferences.
For example, preferences could depend only on the cumulative wealth, cumulative target shortfall, or current risk budget (rather than the entire problem history).
In such cases there is a compact sufficient statistic for the risk preferences.
This is analogous to the motivation of the ``information state'' framework developed in \cite{subramanian2022approximate} for risk-neutral DP.

We now extend the information state framework of \cite{subramanian2022approximate} to GCRs.
Let $\{\mathbb Y_t\}_{t \in [T+1]}$ be a pre-specified collection of Banach spaces, where $\mathbb Y_t$ contains the essential information about DM's risk preferences up to the beginning of period $t \in [T+1]$ (and $\mathbb Y_{T+1}$ corresponds to any terminal conditions).
Under Assumption~\ref{assu:markov}, we absorb $\mathbb{S} \subset \mathbb Y_t$ for all $t \in [T+1]$ to capture the dependence of the Markov payoff stream on the state.

We summarize the problem history using a collection of history compression functions $\{\sigma_t\}_{t \in [T+1]}$ where $\sigma_t : \mathbb{H}_t \rightarrow \mathbb Y_t$ for all $t \in [T+1]$. For each period $t \in [T+1]$, $\sigma_t$ summarizes the relevant historical information as a sufficient statistic on $\mathbb Y_t$.
We recall that $h_t$ implicitly determines the history of payoffs $[X_{0:t-1} \vert h_t]$ and disbursements $[Z_{0:t-1} \vert h_t]$ under Assumption~\ref{assu:recall}, which are inputs into the history compression functions.
We let $\mathcal G_t \triangleq \sigma_t(\mathcal F_t)$ be the corresponding $\sigma-$field on $\mathbb Y_t$ for all $t \in [T+1]$, and we let $\bar{\mathcal L}_t = \mathcal L^{\infty}(\Omega, \mathcal G_t, \mathbb P)$ be the set of real-valued $\mathcal G_t-$measurable functions for all $t \in [T+1]$. Since $\mathcal G_t \triangleq \sigma_t(\mathcal F_t)$ by construction, we have $\bar{\mathcal L}_t \subset \mathcal L_t$.
We let $Y_t \in \bar{\mathcal L}_t$ be the information state in period $t \in [T+1]$.
We then have $X_t(h_t) = X_t(\sigma_t(h_t))$ for all $h_t \in \mathbb{H}_t$ for any $X_t \in \bar{\mathcal{L}}_t$.

To continue, we introduce modified disbursement strategies $\bar Z_{0:T} \in \bar{\mathcal L}_{0:T}$ and premium schedules $\bar{\Phi}_{0:T} \in \bar{\mathcal L}_{0:T}$ just defined on the information state.
We also take $\bar{\Phi}_{T+1} \in \bar{\mathcal L}_{T+1}$ as an input corresponding to the terminal period.
Given a Markov payoff stream $X_{0:T}$, we want to evaluate the risk of $X_{0:T}$ just using $(\bar Z_{0:T}, \bar{\Phi}_{0:T}) \in \bar{\mathcal L}_{0:T} \times \bar{\mathcal L}_{0:T}$ to get a more tractable GCR.
The next definition formalizes the key requirements of an information state for a GCR.

\begin{defn}
\label{defn:information}
(Information state generator for GCR) $\mathcal I = \{ (\mathbb Y_t, \sigma_t) \}_{t \in [T+1]}$ is an information state generator for $\varrho_{\mathcal U}$ when the following criteria are satisfied:
\begin{itemize}
    \item[(i)] Sufficient to predict itself: ${\rm Pr}(Y_{t+1} \in B \vert h_t, z_t) = {\rm Pr}(Y_{t+1} \in B \vert \sigma_t(h_t), z_t)$ for all $t \in [T-1]$, $B \in \mathcal G_{t+1}$, $h_t \in \mathbb{H}_t$, and $z_t \in \Re$.
    \item[(ii)] Sufficient to characterize terminal conditions: $\bar{\Phi}_{T+1} = \Phi_{T+1} \in \bar{\mathcal L}_{T+1}$.
    \item[(iii)] Sufficient to characterize the risk frontier: $\mathcal U_t(h_t, \bar{\Phi}_{t+1}) = \mathcal U_t(h_t', \bar{\Phi}_{t+1})$ for $\bar{\Phi}_{t+1} \in \bar{\mathcal L}_{t+1}$ and all $h_t, h_t' \in \mathbb{H}_t$ such that $\sigma_t(h_t) = \sigma_t(h_t') = y_t$ for all $t\in[T]$.
\end{itemize}
\end{defn}
\noindent
We model the evolution of the information state with transition functions $g_t : \mathbb Y_t \times \Re \times \Xi \rightarrow \mathbb Y_{t+1}$ where $y_{t+1} = g_t(y_t, z_t, \xi_t)$ for all $t \in [T]$.

Under Definition~\ref{defn:information}, we let $\bar{\mathcal U} = \mathcal U \cap \{\bar{\mathcal L}_{1:T+1} \times \bar{\mathcal L}_{0:T} \times \bar{\mathcal L}_{0:T}\}$ be the subset of triples within the risk frontier which are Markov with respect to the information state.
Definition~\ref{defn:information}(iii) addresses the relationship between the risk frontier and the information state. For each $t \in [T]$, we define correspondences $\bar{\mathcal U}_t:\mathbb Y_t \times \bar{\mathcal L}_{t+1}\rightrightarrows \bar{\mathcal L}_{t+1} \times \Re \times \Re$ via
$$
\bar{\mathcal U}_t(y_t, \bar{\Phi}_{t+1}) \triangleq \mathcal U_t(h_t, \bar{\Phi}_{t+1}) \text{ s.t. } \sigma_t(h_t) = y_t,
$$
for all $\bar{\Phi}_{t+1} \in \bar{\mathcal L}_{t+1}$. We refer to $\{\bar{\mathcal U}_{t}\}_{t \in [T]}$ as the {\it compressed risk frontiers}, they give a temporal decomposition of $\bar{\mathcal U}$ on the information state:
\begin{equation}
\label{eq:bar_risk_frontier_recursive}
\bar{\mathcal U} = \{(X_{0:T}, \bar Z_{0:T}, \bar{\Phi}_{0:T}) 
 \text{ : } [(X_{t}, \bar Z_{t}, \bar{\Phi}_{t}) \vert y_t] \in \bar{\mathcal U}_{t}(y_t, \bar{\Phi}_{t+1}),\, \forall y_t \in \mathbb Y_t,\, \forall t \in [T]\}.
\end{equation}

The following theorem establishes the sufficiency of $\mathcal I$ and $\bar{\mathcal U}$ for the purpose of evaluating the GCR for Markov $X_{0:T}$.
Since $X_{0:T}$ is Markov on $\mathbb{S}$ only, any additional history dependence in $(\bar Z_{0:T}, \bar{\Phi}_{0:T})$ is solely due to DM's risk preferences.
\begin{thm}
\label{thm:Information_sufficiency}
Suppose $\mathcal I$ is an information state generator for $\varrho_{\mathcal U}$ and let $X_{0:T} \in \bar{\mathcal L}_{0:T}$.

(i) If $(X_{0:T}, Z_{0:T}, \Phi_{0:T}) \in \mathcal U$, then there exists $(X_{0:T}, \bar Z_{0:T}, \bar{\Phi}_{0:T}) \in \bar{\mathcal U}$ such that $\bar{\Phi}_{0:T} \leq \Phi_{0:T}$.

(ii) $\varrho_{\bar{\mathcal U}}(X_{0:T}) = \varrho_{\mathcal U}(X_{0:T})$.
\end{thm}

We can construct the tail risk frontiers on the information state as follows:
\begin{equation*}
\bar{\mathcal U}_{t:T}(y_t) = \{(X_{t:T}, \bar Z_{t:T}, \bar{\Phi}_{t:T}) 
 \text{ : } [(X_{k}, \bar Z_{k}, \bar{\Phi}_{k}) \vert y_k] \in \bar{\mathcal U}_{k}(y_k, \bar{\Phi}_{k+1}),\, \forall y_k \in \mathbb Y_k,\, \forall k \in [t, T]\}.
\end{equation*}
For all $t \in [T]$, we then define $\varrho_{\bar{\mathcal U}, t:T} : \bar{\mathcal L}_{t+1:T+1} \rightarrow \bar{\mathcal L}_t$ by:
\[
    [\varrho_{\bar{\mathcal U}, t:T}(X_{t:T})](y_t) \triangleq \min_{\bar{Z}_{t:T}, \bar{\Phi}_{t:T}} \{ \bar{\Phi}_t(y_t) : [(X_{t:T}, \bar{Z}_{t:T}, \bar{\Phi}_{t:T}) \vert y_t] \in \bar{\mathcal U}_{t:T}(y_t) \},\, \forall y_t \in \mathbb{Y}_t,
\]
for all $X_{t:T} \in \bar{\mathcal L}_{t:T}$.
The resulting dynamic risk measure induced by $\bar{\mathcal U}$ is then $\{\varrho_{\bar{\mathcal U}, t:T}\}_{t \in [T]}$.
We take $\varrho_{\bar{\mathcal U}, T+1 : T} = \bar{\Phi}_{T+1}$ for the terminal period, and obtain the following recursion on the information state the risk evaluation of Markov payoff streams.

\begin{thm}
\label{thm:Information_recursive}
Suppose Assumption~\ref{assu:one-step} holds and $\mathcal I$ is an information state generator for $\varrho_{\mathcal U}$.
Then $\{\varrho_{\bar{\mathcal U}, t:T}\}_{t \in [T]}$ has a recursive decomposition on the information state. For all $t \in [T]$ and $X_{t:T} \in \bar{\mathcal L}_{t+1:T+1}$, we have:
\begin{equation}\label{eq:Information_recursive}
[\varrho_{\bar{\mathcal U}, t:T}(X_{t:T})](\sigma_t(h_t)) = \min_{z_t, \varphi_t} \{\varphi_t : ([X_t \vert s_t], z_t, \varphi_t) \in \bar{\mathcal U}_t(\sigma_t(h_t), \varrho_{\bar{\mathcal U}, t+1 : T}(X_{t+1:T}))\},\, \forall h_t \in \mathbb{H}_t. 
\end{equation}
\end{thm}

\begin{example}
Suppose $\mathcal I$ is an information state generator for $\varrho_{\mathcal A, \mathcal C, \phi}$.
Then, there are correspondences: $\bar{\mathcal A}_t:\mathbb Y_t\rightrightarrows \bar{\mathcal L}_{t+1}$ such that $\bar{\mathcal A}_t(y_t) = \bar{\mathcal A}_t(\sigma_t(h_t)) = \mathcal A_t(h_t)$; and $\bar{\mathcal C}_t:\mathbb Y_t\rightrightarrows \bar{\mathcal L}_t$ such that $\bar{\mathcal C}_t(y_t) = \bar{\mathcal C}_t(\sigma_t(h_t)) = \mathcal C_t(h_t)$.
Let $\bar{\Phi}_{T+1} \equiv 0$, then $\bar{\mathcal U}_t(y_t, \bar{\Phi}_{t+1})$ are determined by the constraints: $\varphi_t \geq \phi_t(y_t, z_t) + \mathbb E[\bar{\Phi}_{t+1} \vert y_t, z_t]$; $[X_t \vert s_t] + z_t \in \bar{\mathcal A}_t(y_t)$; and $z_t \in \bar{\mathcal C}_t(y_t)$.
\end{example}

\section{GCR-MDP Formulation}\label{sec:problem}

In this section, we consider an MDP which generates a payoff stream.
We want to find a policy that minimizes the GCR which reflects DM's risk preferences.

The action space is $\mathbb{A}$, and $a_{t} \in \mathbb A$ is the action in period $t \in [T]$.
The transition law now depends on the state and action and is given by $f_t : \mathbb{S} \times \mathbb A \times \Xi \rightarrow \mathbb{S}$ for all $t \in [T-1]$.
The state transitions then follow $s_{t+1} = f_t(s_t,a_t,\xi_t)$ for all $t \in [T-1]$.
The payoffs $r_t\text{ : }\mathbb{S} \times \mathbb{A} \times \Xi \rightarrow \Re$ for $t \in [T]$ also depend on the state and action.
We let $R_t(s_t, a_t)$ denote the random payoff in period $t$ with realizations defined by $[R_t(s_t, a_t)](\xi_t) = r_t(s_t, a_t, \xi_t)$ for all $\xi_t \in \Xi$.

We make the following basic assumptions on the underlying MDP.
\begin{assumption}
\label{assu:MDP}
(i) $\mathbb A$ is finite.

(ii) $|r_t(s, a, \xi)| \leq R_{\max} < \infty$ for all $(s, a, \xi) \in \mathbb{S} \times \mathbb{A} \times \Xi$ and $t \in [T]$.
\end{assumption}
\noindent
Under Assumptions~\ref{assu:preliminaries} and \ref{assu:MDP}, the state and action spaces are both finite.

We modify the history to account for the actions of the MDP, where $h_t = (s_0, a_0, \xi_0, \ldots, a_{t-1}, \xi_{t-1})$ for all $t \in [T]$, and $h_{T+1} = (s_0, a_0, \xi_0, \ldots, a_{T}, \xi_{T})$.
In this setting, DM's policy covers both operational decisions which generate the payoff stream as well as the disbursement strategy which is part of the GCR. We let $\pi=\left(\pi_{t}\right)_{t\in[T]}$ be a sequence of mappings $\pi_t : \mathbb{H}_t \rightarrow \mathbb A \times \mathbb R$ for all $t \in [T]$, and we let $\Pi$ denote the set of all such mappings.
We write $\pi_t = (\pi_t^a, \pi_t^z)$ where $\pi_t^a : \mathbb{H}_t \rightarrow \mathbb A$ is the operational policy and $\pi_t^z : \mathbb{H}_t \rightarrow \Re$ is the financial policy.
For a given policy $\pi \in \Pi$, DM's payoff stream is $R_{0:T}^\pi = (R_0^{\pi}, \ldots, R_T^{\pi}) \in \mathcal L_{1:T+1}$, where $R_t^{\pi}(h_t, \xi_t) \triangleq r_t(s_t, \pi_t^a(h_t), \xi_t)$ for all $h_t \in \mathbb{H}_t$, $\xi_t \in \Xi$, and $t \in [T]$.
The disbursement strategy is $Z_{0:T}^\pi = (Z_0^\pi, \ldots, Z_T^\pi) \in \mathcal L_{0:T}$, where $Z_t^\pi(h_t) = \pi_t^z(h_t)$ for all $h_t \in \mathbb{H}_t$ and $t \in [T]$.

Let $\mathcal U$ be DM's risk frontier for $R_{0:T}^\pi$, then we define the GCR-MDP:
\begin{equation}
\label{MDP_CR}
    \min_{\pi \in \Pi} \varrho_{\mathcal U}(R_{0:T}^\pi) \equiv \min_{\pi \in \Pi,\, \Phi_{0:T}} \left\{ \varphi_0 : (R_{0:T}^\pi, Z_{0:T}^\pi, \Phi_{0:T}) \in \mathcal U \right\}.
\end{equation}
We treat the premium schedule $\Phi_{0:T}$ as part of DM's decision variables in Eq.~\eqref{MDP_CR} in addition to the policy $\pi \in \Pi$ (which determines $R_{0:T}^\pi$ and $Z_{0:T}^\pi$).
The optimal premium schedule is uniquely determined by the minimization objective for a given $\pi \in \Pi$.
We denote an optimal policy for Eq.~\eqref{MDP_CR} as $\pi^* \in \Pi$ (which is not necessarily unique).

Earlier we considered time consistency of a GCR in terms its dynamic risk measures.
We can consider time consistency of a GCR-MDP in terms of its optimal policies.

\begin{defn}
\label{defn:time_consistency_policy}
An optimal policy $\pi^*$ is time-consistent for Eq.~\eqref{MDP_CR} if $\pi_{t:T}^* = (\pi_t^*, \ldots, \pi_T^*)$ is optimal for $V_t(h_t)$ for all $h_t \in \mathbb{H}_t$ and $t \in [T]$.
\end{defn}

\subsection{Dynamic Programming on Histories}
We first develop the general DP decomposition of Eq.~\eqref{MDP_CR} and identify an optimal policy.
For each $t \in [T+1]$, we let $\mathcal V_t$ be the space of bounded value functions $V : \mathbb{H}_t \rightarrow \Re$ defined on the history $\mathbb{H}_t$, equipped with the supremum norm.
We take $V_{T+1} \in \mathcal V_{T+1}$ as an input.
Given $\pi \in \Pi$, we let $\pi_{t:T} = (\pi_t, \ldots, \pi_T)$ denote the tail policy starting from period $t \in [T]$.
We define the value functions
\[
    V_t^{\pi_{t:T}}(h_t) \triangleq \min_{\Phi_{t:T}} \{\Phi_t(h_t) : [(R_{t:T}^{\pi_{t:T}}, Z_{t:T}^{\pi_{t:T}},\Phi_{t:T}) \vert h_t] \in \mathcal U_{t:T}(h_t)\},
\]
corresponding to the tail policy $\pi_{t:T}$ conditional on $h_t \in \mathbb{H}_t$.
Then for each period $t \in [T]$ we define $V_t \in \mathcal V_t$ via:
\[
    V_t(h_t) \triangleq \min_{\pi_{t:T}} V_t^{\pi_{t:T}}(h_t),\, \forall h_t \in \mathbb{H}_t,
\]
which optimizes over the tail policy $\pi_{t:T}$ conditional on $h_t$.
We interpret $V_t(h_t)$ as the optimal total cost of the remaining premiums in periods $[t, T]$, conditional on history $h_t \in \mathbb{H}_t$.
We will establish that $\{V_t\}_{t=0}^{T+1}$ are the optimal value functions for Eq.~\eqref{MDP_CR}.

In each period $t \in [T]$, we will suppose that DM directly optimizes over $(X_t, z_t, \varphi_t)$ and the next period premium $\Phi_{t+1} \in \mathcal L_{t+1}$.
The risk frontier $\mathcal U_t(h_t, \Phi_{t+1})$ captures the requirements on $(X_t, z_t, \varphi_t)$ due to the GCR.
We keep the optimization over $a_t \in \mathbb A$ implicit to emphasize the dynamics of the GCR, where $X_t$ is understood as a dummy variable for $R_t$, and $\Phi_{t+1}$ is understood as a dummy variable for $V_{t+1}$.
For $(X_t, \Phi_{t+1})$ to be achievable for the underlying MDP, it must be almost surely dominated by the actual reward and cost-to-go.
Let $\mathcal D_t : \mathbb{H}_t \times \mathcal L_{t+1} \rightrightarrows \mathcal L_{t+1} \times \Re \times \mathcal L_{t+1}$ be the correspondence for the dominance constraints defined by:
\begin{align*}
   \mathcal D_t(h_t, V_{t+1}) \triangleq & \Big\{(X_t, z_t, \Phi_{t+1}) : \exists a \in \mathbb{A} \text{ s.t. } X_t(h_t, a, \xi_t) \leq r_t(s_t, a, \xi_t),\\
   & V_{t+1}(h_t, a, \xi_t) \leq \Phi_{t+1}(h_t, a, \xi_t),\, \forall \xi_t \in \Xi \Big\}.
\end{align*}
Then, we define correspondences $\mathcal Q_t : \mathbb{H}_t \times \mathcal L_{t+1} \rightrightarrows \mathcal L_{t+1} \times \Re \times \Re$ via:
\begin{align*}
   \mathcal Q_t(h_t, V_{t+1}) \triangleq & \Big\{(X_t, z_t, \varphi_t) \in \mathcal U_t(h_t, \Phi_{t+1}) : (X_t, z_t, \Phi_{t+1}) \in \mathcal D_t(h_t, V_{t+1}) \Big\}
\end{align*}
to account for the requirements of the GCR (through the compressed risk frontier) and for the underlying MDP (through the dominance constraints).
Additionally, $\mathcal Q_t(h_t, V_{t+1})$ encodes the dependence of both the current period payoff and the next period cost-to-go on the current action $a_t \in \mathbb A$.

In the next theorem, we derive the DP recursion for the GCR-MDP under the recursive structure of Assumption~\ref{assu:one-step}.
This result is based on lattice-theoretic arguments to justify the interchange of the order of the minimization and the inclusion in the risk frontier.

\begin{thm}
\label{thm:History_DP_decomposition}
Suppose Assumption~\ref{assu:one-step} holds. The value functions $\{V_t\}_{t=0}^{T+1}$ satisfy the decomposition:
\begin{equation}
\label{eqn:Vthistory}
    V_t(h_t) = \min_{(X_t, z_t, \varphi_t)} \left\{\varphi_t : (X_t, z_t, \varphi_t) \in \mathcal Q_t(h_t, V_{t+1}) \right\}, \forall h_t \in \mathbb{H}_t, \forall t \in [T].
\end{equation}
Furthermore, $V_0(s_0) = \min_{\pi \in \Pi} \varrho_{\mathcal U}(R_{0:T}^\pi)$.
\end{thm}

Let $\pi^* = (\pi_t^*)_{t=0}^T$ be the greedy policy with respect to $\{V_t\}_{t=0}^T$, where $\pi_t^*$ is defined by:
\begin{subequations}
\label{eqn:Vthistory_policy}
\begin{align}
    \pi_t^{a*}(h_t) = & a_t^*,\\
    R_t(s_t, a_t^*) = & X_t^*,\, a_t^* \in \mathbb{A},\\
    \pi_t^{z*}(h_t) = & z_t^*,\\
    (X_t^*, z_t^*, \varphi_t^*) \in & \argmin_{(X_t, z_t, \varphi_t)} \left\{\varphi_t : (X_t, z_t, \varphi_t) \in \mathcal Q_t(h_t, V_{t+1}) \right\},\, \forall h_t \in \mathbb{H}_t.
\end{align}
\end{subequations}

\begin{thm}
\label{thm:History_policy}
Suppose Assumption~\ref{assu:one-step} holds.

(i) $\pi^*$ is optimal for Eq.~\eqref{MDP_CR}.

(ii) $\pi^*$ is time-consistent.
\end{thm}

\subsection{Dynamic Programming on Information States}\label{sub:infostate}

The next definition builds on \cite{subramanian2022approximate} to extend the information state concept to the GCR-MDPs, where we must additionally account for the controlled payoff stream.

\begin{defn}
\label{defn:information_MDP}
(Information state generator for GCR-MDP) Let $\mathcal I = \{(\mathbb Y_t, \sigma_t) \}_{t \in [T]}$ be an information state generator for the GCR.
Then $\mathcal I$ is an information state generator for the GCR-MDP when we replace Definition~\ref{defn:information}(i) with: ${\rm Pr}(Y_{t+1} \in B \vert h_t, a_t, z_t) = {\rm Pr}(Y_{t+1} \in B \vert \sigma_t(h_t), a_t, z_t)$ for all $B \in \mathcal G_{t+1}$, $a_t \in \mathbb A$, $z_t \in \Re$, for all $t \in [T]$.
\end{defn}
\noindent
We let $g_t : \mathbb Y_t \times \mathbb A \times \Re \times \Xi \rightarrow \mathbb Y_{t+1}$ for $t \in [T]$ be the modified transition functions for the information state to account for the controls. We then have $y_{t+1} = g_t(y_t, a_t, z_t, \xi_t)$ for all $t \in [T]$.

We now consider policies $\bar\pi = (\bar\pi_t)_{t \in [T]}$ where $\bar \pi_t : \mathbb Y_t \to \mathbb A \times \mathbb R$ only depends on the information state for all $t \in [T]$, and we let $\bar \Pi \subset \Pi$ denote the class of such policies.
The GCR-MDP defined on the information state is:
\begin{equation}
\label{MDP_Information}
    \min_{\bar \pi \in \Pi} \varrho_{\bar{\mathcal U}}(R_{0:T}^{\bar \pi}) \equiv \min_{\bar \pi \in \Pi, \bar{\Phi}_{0:T} \in \mathcal L_{0:T}} \left\{ \varphi_0 : (R_{0:T}^{\bar \pi}, Z_{0:T}^{\bar\pi}, \bar{\Phi}_{0:T}) \in \bar{\mathcal U} \right\}.
\end{equation}

We next develop the DP decomposition for Eq.~\eqref{MDP_Information}.
For each $t \in [T]$, let $\bar{\mathcal V}_t$ be the space of bounded value functions $V : \mathbb Y_t \rightarrow \Re$ defined on $\mathbb Y_t$ for all $t \in [T]$, equipped with the supremum norm.
We again take the terminal value function $\bar V_{T+1} \in \bar{\mathcal V}_{T+1}$ as an input.
Given $\bar{\pi} \in \bar \Pi$, we define the tail policy $\bar{\pi}_{t:T} = (\bar{\pi}_t, \ldots, \bar{\pi}_T)$ and value functions
\[
    \bar V_t^{\bar \pi_{t:T}}(y_t) \triangleq \min_{\bar{\Phi}_{t:T}} \{\bar \Phi_t(y_t) : [(R_{t:T}^{\bar \pi_{t:T}}, \bar Z_{t:T}^{\bar \pi_{t:T}}, \bar{\Phi}_{t:T}) \vert y_t] \in \bar{\mathcal U}_{t:T}(y_t)\},\, \forall y_t \in \mathbb Y_t,
\]
for all $t \in [T]$.
We then define the optimal value functions to be:
\begin{equation}
    \bar V_t(y_t) \triangleq \min_{\bar \pi_{t:T}} \bar V_t^{\bar \pi_{t:T}}(y_t),\, \forall y_t \in \mathbb Y_t,
\end{equation}
for all $t \in [T]$.

In each period $t \in [T]$, we continue to let $X_t \in \bar{\mathcal L}_{t+1}$ and $\bar{\Phi}_{t+1} \in \bar{\mathcal L}_{t+1}$ be dummy variables directly controlled by DM.
The compressed risk frontier $\bar{\mathcal U}_t(y_t, \bar{\Phi}_{t+1})$ reflects the requirements on $(X_t, z_t, \varphi_t)$ on the information state due to the GCR.
Let $\bar{\mathcal D}_t : \mathbb{Y}_t \times \bar{\mathcal L}_{t+1} \rightrightarrows \bar{\mathcal L}_{t+1} \times \Re \times \bar{\mathcal L}_{t+1}$ be the correspondence for the dominance constraints on the information state defined by:
\begin{align*}
   \bar{\mathcal D}_t(y_t, \bar V_{t+1}) \triangleq & \Big\{(X_t, z_t, \bar{\Phi}_{t+1}) : \exists a \in \mathbb{A} \text{ s.t. } X_t(y_t, a, \xi_t) \leq r_t(s_t, a, \xi_t),\\
   & \bar{V}_{t+1}(g_t(y_t, z_t, a, \xi_t)) \leq \bar{\Phi}_{t+1}(g_t(y_t, z_t, a, \xi_t)),\, \forall \xi_t \in \Xi \Big\}.
\end{align*}
Then, we define correspondences $\bar{\mathcal Q}_t : \mathbb{Y}_t \times \bar{\mathcal L}_{t+1} \rightrightarrows \bar{\mathcal L}_{t+1} \times \Re \times \Re$ via:
\begin{align*}
   \bar{\mathcal Q}_t(y_t, \bar{V}_{t+1}) \triangleq & \Big\{(X_t, z_t, \varphi_t) \in \bar{\mathcal U}_t(y_t, \bar{\Phi}_{t+1}) : (X_t, z_t, \bar{\Phi}_{t+1}) \in \bar{\mathcal D}_t(y_t, \bar V_{t+1}) \Big\}.
\end{align*}

The next result generalizes \cite[Theorem 5]{subramanian2022approximate}, which does not consider constraints on the payoff stream and feasible controls in an information state MDP.
We show that $\{\bar V_t\}_{t=0}^{T+1}$ are the value functions for Eq.~\eqref{MDP_Information}.
\begin{thm}
\label{thm:Information_DP_decomposition}
Suppose Assumption~\ref{assu:one-step} holds and Eq.~\eqref{MDP_CR} has an information state generator $\mathcal I$.
The value functions $\{\bar V_t\}_{t=0}^{T+1}$ satisfy the decomposition:
\begin{equation}
\label{eq:DP_nonstandard}
    \bar V_t(y_t) = \min_{(X_t, z_t, \varphi_t)} \{ \varphi_t : (X_t, z_t, \varphi_t) \in \bar{\mathcal Q}_t(y_t, \bar V_{t+1}) \},\, \forall y_t \in \mathbb{Y}_t,\, \forall t \in [T].
\end{equation}
Furthermore, $\bar V_0(s_0) = \min_{\bar \pi \in \bar \Pi} \varrho_{\bar{\mathcal U}}(R_{0:T}^{\bar \pi})$.
\end{thm}
\noindent
The decompositions of Theorem~\ref{thm:History_DP_decomposition} and Theorem~\ref{thm:Information_DP_decomposition} have similar structure, with the cost-to-go problems in epigraphical form. Define the greedy policy $\bar \pi^* = (\bar \pi_t^*)_{t=0}^T$ with respect to $\{\bar V_t\}_{t=0}^T$ by:

\begin{subequations}
\label{eq:DP_nonstandard_policy}
\begin{align}
    \bar \pi_t^{a*}(\sigma_t(h_t)) = & a_t^*,\\
    R_t(s_t, a_t^*) = & X_t^*,\, a_t^* \in \mathbb{A},\\
    \bar \pi_t^{z*}(\sigma_t(h_t)) = & z_t^*,\\
    (X_t^*, z_t^*, \varphi_t^*) \in & \argmin_{(X_t, z_t, \varphi_t)} \left\{\varphi_t : (X_t, z_t, \varphi_t) \in \bar{\mathcal Q}_t(\sigma_t(h_t), V_{t+1}) \right\},\, \forall h_t \in \mathbb{H}_t,\, \forall t \in [T].
\end{align}
\end{subequations}

\begin{thm}
\label{thm:Information_policy}
Suppose Assumption~\ref{assu:one-step} holds and the GCR-MDP has an information state generator $\mathcal I$.

(i) $\bar \pi^*$ is optimal for Eq.~\eqref{MDP_Information}.

(ii) $\bar \pi^*$ is time-consistent.
\end{thm}
\noindent
When an information state exists, $\bar \pi^*$ is also optimal for Eq.~\eqref{MDP_CR}.

\section{Applications}
\label{sec:applications}

We develop several GCR examples in this section, where we report the information state and the compressed risk frontiers (the full derivations are given in Appendix~\ref{sec:applications_Appendix}).
Here the information state usually only requires one augmenting state variable (e.g., wealth, target shortfall, or risk budget).
In addition, all of these examples turn out to be time consistent by Theorem~\ref{thm:recursive}.
We may then apply Theorem~\ref{thm:Information_recursive} to get the temporal decomposition of the GCR on the information state.
For these examples, we treat the canonical form of the GCR where there is no underlying MDP.
Based on Section~\ref{sec:problem}, we can easily adapt these examples to the corresponding GCR-MDP by applying Theorem~\ref{thm:Information_DP_decomposition}.

\subsection{Original State Space}
\label{subsec:examples_original}

We start with some GCRs that can be set up on the original state space with $\mathbb Y_t = \mathbb{S}$ where $y_t = \sigma_t(h_t) = s_t$ for all $t \in [T]$.
The information state transition functions are $g_t(y_t, \xi_t) = f_t(s_t, \xi_t)$ for all $t \in [T-1]$.

\subsubsection{Risk-neutral}

Let $\varrho_{{\rm RN}}(X_{0:T}) = - \mathbb E[ \sum_{t=0}^T X_t ]$ be the negative expected total reward (to align with the minimization objective of the GCR).
We set $\Phi_{T+1} \equiv 0$ since there are no terminal conditions, then $\bar{\mathcal U}_t(y_t, \bar{\Phi}_{t+1})$ (for $(X_t, \varphi_t)$) is determined by the constraint $\varphi_t \geq \mathbb E[ \bar{\Phi}_{t+1}(g_t(y_t, \xi_t)) - X_t(s_t, \xi_t)]$.

\subsubsection{Entropic risk measure}

Let $\varrho_{{\rm Ent}}(X_{0:T}) = (-1/\gamma) \log (\mathbb{E} [ \exp (-\gamma \sum_{t = 0}^{T} X_t ) ])$ be the entropic risk measure of the cumulative payoff for $\gamma > 0$ where $X_{0:T} \geq 0$.
Since $\log(\cdot)$ is monotone, we can equivalently minimize $\mathbb{E} [ \exp (- \gamma \sum_{t = 0}^{T} X_t ) ]$.
We set $\Phi_{T+1} \equiv 1$ corresponding to a terminal payoff of $\exp(0) = 1$.
Then, $\bar{\mathcal U}_t(y_t, \bar{\Phi}_{t+1})$ (for $(X_t, \varphi_t)$) is determined by the constraint $\varphi_t \geq \mathbb E[e^{- \gamma\, X_t(s_t, \xi_t)} \bar{\Phi}_{t+1}(g_t(y_t, \xi_t))]$.

\subsubsection{Nested risk measures}

For all $t\in [T]$, let $\mathcal A_t \subset \mathcal L(\Xi, \mathcal B(\Xi), P_t)$ be an acceptance set for payoffs (which is independent of history).
We suppose $\{\mathcal A_t\}_{t \in [T]}$ are all monotone, closed, convex, and pointed.
Then let $\{ \rho_t \}_{t \in [T]}$ be a collection of one-step coherent risk measures defined by $\rho_t(X_t) = \min_{z_t \in \Re} \{z_t : X_t + z_t \in \mathcal A_t\}$ for all $t \in [T]$.

The resulting nested risk measure on payoff streams (see \cite{ruszczynski2006conditional,ruszczynski2010risk}) is defined recursively by:
\[
\varrho_{t:T}(X_{t:T}) = \rho_t(X_t + \varrho_{t+1:T}(X_{t+1:T})),\, \forall X_{t:T} \in \mathcal L_{t+1:T+1},\,\forall t \in [T],
\]
and we let $\varrho_{{\rm N}}(X_{0:T}) \triangleq \varrho_{0:T}(X_{0:T})$.
We set $\Phi_{T+1} \equiv 0$, then $\bar{\mathcal U}_t(y_t, \bar{\Phi}_{t+1})$ is determined by the constraints: $\varphi_t \geq z_t$ and $[X_t \vert s_t] - [\bar{\Phi}_{t+1} \vert y_t] + z_t \in \mathcal A_t$.
The entire distribution of the next period premium enters into the inclusion for the acceptance set, rather than just its expectation.

\subsection{Wealth-based}
\label{subsec:examples_wealth}

Many GCRs naturally correspond to wealth-dependent preferences.
When DM is a firm, regulators may require it to maintain cash reserves to cover existing obligations. For instance, an insurance firm must keep reserves to meet claims; an energy company must maintain reserves for contingencies; investment firms have margin requirements to cover unexpected losses (see, e.g., \cite{capponi2018clearinghouse}).
The information state in these examples is $\mathbb Y_t = \mathbb{S} \times \mathbb W$ where $y_t = \sigma_t(h_t) = (s_t, w_t)$ for all $t \in [T]$.
We embed the wealth update (which is different for each example) into the risk frontier.

\subsubsection{Standard capital requirement}

We take the product acceptance set $\mathcal A_{\times}$ for $\varrho_{\mathcal A_{\times}}$.
The wealth evolves according to $W_{t+1} = W_t - Z_t$ for all $t \in [T]$, where $Z_{0:T}$ are disbursements.
Let $g_t(y_t, z_t, \xi_t) = (f_t(s_t, \xi_t), w_t - z_t)$ be the transition functions for $t \in [T-1]$.
The initial information state is $y_0 = (s_0, 0)$, and the endowment $w_0 \in \Re$ is a decision variable.
We let $\Phi_{T+1} \equiv \delta_{\{W_{T+1} \geq 0\}}$ (which is the indicator function which returns zero if $W_{T+1} \geq 0$ and positive infinity otherwise), then $\bar{\mathcal U}_T(y_T, \bar{\Phi}_{T+1})$ is determined by the constraints: $[X_T \vert s_T] + z_T \in \mathcal A_T$ and $w_T \geq z_T$.
Next, $\bar{\mathcal U}_t(y_t, \bar{\Phi}_{t+1})$ for $t \in [1, T-1]$ is determined by the constraints: $\varphi_t \geq \mathbb E[\bar{\Phi}_{t+1}(g_t(y_t, z_t, \xi_t))]$ and $[X_t \vert s_t] + z_t \in \mathcal A_t$.
Finally, $\bar{\mathcal U}_0(y_0, \bar{\Phi}_{1})$ is determined by the constraints: $\varphi_0 \geq w_0 + \mathbb E[\bar{\Phi}_1(g_0(y_0, z_0, \xi_0))]$ and $X_0 + z_0 \in \mathcal A_0$.

\subsubsection{Consumption}

This example is based on the dynamic risk measures developed in \cite{pflug2001risk,pflug2005measuring}, where $Z_{0:T} \in \mathcal{C}_{\geq 0}$ are consumption decisions.
The wealth update is given by $\gamma(W_t, X_t, Z_t) \triangleq \max\{ W_t + X_t - Z_t, 0 \}$.
Since consumption is pre-committed, there may be consumption shortfall $\max \{Z_t - W_t - X_t, 0 \}$ in period $t \in [T]$.

Let $\alpha_t > 0$ be the unit utility from consumption and $c_t > \alpha_t$ be the unit cost for consumption shortfall in period $t \in [T]$, then let $\psi_t(W_t, X_t, Z_t) \triangleq c_t \max \{Z_t - W_t - X_t, 0 \} - \alpha_t Z_t$ be the loss (shortfall cost minus utility) in period $t \in [T]$.
DM wants to solve:
\begin{subequations}
\label{eq:consumption}
\begin{align}
    \varrho_{{\rm C}}(X_{0:T}) = \min_{Z_{0:T}, W_{1:T+1}} \quad & \mathbb E \left[\sum_{t=0}^T \psi_t(W_t, X_t, Z_t) \right],\\
    \text{s.t.} \quad & W_{t+1} = \gamma(W_t, X_t, Z_t),\, \forall t \in [T],\\
    & Z_{0:T} \in \mathcal{C}_{\geq 0}.
\end{align}
\end{subequations}
The transition functions are $g_t(y_t, z_t, \xi_t) = (f_t(s_t, \xi_t), \gamma(w_t, X_t(s_t, \xi_t), z_t))$ for all $t \in [T-1]$.
Let $\Phi_{T+1} \equiv 0$, then $\bar{\mathcal U}_t(y_t, \bar{\Phi}_{t+1})$ is determined by the constraints:
\[
\varphi_t \leq \mathbb E[\psi_t(w_t, [X_t \vert s_t], z_t) + \bar{\Phi}_{t+1}(g_t(y_t, z_t, \xi_t))]\}.
\]

\subsubsection{Consumption excess}

This example is based on \cite{chen2015dynamic}, where $Z_{0:T}$ are consumption decisions.
Let $0 \leq \beta_t^S \leq \beta_t^B$ be the savings return rate and loan interest rate, respectively, where wealth evolves according to
$W_{t+1} = \gamma_t(W_t, Z_t) \triangleq \min\{(1 + \beta_t^S)W_t - Z_t, (1 + \beta_t^B)W_t - Z_t\}$ for all $t \in [T]$.
Wealth may become negative but we require all loans to be repaid by the end of the planning horizon.

The total amount consumed in period $t \in [T]$ is $X_t + z_t$, the realized income plus consumption (positive $z_t$ are a withdrawal from savings to augment consumption, negative $z_t$ are a deposit to savings from realized income).
Let $\zeta_t : \mathbb{S} \rightarrow \Re$ be (exogenous) state-dependent consumption targets for all $t \in [T]$, where $\zeta_t(s_t)$ is the target in period $t$ in state $s_t$. The {\em consumption excess} profile is $(X_t + z_t - \zeta_t(s_t))_{t \in [T]}$, where $X_t + z_t - \zeta_t(s_t)$ is the consumption excess in period $t$.
Let $u : \Re \rightarrow \Re$ be an increasing concave utility function, and introduce auxiliary variables $\alpha_t \in \Re$ for all $t \in [T]$. We define $\psi_t(\alpha_t, X_t, z_t; s_t) \triangleq \alpha_t \mathbb E[u((X_t + z_t - \zeta_t(s_t)) / \alpha_t) \vert s_t]$ for all $t \in [T]$, and $\Psi(\alpha_{0:T}, X_{0:T}, Z_{0:T}) \triangleq ((\psi_t(\alpha_t, X_t, Z_t; s_t))_{s_t \in \mathbb{S}})_{t \in [T]}$ to be the vector of consumption excess measures.
For fixed $\epsilon > 0$, we consider:
\begin{subequations}
\label{eq:consumption_excess}
\begin{align}
    \varrho_{{\rm CE}}(X_{0:T}) = \min_{Z_{0:T}, W_{1:T+1}, \alpha_{0:T}} \quad & \sum_{t=0}^T \alpha_t\\
    \text{s.t.} \quad & W_{t+1} = \gamma_t(W_t, Z_t),\, \forall t \in [T],\\
    & \Psi(\alpha_{0:T}, X_{0:T}, Z_{0:T}) \geq 0,\\
    & Z_{0:T} \in \mathcal{C}_{w},\\
    & \alpha_t \geq \epsilon,\, \forall t \in [T],
\end{align}
\end{subequations}
which returns $+\infty$ unless all consumption targets are met due to $\Psi(\alpha_{0:T}, X_{0:T}, Z_{0:T}) \geq 0$.

The transition functions are $g_t(y_t, z_t, \xi_t) = (f_t(s_t, \xi_t), \gamma_t(w_t, z_t))$ for all $t \in [T-1]$.
We let $\Phi_{T+1} \equiv 0$, then $\bar{\mathcal U}_t(y_t, \bar{\Phi}_{t+1})$ is determined by the constraints:
\[
\varphi_t \geq \alpha_t + \max_{\xi_t \in \Xi} \bar{\Phi}_{t+1}(g_t(y_t, z_t, \xi_t)),\, \psi_t(\alpha_t, X_t, z_t; s_t) \geq 0,\, \alpha_t \geq \epsilon.
\]


\subsection{Target-based}
\label{subsec:examples_target}

The following examples are all based on performance targets and different measures of target shortfall.
We let $Z_{0:T} \in \mathcal{C}_{\geq 0}$ be the degree of target shortfall (which is non-negative so we cannot make up for past shortfall).
We take the acceptance set $\mathcal A_{\times}$ so the performance targets are represented by $\mathcal A_t$ in each period, and we require $X_{0:T} + Z_{0:T} \in \mathcal A_{\times}$.
For all $t \in [T]$ we let $M_t \in \mathcal L(\Omega, \mathcal{F}_t, \Re^{|\Xi|})$ denote an auxiliary variable, where $M_{0:T} = (M_t)_{t=0}^T$ is the sequence of auxiliary variables.
We let $m_t = (m_t(\xi))_{\xi \in \Xi} \in \Re^{|\Xi|}$ denote a realization of $M_t$.
We also let $\eta_t \in \Re$ denote an augmenting state variable.

\subsubsection{Expected utility}

Let $u : \Re \rightarrow \Re$ be an increasing convex dis-utility function.
DM wants to minimize the expected dis-utility of total target shortfall:
\begin{equation*}
    \varrho_{{\rm EU}}(X_{0:T}) = \min_{Z_{0:T}} \left\{ \mathbb E \left[ u \left(\sum_{t=0}^T Z_t \right) \right] : X_{0:T} + Z_{0:T} \in \mathcal A_{\times}, Z_{0:T} \in \mathcal{C}_{\geq 0} \right\}.
\end{equation*}
Define $\eta_t = \sum_{k=0}^{t-1} Z_k$ to be the cumulative target shortfall up to the beginning of period $t \in [T+1]$ (where $\eta _0 = 0$ and $\eta_{T+1}$ is the total shortfall).
The information state is $\mathbb Y_t = \mathbb{S} \times \Re_{\geq 0}$ where $y_t = \sigma_t(h_t) = (s_t, \eta_t)$ for all $t \in [T]$.
The transition functions are $g_t(y_t, z_t, \xi_t) = (f_t(s_t, \xi_t), \eta_t + z_t)$ for $t \in [T-1]$ and $g_T(y_T, z_T, \xi_T) = \eta_T + z_T$.
Let the terminal premium be $\Phi_{T+1}(h_{T+1}) \equiv u( \sum_{t=0}^T Z_t(h_{T+1}))$ for all $h_{T+1} \in \mathbb{H}_{T+1}$.
Then $\bar{\mathcal U}_t(y_t, \bar{\Phi}_{t+1})$ is determined by the constraints: $\varphi_t \geq \mathbb E[\bar{\Phi}_{t+1}(g_t(y_t, z_t, \xi_t))]$ and $[X_t \vert s_t] + z_t \in \mathcal A_t$.

\subsubsection{Worst-case expectation}

Suppose there is an adversary who can perturb the transition kernels.
For each $t \in [T]$, let $\Delta_t \triangleq \{ m_t \geq 0 : \mathbb E_{P_t}[m_t] = 1 \}$ be the set of densities with respect to $P_t$.
Let $M_t \in \Delta_t$ be the adversary's perturbation of $P_t$ to obtain $M_t P_t$.
Let $K \geq 0$ be a perturbation budget and $d(\cdot, \cdot)$ be a metric on $\mathcal P(\Xi)$.
Let
$$
\mathcal Q = \left\{ M_{0:T} : \sum_{t=0}^T d(M_t P_t, P_t) \leq K, M_t \in \Delta,\, \forall t \in [T] \right\}
$$
be the set of feasible adversary strategies with bounded total perturbation (see, e.g., \cite{mannor2016robust}).
DM wants to minimize the worst-case expected cumulative target shortfall:
\begin{equation}
    \varrho_{{\rm R}}(X_{0:T}) = \min_{Z_{0:T}} \left\{ \sup_{M_{0:T} \in \mathcal Q} \mathbb E_{M_{0:T}} \left[ \sum_{t=0}^T Z_t \right] : X_{0:T} + Z_{0:T} \in \mathcal A_{\times}, Z_{0:T} \in \mathcal{C}_{\geq 0} \right\}.
\end{equation}

Let $\eta_t$ be the remaining perturbation budget and let $\Delta_t(\eta_t) \triangleq \{m_t \in \Delta_t : d(m_t P_t, P_t) \leq \eta_t\}$ denote the constraints on the adversary's selection for all $t \in [T]$, where $\eta_0 = K$.
The information state is then $\mathbb Y_t = \mathbb{S} \times \Re$ where $y_t = \sigma_t(h_t) = (s_t, \eta_t)$ for all $t \in [T]$. 
Let $g_t(s_t, \eta_t, \xi_t) = (f_t(s_t, \xi_t), \eta_t)$ for $y_t = (s_t, \eta_t)$ for all $t \in [T]$ be the transition functions.
Let $\Phi_{T+1} \equiv 0$, then $\bar{\mathcal U}_t(y_t, \bar{\Phi}_{t+1})$ is determined by the constraints:
\begin{align*}
\varphi_t \geq & z_t + \mathbb E[\bar{\Phi}_{t+1}(g_t(s_t, \eta_t - d(m_t P_t, P_t), \xi_t)) m_t(\xi_t)],\, \forall m_t \in \Delta_t(\eta_t),\\
[X_t \vert s_t] + z_t \in & \mathcal A_t,\, z_t \geq 0.
\end{align*}

\subsubsection{Conditional value-at-risk}

Recall the conditional value-at-risk (for cost) at level $\alpha \in (0,1)$ is given in variational form by ${\rm CVaR}_{\alpha}(X) \triangleq \min_{\eta \in \Re} \{ \eta + (1 - \alpha)^{-1} \mathbb E[\max\{X-\eta, 0\}] \}$.
For $\alpha \in (0, 1)$, DM wants to solve:
\begin{equation}
    \varrho_{{\rm CVaR}}(X_{0:T}) = \min_{Z_{0:T}} \left\{ {\rm CVaR}_{\alpha} \left( \sum_{t=0}^T Z_t \right) : X_{0:T} + Z_{0:T} \in \mathcal A_{\times}, Z_{0:T} \in \mathcal{C}_{\geq 0} \right\},
\end{equation}
which minimizes the CVaR of the cumulative target shortfall (see \cite{chow2015risk,pflug2016time}).

We let $\eta_t$ denote the remaining risk budget and let $\Delta_t(\eta_t) \triangleq \{m_t \geq 0 : m_t \in [0, 1/\eta_t],\, \mathbb E_{P_t}[m_t] = 1 \}$ in period $t \in [T]$, where $\eta_0 = \alpha$. The information state $\mathbb Y_t = \mathbb{S} \times [0, 1)$ tracks the state and DM's risk level, and the history compression functions are $y_t = \sigma_t(h_t) = (s_t, \eta_t)$ for all $t \in [T]$.
The transition functions are $g_t(y_t, z_t, \xi_t) = (f_t(s_t, \xi_t), \eta_t)$ for all $t \in [T-1]$.
Let $\Phi_{T+1} \equiv 0$, then $\bar{\mathcal U}_t(y_t, \bar{\Phi}_{t+1})$ is determined by the constraints:
\begin{align*}
\varphi_t \geq & z_t + \mathbb E \left[\bar{\Phi}_{t+1}(g_t(s_t, \eta_t/m_t(\xi_t), \xi_t))m_t(\xi_t) \right],\, \, \forall m_t \in \Delta_t(\eta_t),\\
[X_t \vert s_t] + z_t \in & \mathcal A_t,\, z_t \geq 0.
\end{align*}

\subsubsection{Quantile risk measure}

Let $Q_{\tau}(X) \triangleq \inf\{ x : \mathbb P(X \leq x) \geq \tau \}$ for $\tau \in (0, 1)$ be the $\tau-$quantile of $X \in \mathcal L$.
We want to maximize the $\tau-$quantile of $- \sum_{t=0}^T Z_t$ (see \cite{li2017quantile}) by solving:
\begin{equation*}
    \varrho_{{\rm Q}}(X_{0:T}) = \min_{Z_{0:T}} \left\{ - Q_{\tau} \left(- \sum_{t=0}^T Z_t \right) : X_{0:T} + Z_{0:T} \in \mathcal A_{\times}, Z_{0:T} \in \mathcal{C}_{\geq 0} \right\}.
\end{equation*}

We let $\eta_t$ denote DM's risk budget and let $\Delta_t(\eta_t) \triangleq \{m_t \geq 0 : m_t(\xi_t) \in [0, 1],\, \mathbb E_{P_t}[m_t] \leq \eta_t \}$ in period $t \in [T]$, where $\eta_0 = \tau$.
The information state is $\mathbb Y_t = \mathbb{S} \times [0,1]$ and $y_t = \sigma_t(h_t) = (s_t, \eta_t)$ for all $t \in [T]$ where $\eta_t \in (0, 1)$ is DM's risk level in period $t$.
The transition function is $g_t(s_t, \eta_t, \xi_t) = (f_t(s_t, \xi_t), \eta_t)$ for $y_t = (s_t, \eta_t)$ in periods $t \in [T-1]$.
We set $\Phi_{T+1} \equiv 0$, then $\bar{\mathcal U}_t(y_t, \bar{\Phi}_{t+1})$ is determined by the constraints:
\begin{align*}
\varphi_t \geq & z_t + \max_{\{\xi_t \in \Xi,\, m_t(\xi_t) \ne 1\}} \left\{ \bar{\Phi}_{t+1}(g_t(s_t, \xi_t), m_t(\xi_t)) \right\},\\
[X_t \vert s_t] + z_t \in & \mathcal A_t,\, m_t \in \Delta_t(\eta_t).
\end{align*}

\subsubsection{Growth profile}

For each $t \in [T]$, let $\mathcal A_t(W_{t}) \subset \mathcal L(\Xi, \mathcal B(\Xi), P_t)$ be an acceptance set for $W_{t+1}$ defined on $\Xi$, which depends on the current wealth $W_t$.
We minimize the worst-case target shortfall by solving:
\begin{subequations}
\label{eq:growth}
\begin{align}
    \varrho_{{\rm G}}(X_{0:T}) = \min_{Z_{0:T}, W_{1:T+1}} \quad & \max\{Z_0, \ldots, Z_T\}\\
    \text{s.t.} \quad & W_{t+1} = W_{t} + X_t,\, \forall t \in [T],\\
    & W_{t+1} + Z_{t} \in \mathcal A_t(W_{t}),\, \forall t \in [T],\\
    & Z_{0:T} \in \mathcal{C}_{\geq 0},
\end{align}
\end{subequations}
which minimizes the expected worst-case target shortfall.

Let $\eta_t = \max\{ Z_0, Z_1, \ldots, Z_t \}$ be the maximum observed shortfall up through period $t \in [T]$.
The information state is $\mathbb Y_t = \mathbb{S} \times \mathbb W \times \Re_{\geq 0}$ where $y_t = \sigma_t(h_t) = (s_t, w_t, \eta_t)$ tracks both wealth and maximum observed shortfall.
Let $g_t(y_t, z_t, \xi_t) = (f_t(s_t, \xi_t), w_t + X_t(s_t, \xi_t), \max\{\eta_{t-1}, z_t \})$ be the transition function for period $t \in [T-1]$.
Let $\Phi_{T+1}(h_{T+1}) \equiv \max\{Z_0(h_{T+1}), \ldots, Z_T(h_{T+1})\}$ be the worst-case shortfall on $h_{T+1} \in \mathbb{H}_{T+1}$.
Then $\bar{\mathcal U}_t(y_t, \bar{\Phi}_{t+1})$ are determined by the constraints:
\begin{align*}
\varphi_t \geq & \bar{\Phi}_{t+1}(g_t(y_t, z_t, \xi_t)),\, \forall \xi_t \in \Xi,\\
w_t + [X_t \vert s_t] + z_t \in & \mathcal A_t(w_t),\, z_t \geq 0.
\end{align*}
This GCR requires two augmenting state variables for cumulative wealth and target shortfall to combine wealth-based and target-based GCRs.

\section{Simulations}\label{sec:simulations}

In this section we develop GCR-MDPs for the dynamic newsvendor and investigate their performance numerically.
We suppose backorders are not allowed and all unmet demand is lost.
Let $\mathbb{S} = \{0, 1, \ldots, s_{\max}\}$ be the space of allowable inventory levels (any inventory exceeding $s_{\max}$ is destroyed). Let the set of feasible order quantities be $\mathbb{A} = \{0, 1, \ldots, a_{\max}\}$ (where $a_{\max} \leq s_{\max}$).
Finally, let the discrete demand be supported on $\Xi = \{0, 1, \ldots, \xi_{\max}\}$.
We let $s_t \in \mathbb{S}$ be the inventory, $a_t \in \mathbb{A}$ be the order quantity, and $\xi_t \in \Xi$ be the random demand in period $t \in [T]$.
The inventory state equation is:
$$
s_{t+1} = f_t(s_t, a_t, \xi_t) = \min\{\max\{s_t + a_t - \xi_t, 0\}, s_{\max}\},\, \forall t \in [T-1].
$$
Let $p > c > h \geq 0$ be the unit selling price, order cost, and holding cost, respectively, so the payoff function is
$$
r(s_t, a_t, \xi_t) \triangleq p\, \min\{s_t + a_t, \xi_t\} - c\, a_t - h\, \max\{s_t + a_t - \xi_t, 0\}.
$$
Since $\mathbb{S}$, $\mathbb{A}$, and $\Xi$ are all finite, $r(s_t, a_t, \xi_t)$ can only take finitely many values.

\subsection{Newsvendor GCR Models}

We constrain disbursement/shortfall decisions to be integer-valued and belong to $\mathcal{C}_{z} \triangleq \{Z_{0:T} : Z_t \in \{0, 1, \ldots, z_{\max}\},\, \forall t \in [T]\}$.
Since payoffs and disbursements only take finitely many allowable values, we suppose DM's wealth only takes values in $\mathbb{W} \triangleq \{ w_{\min}, \ldots, -1, 0, 1, \ldots, w_{\max} \}$.
We define $\gamma : \Re \rightarrow \mathbb W$ via $\gamma(w) \triangleq \min\{ \max\{w, w_{\min}\}, w_{\max} \}$, which maps wealth to its allowable support $\mathbb W$.
The information state for the following wealth-based GCR-MDPs is $\mathbb Y_t = \mathbb{S} \times \mathbb W$, where $y_t = (s_t, w_t)$ for all $t \in [T]$.
We also let $\zeta \geq 0$ be a model parameter for the following three models (which takes a different meaning in each model).

First suppose DM has an income target $\zeta$ for every period, and $Z_{0:T}$ are disbursements.
Suppose the initial wealth is zero before the endowment $w_0$ is made.
We let $\beta > 1$ be the unit penalty for negative terminal wealth.
The corresponding GCR-MDP is:
\begin{subequations}
\label{newsvendor_standard}
\begin{align}
\min_{\pi \in \Pi,\, w_0 \in \mathbb W,\, W_{1:T+1}} \quad & w_0 - \beta\, \mathbb E[\min\{W_{T+1}, 0\}]\\
\text{s.t.} \quad & R_t^\pi + Z_t^\pi \in \mathcal A_t,\, \forall t \in [T],\\
& W_{t+1} = \gamma(W_t - Z_t^\pi),\, \forall t \in [T],\\
& Z_{0:T}^\pi \in \mathcal{C}_{\geq 0} \cap \mathcal{C}_{z}.
\end{align}
\end{subequations}
In the standard CR, there is an infinite penalty when the terminal wealth does not satisfy $W_{T+1} \geq 0$. In Eq.~\eqref{newsvendor_standard}, we instead give a continuous penalty for negative terminal wealth.
We let $\pi_{\rm SC}$ denote the optimal policy for Eq.~\eqref{newsvendor_standard}.

Next suppose DM has a target wealth reserve $\zeta$ for every period, where $Z_{0:T}$ are target shortfall.
Let $\mathcal A_t(W_t) = \{\mathbb E[W_t + R_t^\pi + Z_t^\pi] \geq \zeta\}$ be the acceptance set for the target wealth reserve at the end of each period $t \in [T]$, which depends on the current wealth.
The corresponding GCR-MDP is:
\begin{subequations}
\label{newsvendor_wealth}
\begin{align}
\min_{\pi \in \Pi,\, W_{1:T+1}} \quad & \mathbb E \left[ \sum_{t=0}^T Z_t^\pi \right]\\
\text{s.t.} \quad & R_t^\pi + Z_t^\pi \in \mathcal A_t(W_t),\,\forall t \in [T],\\
& W_{t+1} = \gamma(W_t + R_t^\pi),\, \forall t \in [T],\\
& Z_{0:T}^\pi \in \mathcal{C}_{\geq 0} \cap \mathcal{C}_{z}.
\end{align}
\end{subequations}
We let $\pi_{\rm WR}$ denote the optimal policy for Eq.~\eqref{newsvendor_wealth}.

In our final newsvendor GCR-MDP, DM's orders are constrained by on-hand wealth and the corresponding risk frontier directly constrains the ordering decisions.
Here we let $A_{0:T}^\pi = (A_0^\pi, \ldots, A_T^\pi)$ where $A_t^\pi$ is the order quantity in period $t \in [T]$.
We then let $\mathcal A_t(w_t) = \{A_t : c\,A_t \leq w_t + \zeta\}$ denote the wealth-dependent acceptance sets on feasible actions.
There are no disbursement decisions in this model.
The resulting GCR-MDP is:
\begin{subequations}
\label{newsvendor_order}
\begin{align}
\max_{\pi \in \Pi,\, W_{1:T+1}} \quad & \mathbb E \left[ \sum_{t=0}^T R_t^\pi \right]\\
\text{s.t.} \quad & A_t^\pi \in \mathcal A_t(W_t),\, \forall t \in [T],\\
& W_{t+1} = \gamma(W_t + R_t^\pi),\, \forall t \in [T].
\end{align}
\end{subequations}
We let $\pi_{\rm CO}$ denote the optimal policy for Eq.~\eqref{newsvendor_order}.
For $\zeta = 0$, DM must place cash orders only. For $\zeta = \infty$, we recover the classical newsvendor with unconstrained orders.

We also include two classical benchmarks for comparison.
The risk-neutral newsvendor solves:
\begin{equation}
\label{Newsvendor_classical}
    \max_{\pi \in \Pi} \varrho_{{\rm RN}}(R_{0:T}^\pi),
\end{equation}
and we let $\pi_{\rm RN}$ denote the optimal policy for Eq.~\eqref{Newsvendor_classical}.
The risk-aware newsvendor with a nested risk measure solves:
\begin{equation}
\label{Newsvendor_nested}
    \max_{\pi \in \Pi} \varrho_{{\rm N}}(R_{0:T}^\pi),
\end{equation}
and we let $\pi_{\rm N}$ denote the optimal policy for Eq.~\eqref{Newsvendor_nested}.

\subsection{Performance Analysis}

We now compare the performance of $\pi_{\rm SC}$, $\pi_{\rm WR}$, $\pi_{\rm CO}$, $\pi_{\rm RN}$, and $\pi_{\rm N}$ numerically.
We take the following MDP parameter settings: $s_{\max} = 9$, $a_{\max} = 9$, $\xi_{\max} = 9$, and $T = 10$.
We also set $w_{\min} = -540$, $w_{\max}=539$, and $z_{\max} = 1079$ for the components of the GCR.
We set the economic inputs to be $p = 3$, $c = 2$, and $h = 1$.
The random demand is i.i.d. with probability mass function given in Figure \ref{fig:prob_simple_sa}, which is truncated from a Gussian distribution with $4$ as mean and $1.2$ as variance. In Eq.~\eqref{newsvendor_standard} we pick $\beta=4$, and in Eq.~\eqref{Newsvendor_nested}, we take $\varrho_{\rm M}$ based on CVaR with risk level $\alpha = 0.4$.

\begin{figure}
    \centering
    \includegraphics[width=0.6\textwidth]{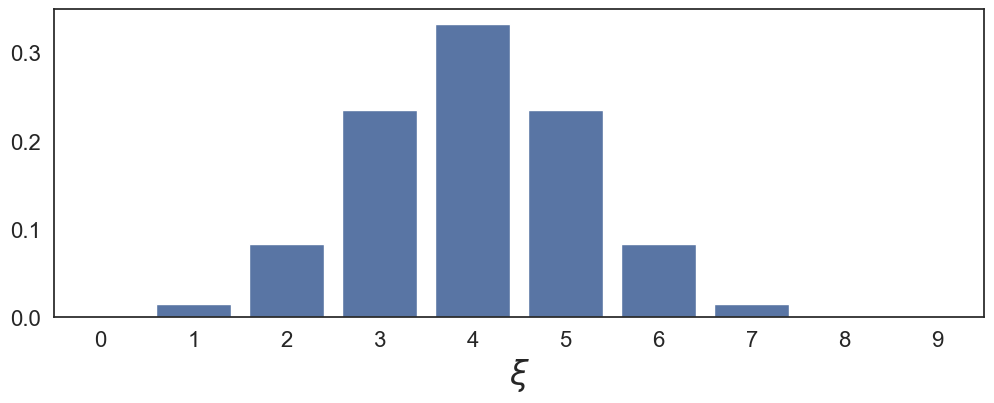}
    \caption{Probability distribution of $\xi$.}
    \label{fig:prob_simple_sa}
\end{figure}

We first study the empirical distribution of cumulative reward $\sum_{t=0}^T R_t^\pi$.
For any $\pi \in \Pi$, we simulate $\pi$ on $\omega \in \Omega$ and then obtain a sample of the corresponding cumulative reward given by $[\sum_{t=0}^T R_t^\pi](\omega)$. We obtain $50$ estimates for i.i.d. trajectories $\omega_1, \ldots, \omega_{50}$, then construct the empirical distribution of $\sum_{t=0}^T R_t^\pi$, which is shown in Figure~\ref{fig:simple_sa} and Figure~\ref{fig:reward_stream}.
In Figure~\ref{fig:simple_sa}, we note that once $\pi_{\rm WR}$ reaches a certain level of wealth, it stops placing any orders (since it just wants to achieve its target).
Figure~\ref{fig:reward_stream} gives a finer view of the distribution of the reward streams in each period.

\begin{figure}
    \centering
    \includegraphics[width=0.75\textwidth]{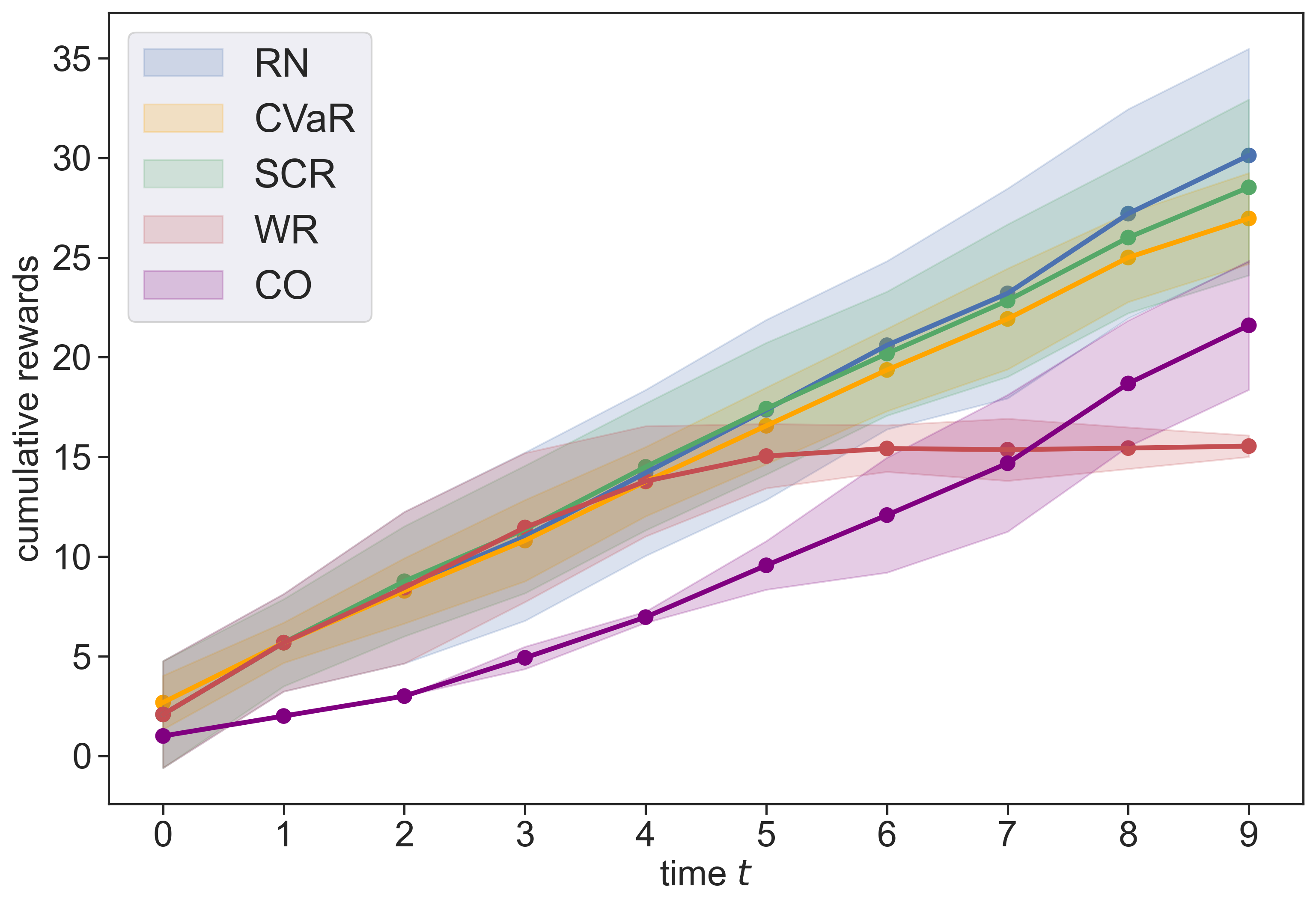}
    \caption{Cumulative payoffs for each policy (for $p = 3, c = 2, h = 1$ and $\zeta=15$).}
    \label{fig:simple_sa}
\end{figure}

\begin{figure}
    \centering
    \includegraphics[width=0.9\textwidth]{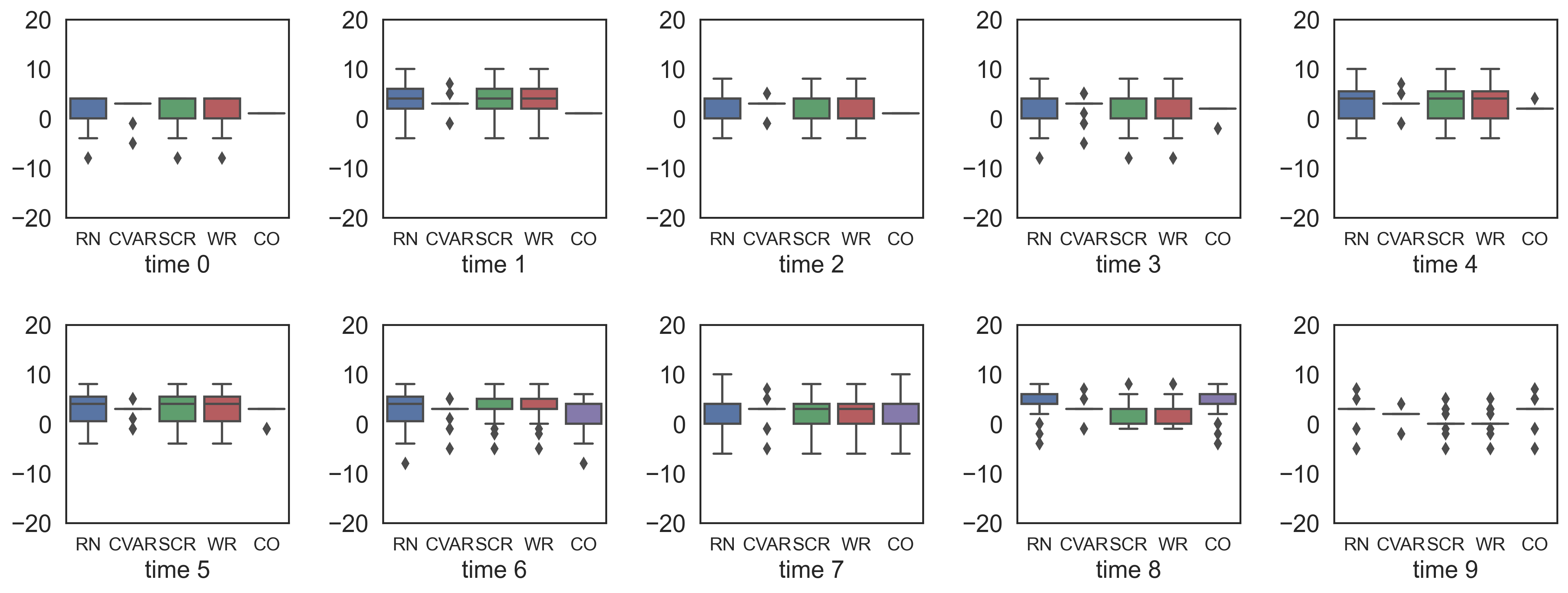}
    \caption{Payoff stream for each policy (for $p = 3, c = 2, h = 1$ and $\zeta=15$).}
    \label{fig:reward_stream}
\end{figure}

Next we consider the effects of different values of $\zeta$ on the cumulative reward for our wealth-based GCR-MDPs.
We plot the cumulative payoff $\sum_{t=0}^T R_t^\pi$ for these three models in Figure~\ref{fig:payoff_zeta}.
Figure~\ref{fig:payoff_zeta} shows that increasing $\zeta$ in Eq.~\eqref{newsvendor_wealth} results in greater cumulative reward, since $\pi_{\rm WR}$ must accumulate more wealth to meet its target. In addition, $\pi_{\rm WR}$ aims at minimizing target shortfall (with respect to the wealth reserve). This leads the sample mean of the cumulative reward to equal $\zeta$ as we observed in Figure~\ref{fig:simple_sa}. 
Increasing $\zeta$ in Eq.~\eqref{newsvendor_order} also results in greater cumulative reward for $\pi_{\rm CO}$ since it gives DM more flexibility. We also see that the standard deviation of the cumulative payoffs is increasing in $\zeta$ under $\pi_{\rm WR}$ and $\pi_{\rm CO}$. DM's flexibility with respect to $\zeta$ is time dependent. In Figure~\ref{fig:payoff_zeta}(b), $R_t^{\pi_{\rm WR}}$ becomes invariant in the tail of the trajectory. In Figure~\ref{fig:payoff_zeta}(c), $R_t^{\pi_{\rm CO}}$ has less variance in the head of the trajectory. Both windows with smaller variance become shorter when $\zeta$ increases.

\begin{figure}
\centering
\begin{subfigure}[b]{0.9\textwidth}
    \centering
    \includegraphics[width=\textwidth]{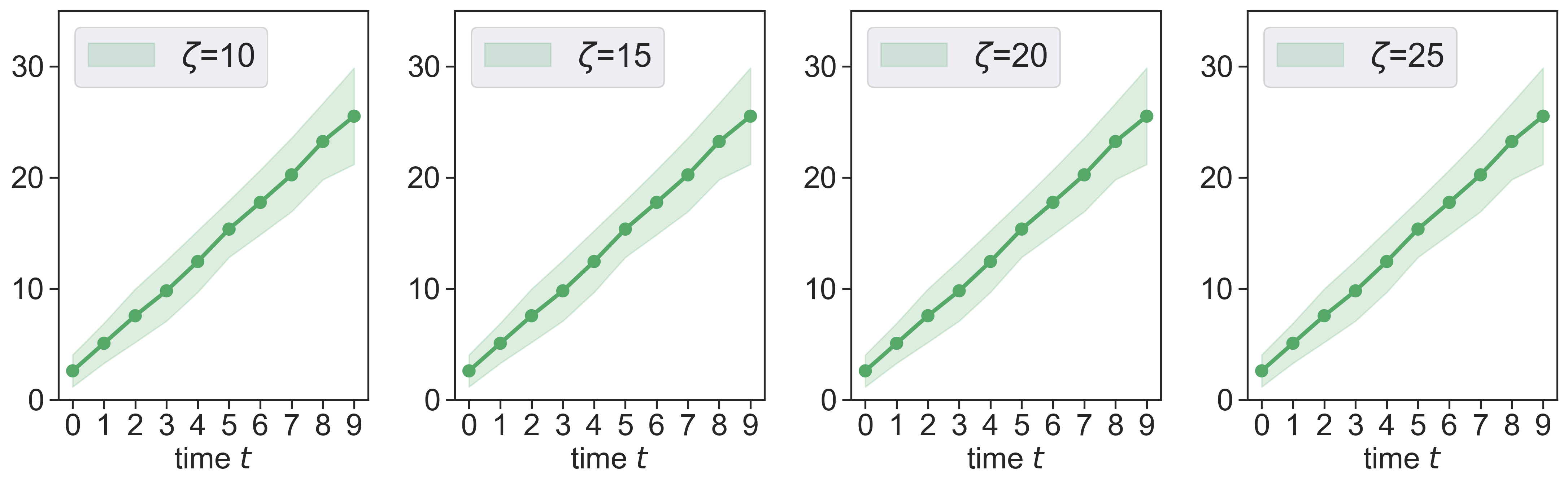}
    \caption{Standard capital requirement}
\end{subfigure}
\hfill
\begin{subfigure}[b]{0.9\textwidth}
    \centering
    \includegraphics[width=\textwidth]{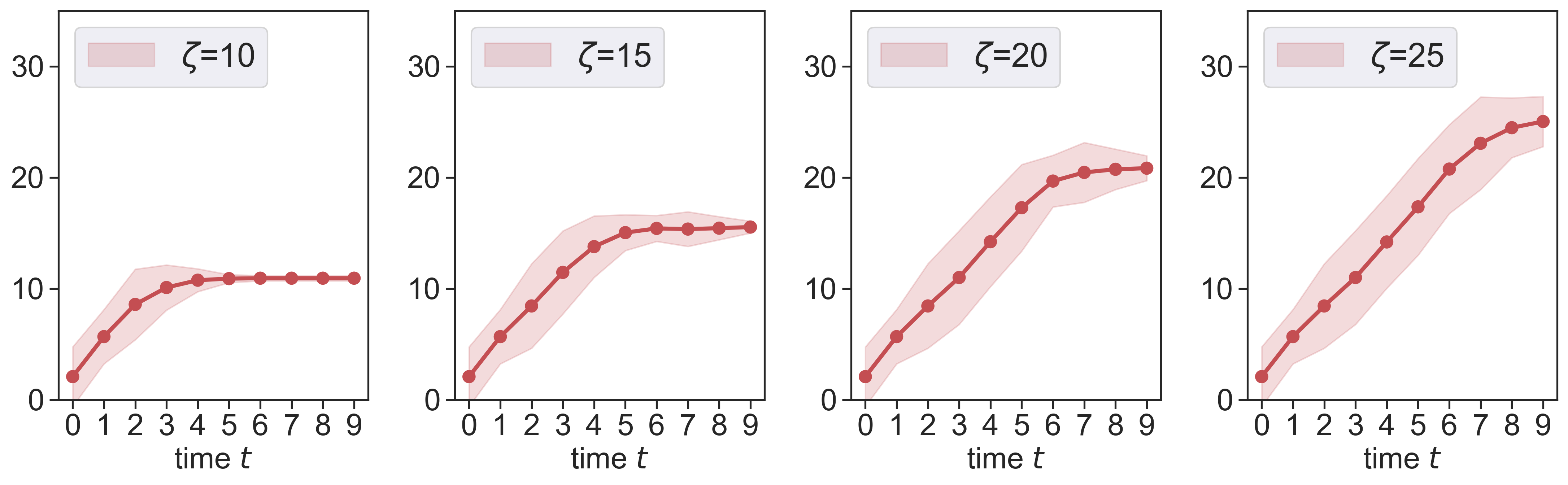}
    \caption{Wealth reserve}
\end{subfigure}
\hfill
\begin{subfigure}[b]{0.9\textwidth}
    \centering
    \includegraphics[width=\textwidth]{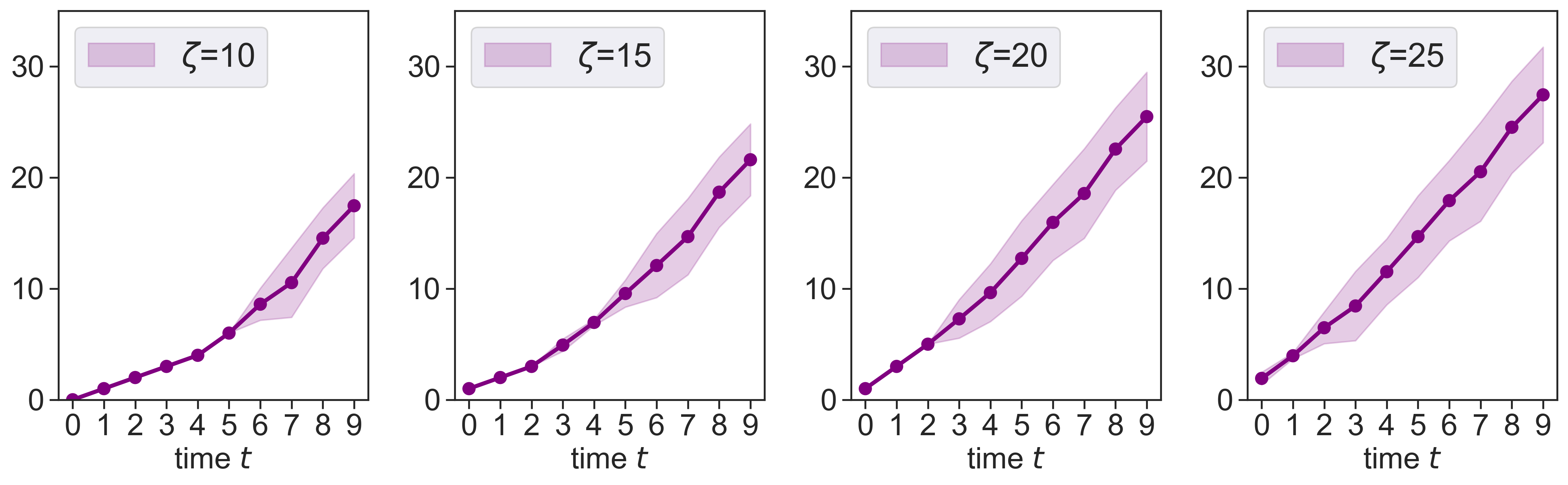}
    \caption{Cash order only}
\end{subfigure}
\caption{Cumulative payoffs under $\zeta=10,15,20,25$.}
\label{fig:payoff_zeta}
\end{figure}

We are also interested in the wealth process $W_{0:T}$ in Eq.~\eqref{newsvendor_wealth} and Eq.~\eqref{newsvendor_order}, as shown in Figure \ref{fig:wealth_trj}.
Figure~\ref{fig:wealth_trj} shows that increasing $\zeta$ in Eq.~\eqref{newsvendor_wealth} (which gives a more aggressive target) results in a steeper wealth trajectory for $\pi_{\rm WR}$.  Increasing $\zeta$ in Eq.~\eqref{newsvendor_order} (which gives DM more ordering flexibility) also results in a steeper wealth trajectory for $\pi_{\rm CO}$. 

\begin{figure}
\centering
\begin{subfigure}[b]{0.9\textwidth}
    \centering
    \includegraphics[width=\textwidth]{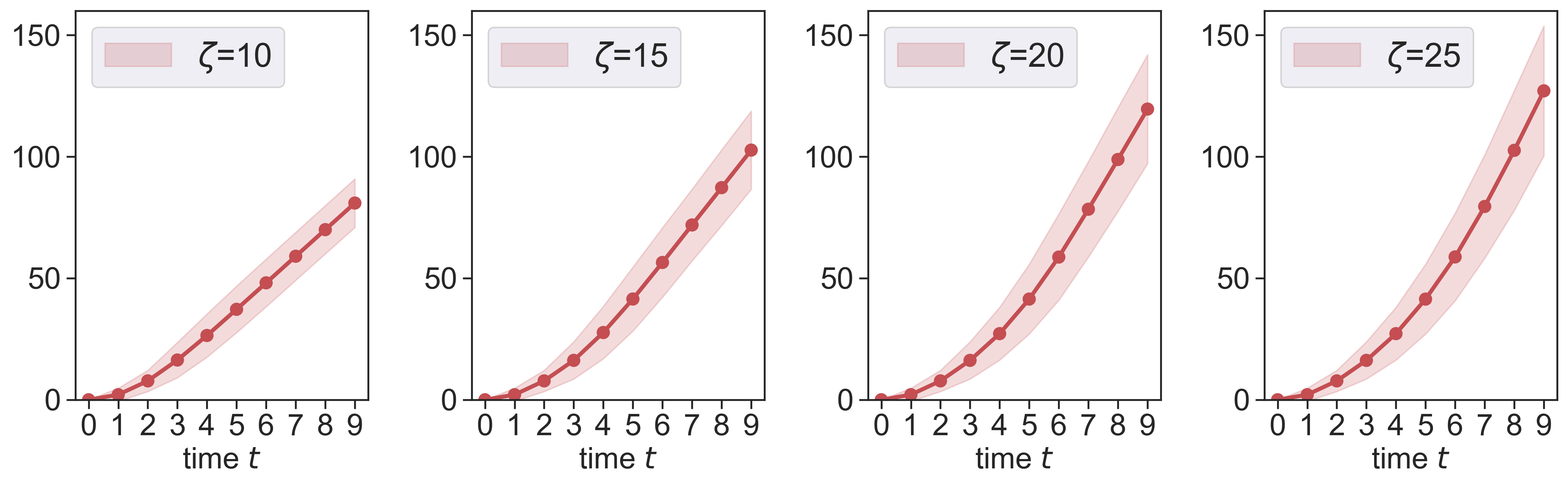}
    \caption{Wealth reserve}
\end{subfigure}
\hfill
\begin{subfigure}[b]{0.9\textwidth}
    \centering
    \includegraphics[width=\textwidth]{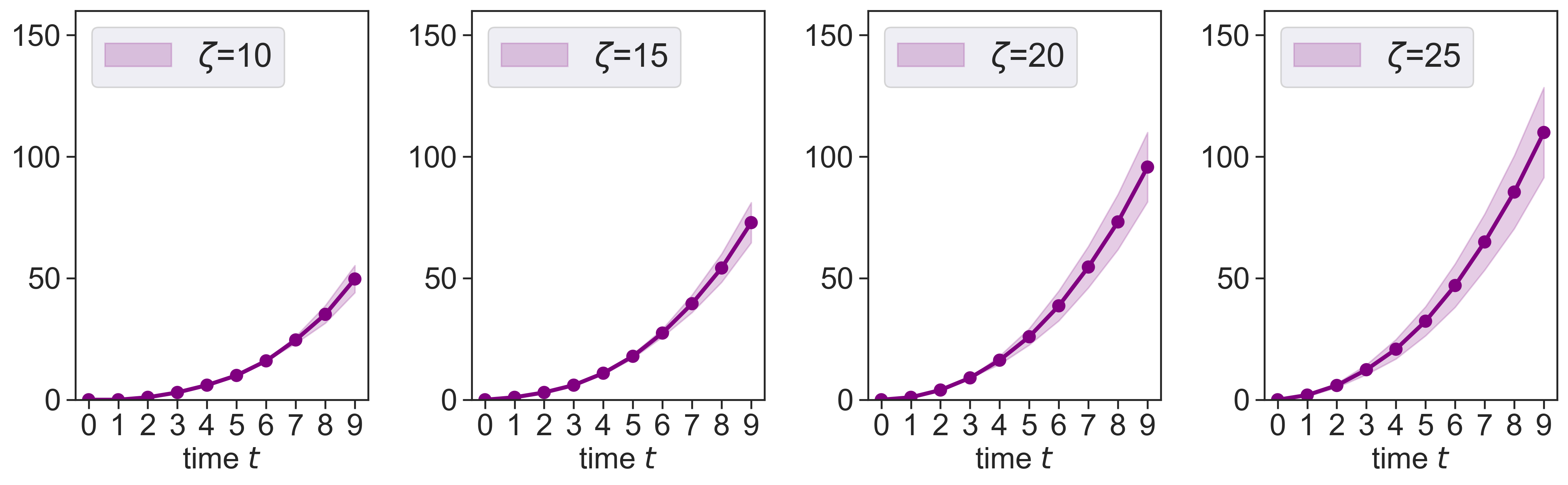}
    \caption{Cash order only}
\end{subfigure}
\caption{Wealth trajectory under $\zeta=10,15,20,25$ for Wealth reserve and Cash order only policies.}
\label{fig:wealth_trj}
\end{figure}

\section{Conclusion}\label{sec:conclusion}

In this paper, we developed the GCR and GCR-MDP framework which builds on the classical theory of CRs and extends it to MDPs.
This is a natural way to handle MDPs with constraints over the entire planning horizon (e.g., constraints on terminal wealth), and MDPs with more general objectives than the expectation (e.g., entropic risk measures, nested risk measures, worst-case expected payoff, etc.).

GCRs have several nice features.
First, the generality of the risk frontier contains many existing risk-aware MDPs as special cases and facilitates the design of new ones.
It also gives a streamlined way to elicit DM's dynamic risk preferences, which can be expressed in terms of performance targets and the available financial apparatus.
Second, it gives a general language for the temporal decomposition and time consistency of GCRs and GCR-MDPs.
In particular, we can analyze these properties in terms of inclusions with respect to the tail risk frontiers.
Finally, our framework can be used in practical computation through augmented DP based on information states.
In the majority of our examples, we only need one augmented state variable to completely express the temporal decomposition of a GCR.

Our investigation into GCR-MDPs also leads to a new perspective on risk-aware and robust MDPs.
In existing works, the optimal policy of a risk-aware MDP is often computed through an application of the standard DP recursion on an augmented state space (which minimizes the sum of current plus future costs).
We have arrived at a more general form of the DP recursion in this study.
Our cost-to-go problems are in epigraph form, where the objective is the current period premium. The dependence on the next period value function (i.e., its entire distribution) is captured by the constraints of this problem.
This result highlights the unifying structure among disparate risk-aware and robust MDP models. It also suggests possibilities for new DP operators and formulations.

In future research, we would like to investigate GCR representation results and their duality theory.
In addition, we would like to develop GCR models for infinite planning horizons, which would require a discounting mechanism within the GCR itself.
Finally, we would like to extend to GCR-MDPs with continuous state/action spaces. This would require further development to merge approximate DP algorithms (e.g., approximate value iteration) with the non-standard DP decomposition developed in this paper.

\bibliographystyle{plain}
\bibliography{References}

\appendix

\section{Additional Material for Section~\ref{sec:gcr}}

\subsection{Generalized Capital Requirements (Static Setting)}

We briefly comment on the specialization of GCRs to the static setting.
Suppose $X \in \mathcal L$ is a single stochastic payoff (e.g., terminal payoff, future portfolio value, or change in portfolio value).
Let $\mathcal A \subset \mathcal L$ be an acceptance set for $X$.
A simple CR is the minimal amount of cash that needs to be added to $X$ to make it acceptable.

\begin{defn}
\label{defn:simple_CR}
\cite[Definition 2.2]{frittelli2006risk}
The simple CR associated to $\mathcal A$ is the map
$\rho_{\mathcal A} : \mathcal L \rightarrow \Re$ defined by:
\[
\rho_{\mathcal A}(X) \triangleq \min_{z \in \Re} \{z : X + z \in \mathcal A \},\, \forall X \in \mathcal L.
\]
\end{defn}

Next, we let $z \in \Re$ be a pre-committed disbursement and $\mathcal C \subset \Re$ be the set of admissible disbursements.
Additionally, we let $\phi : \Re \rightarrow \Re$ be a disbursement cost function where DM must pay $\phi(z)$ to access the disbursement $z$. 

\begin{defn}
\label{defn:static_CR}
\cite[Definition 3.1]{frittelli2006risk}
The CR associated to $(\mathcal A, \mathcal C, \phi)$ is the map
$\rho_{\mathcal A, \mathcal C, \phi} : \mathcal L \rightarrow \Re$ defined by:
\[
\rho_{\mathcal A, \mathcal C, \phi}(X) \triangleq \min_{z \in \Re} \{\phi(z) : X + z \in \mathcal A,\, z \in \mathcal C \},\, \forall X \in \mathcal L.
\]
\end{defn}

\begin{defn}[Risk Frontier]
\label{defn:static_risk_frontier}
A set $\mathcal U \in \mathcal L \times \Re \times \Re$ is a (static) {\em risk frontier} if:

(i) $(X, z, \varphi) \in \mathcal U$ implies $(X', z, \varphi) \in \mathcal U$ for all $X' \geq X$.

(ii) $(X, z, \varphi) \in \mathcal U$ implies $(X, z, \varphi') \in \mathcal U$ for all $\varphi' \geq \varphi$.

(iii) $(X, z, \varphi) \in \mathcal U$ and $z' \geq z$ implies $(X, z', \varphi') \in \mathcal U$ for some $\varphi' \geq \varphi$.

(iv) $\mathcal U$ is closed.

\noindent
We say a triple is {\em acceptable} when $(X, z, \varphi) \in \mathcal U$.
\end{defn}

Given a payoff $X \in \mathcal L$, we define a (static) GCR to optimize over $(z, \varphi)$ to minimize the premium $\varphi$ over acceptable triples $(X, z, \varphi) \in \mathcal U$.
\begin{defn}[Generalized CR]
\label{defn:static_GCR}
Let $\mathcal U \in \mathcal L \times \Re \times \Re$ be a (static) risk frontier. The (static) generalized CR (GCR) associated to $\mathcal U$ is the map $\rho_{\mathcal U} : \mathcal L \rightarrow \Re$ defined by:
\begin{equation*}
    \rho_{\mathcal U}(X) \triangleq \min_{(z, \varphi) \in \Re \times \Re} \{\varphi : (X, z, \varphi) \in \mathcal U\},\, \forall X \in \mathcal L.
\end{equation*}
\end{defn}
\noindent
Next we give some examples of static GCRs.

\begin{example}
(i) Let $\rho(X) = - \mathbb E[X]$ be the expected payoff. Then $\rho$ is recovered by $\mathcal U = \{(X, \varphi) : \varphi \geq - \mathbb E[X]\}$. This $\mathcal U$ is also automatically convex by linearity of the expectation.

(ii) Let $u : \Re \rightarrow \Re$ be an increasing concave function and $\rho(X) = - \mathbb E[u(X)]$ be the expected utility. Then $\rho$ is recovered by $\mathcal U = \{(X, \varphi) : \varphi \geq - \mathbb E[u(X)]\}$, and this $\mathcal U$ is convex by concavity of $u$.
\end{example}

\begin{example}
(i) The simple CR $\rho_{\mathcal A}$ is recovered by $\mathcal U = \{(X, z, \varphi) : X + z \in \mathcal A,\, z \in \Re,\, \varphi \geq z\}$. When $\mathcal A$ is convex then $\mathcal U$ is also convex.

(ii) The classical CR $\rho_{\mathcal A, \mathcal C, \phi}$ is recovered by $\mathcal U = \{(X, z, \varphi) : X + z \in \mathcal A,\, z \in \mathcal C,\, \varphi \geq \phi(z)\}$.
When $\mathcal A$ and $\mathcal C$ are convex sets, and $\phi$ is a convex function, then $\mathcal U$ is convex.

(iii) Let $\nu : \mathcal L \rightarrow \mathbb R$ be convex and consider the map $\rho_{\nu, \mathcal C, \phi} : \mathcal L \rightarrow \Re$ defined as:
\[
\rho_{\nu, \mathcal C, \phi}(X) \triangleq \min_{z \in \Re} \{\phi(z) + \nu(X + z) : z \in \mathcal C \},\, \forall X \in \mathcal L.
\]
Then $\rho_{\nu, \mathcal C, \phi}$ is recovered by $\mathcal U = \{(X, z, \varphi) : z \in \mathcal C,\, \varphi \geq \phi(z) + \nu(X + z)\}$, and $\mathcal U$ is convex when both $\phi$ and $\nu$ are convex.
\end{example}

\subsection{Proof of Theorem~\ref{thm:properties}}

(i) Let $X'_{0:T} \geq X_{0:T}$, then we have
\begin{align*}
    \varrho_{\mathcal U}(X_{0:T}) = & \min_{Z_{0:T}, \Phi_{0:T}}\left\{\varphi_0 : (X_{0:T}, Z_{0:T}, \Phi_{0:T}) \in \mathcal U \right\}\\
    \geq & \min_{Z_{0:T}, \Phi_{0:T}}\left\{\varphi_0 : (X_{0:T}', Z_{0:T}, \Phi_{0:T}) \in \mathcal U \right\} = \varrho_{\mathcal U}(X_{0:T}'),
\end{align*}
since $(X_{0:T}', Z_{0:T}, \Phi_{0:T}) \in \mathcal U$ whenever $(X_{0:T}, Z_{0:T}, \Phi_{0:T}) \in \mathcal U$ by Definition~\ref{defn:risk_frontier}.

(ii) Choose $(X_{0:T}, Z_{0:T}, \Phi_{0:T}) \in \mathcal U$, $(X_{0:T}', Z_{0:T}', \Phi_{0:T}') \in \mathcal U$, and choose $\lambda \in [0, 1]$.
Then $\varrho_{\mathcal U}(X_{0:T}) \leq \varphi_0$ and $\varrho_{\mathcal U}(X_{0:T}') \leq \varphi_0'$.
By convexity of $\mathcal U$, we have
$$
\lambda (X_{0:T}, Z_{0:T}, \Phi_{0:T}) + (1 - \lambda) (X_{0:T}', Z_{0:T}', \Phi_{0:T}') \in \mathcal U,
$$
so $\varrho_{\mathcal U}(\lambda X_{0:T} + (1 - \lambda) X_{0:T}') \leq \lambda\, \varphi_0 + (1 - \lambda) \varphi_0'$. Since $(X_{0:T}, Z_{0:T}, \Phi_{0:T})$ and $(X_{0:T}', Z_{0:T}', \Phi_{0:T}')$ were arbitrary (and the upper bounds $\varrho_{\mathcal U}(X_{0:T}) \leq \varphi_0$ and $\varrho_{\mathcal U}(X_{0:T}') \leq \varphi_0'$ were arbitrary), we conclude that $\varrho_{\mathcal U}(\lambda X_{0:T} + (1 - \lambda) X_{0:T}') \leq \lambda\, \varrho_{\mathcal U}(X_{0:T}) + (1 - \lambda) \varrho_{\mathcal U}(X_{0:T}')$.

(iii) Choose $X_{0:T} \in \mathcal L_{1:T+1}$ and $\alpha > 0$, and define $\tilde X_{0:T} = \alpha\, X_{0:T}$. Then, we have
\begin{align*}
    \varrho_{\mathcal U}( \tilde X_{0:T} ) = & \min_{\tilde Z_{0:T}, \tilde \Phi_{0:T}} \{ \tilde\varphi_0 : (\alpha\, X_{0:T}, \tilde Z_{0:T}, \tilde \Phi_{0:T}) \in \mathcal U\}\\
    = & \min_{\tilde Z_{0:T}, \tilde \Phi_{0:T}} \{ \tilde \varphi_0 : \alpha\,(X_{0:T}, \tilde Z_{0:T} / \alpha, \tilde \Phi_{0:T} / \alpha) \in \mathcal U\}\\
    = & \min_{\alpha Z_{0:T}, \alpha \Phi_{0:T}} \{ \alpha\, \varphi_0 : \alpha (X_{0:T}, Z_{0:T}, \Phi_{0:T}) \in \mathcal U\}\\
    = & \alpha\, \min_{Z_{0:T}, \Phi_{0:T}} \{ \varphi_0 : (X_{0:T}, Z_{0:T}, \Phi_{0:T}) \in \mathcal U\}\\
    = & \alpha\, \varrho_{\mathcal U} (X_{0:T}),
\end{align*}
where the third equality uses the change of variables $\alpha\, Z_{0:T} = \tilde Z_{0:T}$ and $\alpha\, \Phi_{0:T} = \tilde \Phi_{0:T}$, and the fourth equality uses the fact that $\mathcal U$ is a cone.

\section{Additional Material for Section~\ref{sec:tempCR}}
\subsection{Proof of Theorem~\ref{thm:time_consistent}}

First we establish the following property of DM's preferences induced by $\varrho_{\mathcal U}$ and $\{\varrho_{\mathcal U, t:T}\}_{t \in [T]}$. For payoff streams $X_{0:T},\, X_{0:T}' \in \mathcal L_{1:T+1}$, we have $\varrho_{\mathcal U}(X_{0:T}') \leq \varrho_{\mathcal U}(X_{0:T})$ if and only if for any $(X_{0:T}, Z_{0:T}, \Phi_{0:T}) \in \mathcal U$, there exists some $(X_{0:T}', Z_{0:T}', \Phi_{0:T}') \in \mathcal U$ with $\varphi_0' \leq \varphi_0$. This claim follows by definition of $\varrho_{\mathcal U}$, since otherwise we must have $\varrho_{\mathcal U}(X_{0:T}') > \varrho_{\mathcal U}(X_{0:T})$.

Similarly, for $X_{t:T}, X_{t:T}' \in \mathcal L_{t+1:T+1}$ we have $[\varrho_{\mathcal U, t:T}(X_{t:T}')](h_t) \leq [\varrho_{\mathcal U, t:T}(X_{t:T})](h_t)$ if and only if for any $[(X_{t:T}, Z_{t:T}, \Phi_{t:T}) \vert h_t] \in \mathcal U_{t:T}(h_t)$, there exists some $[(X_{t:T}', Z_{t:T}', \Phi_{t:T}') \vert h_t] \in \mathcal U_{t:T}(h_t)$ such that $\Phi_t'(h_t) \leq \Phi_t(h_t)$; this claim follows from definition of $\varrho_{\mathcal U, t:T}$.

Now choose $0 \leq t_1 < t_2 \leq T$ and $X_{t_1:T},\, X_{t_1:T}' \in \mathcal L_{t_1:T}$ with $X_{t} = X'_{t}$ for all $t \in [t_1, t_2 - 1]$ and $\varrho_{\mathcal U, t_2: T}(X_{t_2:T}') \leq \varrho_{\mathcal U, t_2: T}(X_{t_2:T})$.
By the preceding discussion, for any $[(X_{t_2:T}, Z_{t_2:T}, \Phi_{t_2:T}) \vert h_{t_2}] \in \mathcal U_{t_2:T}(h_{t_2})$, there exists $[(X_{t_2:T}', Z_{t_2:T}', \Phi_{t_2:T}') \vert h_{t_2}] \in \mathcal U_{t_2:T}(h_{t_2})$ with $\Phi_{t_2}'(h_{t_2}) \leq \Phi_{t_2}(h_{t_2})$ for all $h_{t_2} \in \mathbb{H}_{t_2}$.

For any $[(X_{t_1:T}, Z_{t_1:T}, \Phi_{t_1:T}) \vert h_{t_1}] \in \mathcal U_{t_1:T}(h_{t_1})$, we have $\varrho_{\mathcal U, t_1: T}(X_{t_1:T}) \leq \Phi_{t_1}(h_{t_1})$.
Let $(\tilde X_{k}, \tilde Z_{k}, \tilde \Phi_{k}) = (X_{k}, Z_{k}, \Phi_{k})$ for all $k \in [t_1, t_2 - 1]$ and $(\tilde X_{k}, \tilde Z_{k}, \tilde \Phi_{k}) = (X_{k}', Z_{k}', \Phi_{k}')$ for all $k \in [t_2, T]$ be the construction in Assumption~\ref{assu:risk_frontier_consistency}.
Then, by Assumption~\ref{assu:risk_frontier_consistency} we have $[(\tilde X_{t_1:T}, \tilde Z_{t_1:T}, \tilde \Phi_{t_1:T}) \vert h_{t_1}] \in \mathcal U_{t_1:T}(h_{t_1})$ and so $\varrho_{\mathcal U, t_1: T}(\tilde X_{t_1:T}) \leq \tilde \Phi_{t_1}(h_{t_1})$.
Since $\tilde \Phi_{t_1}(h_{t_1}) = \Phi_{t_1}(h_{t_1})$, we have $[\varrho_{\mathcal U, t_1: T}(\tilde X_{t_1:T})](h_{t_1}) \leq \Phi_{t_1}(h_{t_1})$.
The original $[(X_{t_1:T}, Z_{t_1:T}, \Phi_{t_1:T}) \vert h_{t_1}] \in \mathcal U_{t_1:T}(h_{t_1})$ and $h_{t_1}$ was arbitrary, so we conclude $\varrho_{\mathcal U, t_1: T}(\tilde X_{t_1:T}) \leq \varrho_{\mathcal U, t_1: T}(X_{t_1:T})$.

\subsection{Proof of Theorem~\ref{thm:recursive}}
\label{app:thm:recursive}

\textbf{Part (i).} We first introduce some terminology from lattice theory.
Let $\mathbb L$ be a normed vector space of functions $\{w:\Omega\to\Re\}$.
We use the notation $\sqmin \mathcal W \triangleq \sqmin_{w\in\mathcal W} w$ for a set $\mathcal W \subset \mathbb L$.

\begin{defn}
(i) Let $\preceq$ be the partial order on $\mathbb L$ where $w\preceq w'$ if and only if $w(\omega)\leq w'(\omega)$ for almost all $\omega\in\Omega$.

(ii) For $\mathcal W\subset\mathbb L$, the lattice supremum is $\sqsup \mathcal W :\omega\mapsto \sup \{w(\omega):w\in\mathcal W\}$.

(iii) For $\mathcal W\subset\mathbb L$, the lattice infimum is $\sqinf \mathcal W :\omega\mapsto \inf \{w(\omega):w\in\mathcal W\}$.

(iv) For $\mathcal W\subset\mathbb L$, the lattice minimum is $\sqmin \mathcal W :\omega\mapsto \min \{w(\omega):w\in\mathcal W\}$ (when the minimum is attained for all $\omega \in \Omega$).

(v) $\mathcal W\subset\mathbb L$ is a complete lattice if $\sqsup W \in \mathcal W$ and $\sqinf W \in \mathcal W$ for all $W\subset\mathcal W$.
\end{defn}

\begin{lem}\label{lem:com}
\cite[Lemma 2.6.1, p. 115]{meyer1991banach}
Let $B \subset \Re$ be closed and bounded and let $\mathcal W\triangleq\{w:\Omega\to B\}\subset\mathbb L$, then $\mathcal W$ is a complete lattice.
\end{lem}

\begin{lem}\label{lem:com2}
Let $\mathcal W_1,\, \mathcal W_2\subset\mathbb L$ be two complete lattices and define $w_i^* \triangleq \sqmin\mathcal W_i$. If $\mathcal W_1\subseteq\mathcal W_2$, then $w_2^*\preceq w_1^*$.
\end{lem}
\begin{proof}
Since $w_1^*\in\mathcal W_1\subset\mathcal W_2$, we conclude that $w_2^*\preceq w_1^*$.
\end{proof}

Next we establish our key result on the lattice structure of the function spaces that appear in our temporal decomposition.

\begin{lem}\label{lem:lattice-decomposition}
Let $\mathbb L$ be a vector space equipped with a partial order $\preceq$, and let $\mathcal W_2\subset \mathbb L$ be a complete lattice. Let $\mathcal W_1:\mathcal W_2\rightrightarrows \mathbb L$ be a correspondence such that: 

(i) $\mathcal W_1(\alpha_2)$ is a complete sublattice of $\mathbb L$ for all $\alpha_2\in\mathcal W_2$.

(ii) $\alpha_2\preceq \alpha_2'$ implies $\mathcal W_1(\alpha_2')\subseteq \mathcal W_1(\alpha_2)$ for all $\alpha_2,\alpha_2'\in\mathcal W_2$.
\noindent
Define $\mathcal W \subset \mathbb L\times\mathbb L$ via:
\[\mathcal W \triangleq\{(\alpha_1,\alpha_2)\in \mathbb L\times\mathbb L:\alpha_1\in\mathcal W_1(\alpha_2), \alpha_2\in\mathcal W_2\}.\]
Let $\alpha_2^* = \sqmin \mathcal W_2$, then
\begin{align}\label{eqn:alphamin}
    \sqmin\{\alpha_1:(\alpha_1,\alpha_2)\in\mathcal W\} = \sqmin  \mathcal W_1(\alpha_2^*).
\end{align}
\end{lem}
\begin{proof}
Define $\alpha_1^* \in \mathcal W_1$ and $\beta^*:\mathcal W_2\to \mathcal W_1$ via: 
\[\alpha_1^* \triangleq \sqmin \mathcal W_1(\alpha_2^*),\qquad \beta^*(\alpha_2) \triangleq \sqmin \mathcal W_1(\alpha_2).\]
Let $(\bar \alpha_1, \bar \alpha_2)$ be a minimizer on the left side of Eq.~\eqref{eqn:alphamin}. We clearly have $\bar \alpha_1\preceq \alpha_1^*$. We next establish the reverse inequality $\alpha_1^*\preceq \bar\alpha_1$. By definition of lattice minimum, for any $\alpha_2\in\mathcal W_2$, we have $\alpha_2\succeq \alpha_2^*$; in particular, we have $\bar\alpha_2\succeq \alpha_2^*$. Due to the monotonicity hypothesis on $\mathcal W_1$, we have $\mathcal W_1(\bar\alpha_2)\subseteq\mathcal W_1(\alpha_2^*)$ and so Lemma \ref{lem:com2} implies $\alpha_1^*\preceq \bar\alpha_1$, which yields the desired result.
\end{proof}

We can now derive the recursive decomposition of $\{\varrho_{\mathcal U, t:T}\}_{t \in [T]}$ by induction starting from period $t = T$ to $t = 0$. Recall $\Phi_{T+1}$ is fixed, and define
\begin{equation*}
\mathcal W_T(\Phi_{T+1}) \triangleq \bigcup_{Z_T\in\mathcal L_{T}} \Big\{\Phi_T: [(X_T, Z_T, \Phi_T) \vert h_T] \in \mathcal U_T(h_T, \Phi_{T+1}),\, \forall h_T\in\mathbb{H}_T\Big\},
\end{equation*}
and
\begin{equation*}
\mathcal W_t(\Phi_{t+1}) \triangleq \bigcup_{Z_t\in\mathcal L_{t}} \Big\{\Phi_t: [(X_t, Z_t, \Phi_t) \vert h_t] \in \mathcal U_t(h_t,\Phi_{t+1}),\, \forall h_t\in\mathbb{H}_t\Big\},\, \forall t \in [T-1].
\end{equation*}

\begin{rem}\label{rem:Wt}
    For any $t \in [T]$, if $\Phi_{t+1}\leq \Phi_{t+1}'$ then $\mathcal U_t(h_t, \Phi_{t+1}') \subset \mathcal U_t(h_t, \Phi_{t+1})$ by Assumption~\ref{assu:one-step}(iii). Consequently, we have $\mathcal W_t(\Phi_{t+1}')\subset \mathcal W_t(\Phi_{t+1})$.
\end{rem} 

For the terminal period, let $\sqmin \mathcal W_T(\Phi_{T+1}) (h_T)$ denote the function $\sqmin \mathcal W_T(\Phi_{T+1})$ evaluated at $h_T \in \mathbb{H}_T$.
Then we have:
\begin{align*}
    & \sqmin \mathcal W_T(\Phi_{T+1}) (h_T)\\
    = & \min \{\Phi_T(h_T) : \exists (Z_T, \Phi_T) \in \mathcal L_{T} \times \mathcal L_{T} \text{ s.t. } [(X_T, Z_T, \Phi_T) \vert h_T'] \in \mathcal U_T(h_T', \Phi_{T+1}),\, \forall h_T' \in \mathbb{H}_T \}\\
    = & \min_{z_T, \varphi_T} \{\varphi_T: ([X_T \vert h_T], z_T, \varphi_T) \in \mathcal U_T(h_T, \Phi_{T+1})\}.
\end{align*}

Thus, at period $t = T$, we have $\sqmin \mathcal W_T(\Phi_{T+1}) = \varrho_{\mathcal U, T:T}(X_T)$. At period $t = T-1$, we have
\begin{align*}
[\varrho_{\mathcal U, T-1:T}(X_{T-1:T})](h_{T-1}) = & \sqmin \{\Phi_{T-1}(h_{T-1}): \Phi_{T-1}\in \mathcal W_{T-1}(\Phi_T), \Phi_T\in\mathcal W_T(\Phi_{T+1})\}. 
\end{align*}
We next apply Lemma~\ref{lem:lattice-decomposition} and Remark \ref{rem:Wt} to conclude that 
\begin{align*}
\varrho_{\mathcal U, T-1:T}(X_{T-1:T}) = & \sqmin \mathcal W_{T-1}(\varrho_{\mathcal U, T:T}(X_T)).
\end{align*}

For the induction step, suppose the assertion is true for period $t+1$ so we have
\begin{align*}
\varrho_{\mathcal U, t+1:T}(X_{t+1:T}) = & \sqmin\{\Phi_{t+1}: \Phi_k\in\mathcal W_k(\Phi_{k+1}),\, \forall k \in [t+1, T]\}\\
= & \sqmin \mathcal W_{t+1}( \varrho_{\mathcal U, t+2, T}(X_{t+2:T})). 
\end{align*}
For period $t$, we follow the same argument as above and apply Lemma~\ref{lem:lattice-decomposition} and Remark \ref{rem:Wt} to conclude that $\varrho_{\mathcal U, t:T}(X_{t:T})  = \sqmin \mathcal W_t(\varrho_{\mathcal U, t+1, T}(X_{t+1:T}))$.

\textbf{Part (ii).} Choose $0 \leq t_1 < t_2 \leq T$ and $X_{t_1:T}, X_{t_1:T}' \in \mathcal L_{t_1:T}$ with $X_k = X_k'$ for $k \in [t_1, t_2 - 1]$ and $\varrho_{\mathcal U, t_2: T}(X_{t_2:T}') \leq \varrho_{\mathcal U, t_2: T}(X_{t_2:T})$.
Starting with period $t = t_2 - 1$, let $[(X_{t:T}, Z_{t:T}, \Phi_{t:T}) \vert h_t] \in \mathcal U_{t:T}(h_{t})$ so that
\[
\varrho_{\mathcal U, t:T}(X_{t:T}) \leq \min_{z_{t}, \varphi_{t}} \{ \varphi_t : ([X_{t} \vert h_t], z_{t}, \varphi_{t}) \in \mathcal U_{t}(h_t, \Phi_{t_2}) \},
\]
by monotonicity of $\mathcal U_{t}(h_t, \cdot)$.
Since $\varrho_{\mathcal U, t_2: T}(X_{t_2:T}') \leq \varrho_{\mathcal U, t_2: T}(X_{t_2:T})$, there exists $[(X_{t_2:T}, Z_{t_2:T}, \Phi_{t_2:T}) \vert h_{t_2}] \in \mathcal U_{t_2:T}(h_{t_2})$ and $[(X_{t_2:T}', Z_{t_2:T}', \Phi_{t_2:T}') \vert h_{t_2}] \in \mathcal U_{t_2:T}(h_{t_2})$ with $\Phi_{t_2}'(h_{t_2}) \leq \Phi_{t_2}(h_{t_2})$.
Again by monotonicity of $\mathcal U_{t}(h_t, \cdot)$, we obtain
\begin{align*}
    [\varrho_{\mathcal U, t:T}(X_{t:T}')](h_t) \leq & \min_{z_{t}, \varphi_{t}} \{ \varphi_t : ([X_{t}' \vert h_t], z_{t}, \varphi_{t}) \in \mathcal U_{t}(h_t, \Phi_{t_2}') \}\\
    \leq & \min_{z_{t}, \varphi_{t}} \{ \varphi_t : ([X_{t} \vert h_t], z_{t}, \varphi_{t}) \in \mathcal U_{t}(h_t, \Phi_{t_2}) \} = [\varrho_{\mathcal U, t:T}(X_{t:T})](h_t),
\end{align*}
using $X_t = X_t'$, so $[\varrho_{\mathcal U, t_2 - 1: T}(X_{t_2 - 1:T}')](h_{t_2-1}) \leq [\varrho_{\mathcal U, t_2 - 1: T}(X_{t_2 - 1:T})](h_{t_2-1})$. The result for earlier periods $t \in [t_1, t_2 - 2]$ follows by induction using the same reasoning.

\subsection{Proof of Theorem~\ref{thm:Information_sufficiency}}

\textbf{Part (i).} The terminal premium $\Phi_{T+1} = \bar{\Phi}_{T+1}$ only depends on the information state by Definition~\ref{defn:information}(ii).
For period $T$, we then have 
$$
[(X_T, Z_T, \Phi_T) \vert h_T] \in \mathcal U_{T}(h_T, \bar{\Phi}_{T+1}),\, \forall h_T \in \mathbb{H}_T,
$$
where $[X_T \vert h_T] = [X_T \vert s_T]$. We then construct $(\bar Z_T, \bar{\Phi}_T)$ as follows.
For all $y_T \in \mathbb Y_T$, take
$$
(\bar Z_T(y_T), \bar{\Phi}_{T}(y_T)) \in \arg \min \{\Phi_T(h_T) : [(X_T, Z_T, \Phi_T) \vert h_T] \in \mathcal U_{T}(h_T, \bar{\Phi}_{T+1}), \sigma_T(h_T) = y_T\}.
$$
We then obtain $\bar Z_T \in \bar{\mathcal L}_{T}$ and $\bar{\Phi}_T \in \bar{\mathcal L}_T$ with $\bar{\Phi}_T \leq \Phi_T$.

For the induction step for period $t$, suppose $\bar{\Phi}_{t+1} \leq \Phi_{t+1}$.
Then we have
$$
[(X_t, Z_t, \Phi_t) \vert h_t] \in \mathcal U_{t}(h_t, \Phi_{t+1}) \subset \mathcal U_{t}(h_t, \bar{\Phi}_{t+1}),\, \forall h_t \in \mathbb{H}_t,
$$
by monotonicity of $\mathcal U_{t}(h_t, \cdot)$ under Assumption~\ref{assu:one-step}(iii).
Next we construct $(\bar Z_t, \bar \varphi_t)$ as follows. For all $y_t \in \mathbb Y_t$, we take
$$
(\bar Z_t(y_t), \bar{\Phi}_{t}(y_t)) \in \arg \min \{\Phi_t(h_t) : [(X_t, Z_t, \Phi_t) \vert h_t] \in \mathcal U_{t}(h_t, \bar{\Phi}_{t+1}),\, \sigma_t(h_t) = y_t\}.
$$
We then obtain $\bar Z_t \in \bar{\mathcal L}_{t}$ and $\bar{\Phi}_t \in \bar{\mathcal L}_t$ with $\bar{\Phi}_t \leq \Phi_t$.
To conclude, we then have $(X_{0:T}, \bar Z_{0:T}, \bar{\Phi}_{0:T}) \in \bar{\mathcal U}$ with $\bar{\Phi}_{0:T} \leq \Phi_{0:T}$.

\textbf{Part (ii).} Since $\bar{\mathcal U} \subset \mathcal U$, we automatically have $\varrho_{\bar{\mathcal U}}(X_{0:T}) \geq \varrho_{\mathcal U}(X_{0:T})$.
For any $(X_{0:T}, Z_{0:T}, \Phi_{0:T}) \in \mathcal U$, we have $\varrho_{\mathcal U}(X_{0:T}) \leq \varphi_0$.
By the previous part, there exists $(X_{0:T}, \bar Z_{0:T}, \bar{\Phi}_{0:T}) \in \bar{\mathcal U}$ such that $\bar \varphi_0 \leq \varphi_0$ and $\varrho_{\bar{\mathcal U}}(X_{0:T}) \leq \bar \varphi_0$. Since $(X_{0:T}, Z_{0:T}, \Phi_{0:T}) \in \mathcal U$ was arbitrary, we conclude that $\varrho_{\bar{\mathcal U}}(X_{0:T}) \leq \varrho_{\mathcal U}(X_{0:T})$ which completes the argument.

\subsection{Proof of Theorem~\ref{thm:Information_recursive}}

To establish this result, we follow the same line of reasoning as the proof of Theorem \ref{thm:recursive} (see Appendix~\ref{app:thm:recursive}). By Definition~\ref{defn:information}(iii), for $\bar{\Phi}_{t+1} \in \bar{\mathcal L}_{t+1}$ we have $\mathcal U_t(h_t, \bar{\Phi}_{t+1}) = \bar{\mathcal U}_t(y_t, \bar{\Phi}_{t+1})$ for all $h_t, h_t' \in \mathbb{H}_t$ such that $\sigma_t(h_t) = y_t$.

Define
\begin{equation*}
\bar{\mathcal W}_T(\bar{\Phi}_{T+1}) \triangleq \bigcup_{\bar{Z}_T\in \bar{\mathcal L}_{T}} \Big\{\bar{\Phi}_T: [(X_T, \bar{Z}_T, \bar{\Phi}_T) \vert y_T] \in \bar{\mathcal U}_T(y_T, \Phi_{T+1}),\, \forall y_T\in\mathbb Y_T\Big\},
\end{equation*}
and let $\sqmin \bar{\mathcal W}_T(\Phi_{T+1}) (y_T)$ denote the function $\sqmin \bar{\mathcal W}_T(\Phi_{T+1})$ evaluated at $y_T$.
Then for any $h_T$ with $\sigma_T(h_T) = y_T$, we have
\begin{align*}
    & \varrho_{\bar{\mathcal U}, T:T}(X_T)(y_T)\\
    = & \sqmin \mathcal W_T(\Phi_{T+1}) (y_T) \\
    = & \min 
    \{\bar \Phi_T(y_T): \exists (\bar Z_T, \bar \Phi_T) \in \bar{\mathcal L}_{T} \times \bar{\mathcal L}_{T} \text{ s.t. } [(X_T, \bar Z_T, \bar \Phi_T) \vert y_T'] \in \bar{\mathcal U}_T(y_T', \bar{\Phi}_{T+1}),\, \forall y_T \in \mathbb{Y}_T\}\\
    = & \min_{z_T, \varphi_T} \{\varphi_T : ([X_T \vert s_T], z_T, \varphi_T) \in \bar{\mathcal U}_T(y_T, \bar{\Phi}_{T+1})\}.
\end{align*}
The argument for periods $t=T-1,\ldots,1$ is by induction.

\section{Additional Material for Section~\ref{sec:problem}}

\subsection{Proof of Theorem~\ref{thm:History_DP_decomposition}}
For period $T$,
define the correspondence $\mathcal V_T:\mathcal L_{T+1}\rightrightarrows\mathcal L_T$ via:

\begin{align*}
    \mathcal V_T(\Phi_{T+1}) = \bigcup_{(X_T, Z_T) \in \mathcal L_{T+1} \times \mathcal L_{T}}\Big\{\Phi_T: [(X_T, Z_T,\Phi_T) \vert h_T] \in\mathcal Q_T(h_T, \Phi_{T+1}),\, \forall h_T\in\mathbb{H}_T\Big\},
\end{align*}
where $\Phi_{T+1} \equiv V_{T+1}$ is fixed. For periods $t \in [T - 1]$,
define the correspondence $\mathcal V_t:\mathcal L_{t+1}\rightrightarrows\mathcal L_t$ via
\begin{align*}
    \mathcal V_t( \Phi_{t+1}) = \bigcup_{(X_t, Z_t) \in \mathcal L_{t+1} \times \mathcal L_{t}}\Big\{\Phi_t: [(X_t, Z_t,\Phi_t) \vert h_t] \in \mathcal Q_t(h_t,\Phi_{t+1}),\, \forall h_t\in\mathbb{H}_t\Big\}.
\end{align*}
By Lemma~\ref{lem:com}, $\mathcal V_t( \Phi_{t+1})$ is a complete lattice. By construction and Assumption~\ref{assu:one-step}(iii), if $\Phi_{t+1}\leq \Phi_{t+1}'$, then $\mathcal V_t( \Phi_{t+1}')\subset \mathcal V_t( \Phi_{t+1})$. 
Furthermore, the decision sub-problem starting at period $t \in [T]$ is equivalent to:
\begin{align*}
    V_t = \sqmin_{\Phi_{t:T}} \quad & \Phi_t\\
    \text{s.t.} \quad & \Phi_k \in \mathcal V_k( \Phi_{k+1}),\, \forall k \in [t, T].
\end{align*}
We then apply Lemma~\ref{lem:lattice-decomposition} recursively starting in period $T-1$ to conclude that 
\begin{align*}
    V_t = & \sqmin \mathcal V_t(V_{t+1}),\, \forall t \in [T],
\end{align*}
starting with $V_{T+1}$ as input.
The lattice minimum then yields the recursion in Eq.~\eqref{eqn:Vthistory}.

\subsection{Proof of Theorem~\ref{thm:History_policy}}

We establish optimality and time consistency of $\pi^*$ by induction. Starting in period $T$, $\pi_T^*$ is by definition the optimal period $T$ policy. Consequently,
$$
(R_T(s_t,\pi^{*a}_T(h_T)), \pi^{*z}_T(h_T), V_T(h_T)) \in \mathcal U_T(h_T, V_{T+1}),\, \forall h_T\in\mathbb{H}_T.
$$
This inclusion implies that $\pi^*_T$ is the optimal policy for $\mathcal U_{T}(h_T,V_{T+1})$.

Now suppose that $\pi_{t+1:T}^*$ is the optimal policy for periods $[t+1,T]$. Then, the minimum achievable premium on history $h_{t+1} \in \mathbb{H}_{t+1}$ is $V_{t+1}(h_{t+1})$ and it is achieved using $\pi_{t+1:T}^*$. By Assumption~\ref{assu:one-step}(iii), we have that $\mathcal U_t(h_t,\Phi_{t+1}) \subset \mathcal U_t(h_t,V_{t+1})$ holds for all feasible $\Phi_{t+1} \geq V_{t+1}$. By definition, $\pi^*_t(h_t)$ minimizes the premium in period $t$ and achieves the objective value $V_t(h_t)$ for all $h_t\in\mathbb{H}_t$. Consequently, $\pi^*_{t:T}$ is an optimal policy and minimizes the premium for the risk frontier $\mathcal U_{t:T}(h_t)$ due to Lemma~\ref{lem:lattice-decomposition}. This completes the induction step. We conclude that $\pi^*$ is both an optimal policy and time-consistent.

\subsection{Proof of Theorem~\ref{thm:Information_DP_decomposition}}
This result follows by Theorem \ref{thm:History_DP_decomposition}. From Definition~\ref{defn:information}(ii), we have $\Phi_{T+1} = \bar{\Phi}_{T+1}$.
From Definition~\ref{defn:information}(iii), we see that $\mathcal Q_t(h_t,\bar{\Phi}_{t+1}) = \bar{\mathcal Q}_t(y_t,\bar\Phi_{t+1})$ for all $h_t \in \mathbb{H}_t$ such that $\sigma_t(h_t) = y_t$, for all $t\in[T]$. Consequently, for period $T$ we have
\begin{align*}
    \mathcal V_T(\bar{\Phi}_{T+1}) & = \bigcup_{(X_T, Z_T) \in \mathcal L_{T+1} \times \mathcal L_{T} }\Big\{\Phi_T: [(X_T, Z_T,\Phi_T) \vert h_T] \in\mathcal Q_T(h_T, \Phi_{T+1}),\, \forall h_T\in\mathbb{H}_T\Big\}\\
    & = \bigcup_{(X_T, Z_T) \in \bar{\mathcal L}_{T+1} \times \bar{\mathcal L}_{T}}\Big\{\bar{\Phi}_T: [(X_T, \bar Z_T,\bar{\Phi}_T) \vert y_T]\in\bar{\mathcal Q}_T(y_T, \bar{\Phi}_{T+1}),\, \forall y_T\in\mathbb Y_T\Big\}.
\end{align*}
From Theorem \ref{thm:History_DP_decomposition}, $\bar V_T = \sqmin \mathcal V_T(\bar V_{T+1}):\mathbb Y_T\to\Re$. The desired result then follows by induction. 

\subsection{Proof of Theorem~\ref{thm:Information_policy}}
The proof of Theorem \ref{thm:Information_policy} follows the proof of Theorem \ref{thm:History_policy}, except that we replace history with the information state.


\section{Additional Material for Section~\ref{sec:applications}}
\label{sec:applications_Appendix}

We follow three steps to establish the GCR framework for each of our applications.

\textbf{Step one: Formulate risk frontier.}
First, we formulate $\mathcal U$, starting with conditions on the premium schedule, then any conditions on the disbursements, and finally any conditions on auxiliary variables (e.g., wealth).

\textbf{Step two: Information state.} 
Next we obtain the tail risk frontiers $\{\mathcal U_t(h_t, \Phi_{t+1})\}_{t \in [T]}$, which follow immediately from the form of $\mathcal U$ in all of our applications. We can then write $\mathcal U$ in the form of Eq.~\eqref{eq:risk_frontier_recursive}, and time consistency follows immediately by Theorem~\ref{thm:recursive}.
We then identify the information state $\mathcal I = \{ (\mathbb Y_t, \sigma_t) \}_{t \in [T+1]}$ for evaluation of Markov payoff streams, obtain the compressed risk frontiers $\{\bar{\mathcal U}_t(y_t, \bar{\Phi}_{t+1})\}_{t \in [T]}$, and the modified risk frontier $\bar{\mathcal U}$ in the form of Eq.~\eqref{eq:bar_risk_frontier_recursive}.

\textbf{Step three: Attainment.}
We verify Assumption~\ref{assu:compact} using the specific form of each risk frontier.
Under Assumption~\ref{assu:preliminaries}, $\mathcal L_{0:T+1}$ is isomorphic to a Euclidean space and we can map $X_{0:T} \in \mathcal L_{1:T+1}$ to $X_{0:T} \equiv (X_0(\omega), \ldots, X_T(\omega))_{\omega \in \Omega}$ which enumerates its realizations as a vector.
We will verify Assumption~\ref{assu:compact} by showing directly that all $(Z_{0:T}, \Phi_{0:T}) \in \mathcal{U}_{\tau}(X_{0:T})$ are lower and upper bounded, then compactness follows by closedness of $\mathcal U$ under Definition~\ref{defn:risk_frontier}(iv).

We use the following result to upper bound the premium schedule for these applications.
\begin{lem}
\label{lem:upper_bound}
Suppose $\mathbb E[\Phi_t] < \infty$ for all $t \in [T]$ and $\Phi_{0:T}$ is lower bounded, then $\Phi_{0:T}$ is upper bounded (i.e., $\Phi_t(h_t)$ is uniformly upper bounded for all $h_t \in \mathbb{H}_t$ and $t \in [T]$).
\end{lem}
\begin{proof}
First, we suppose $\mathbb P(h_t) > 0$ for all $h_t \in \mathbb{H}_t$ and $t \in [T]$.
If $\mathbb P(h_{t}') = 0$ for any $h_{t}' \in \mathbb{H}_t$, then we can simply delete the corresponding trajectory with history $h_t'$ from $\Omega$ without changing the problem.
Let $\kappa_L > -\infty$ be a lower bound on $\Phi_{0:T}$ where $\Phi_t(h_t) \geq \kappa_L$ for all $h_t \in \mathbb{H}_t$ and $t \in [T]$.

For $t \in [T]$, we have $\mathbb E[\Phi_{t}] = \sum_{h_{t} \in \mathbb{H}_{t}} \Phi_{t}(h_{t}) \mathbb P(h_{t})$. For any fixed $h_t \in \mathbb{H}_t$, we then have
$$
\mathbb E[\Phi_{t}] = \Phi_{t}(h_{t}) \mathbb P(h_{t}) + \sum_{h_{t}' \in \mathbb{H}_t,\, h_{t}' \ne h_{t}} \Phi_{t}(h_{t}') \mathbb P(h_{t}') \leq \kappa_{U}
$$
for some $\kappa_{U} < \infty$.
Using the fact that $\Phi_{0:T}$ is lower bounded, we then have
$$
\left( \min_{h_{t} \in \mathbb{H}_{t}} \mathbb P(h_{t}) \right) \left( \Phi_t(h_t) + \sum_{h_{t}' \ne h_{t}} \kappa_L \right) \leq \left( \min_{h_{t} \in \mathbb{H}_{t}} \mathbb P(h_{t}) \right) \sum_{h_{t} \in \mathbb{H}_t} \Phi_{t}(h_{t}) \leq \mathbb E[\Phi_{t}] \leq \kappa_{U}.
$$
Since $\min_{h_{t} \in \mathbb{H}_{t}} \mathbb P(h_{t}) > 0$, the above inequalities imply that $\Phi_t(h_t)$ is upper bounded.
This argument is identical for all $t \in [T]$ and $h_t \in \mathbb{H}_t$ so $\Phi_{0:T}$ is upper bounded.
\end{proof}

For a fixed payoff stream $X_{0:T} \in \mathcal L_{1:T+1}$, we use the fact that all $\{X_t\}_{t \in [T]}$ are bounded and that $X_t \in [\underline{X}_t, \bar{X}_t]$ holds almost surely for some $-\infty < \underline{X}_t \leq \bar{X}_t < \infty$ for all $t \in [T]$.

\subsection{Original State Space}

\subsubsection{Risk-neutral}

\paragraph*{Step one: Formulate risk frontier.}
The risk frontier for $\varrho_{{\rm RN}}$ is determined by the constraints:
\begin{subequations}
\label{eq:frontier_risk-neutral}
\begin{align}
\Phi_t(h_t) \geq & \mathbb E[\Phi_{t+1} - X_t \vert h_t],\, \forall h_t \in \mathbb{H}_t,\, t \in [T],\\
\Phi_{T+1} \equiv & 0.
\end{align}
\end{subequations}

\paragraph*{Step two: Information state.} 
For period $t \in [T]$, the tail risk frontier $\mathcal U_t(h_t, \Phi_{t+1})$ for $(X_{t}, \varphi_{t})$ is determined by the constraints:
$$
\varphi_t \geq \mathbb E[\Phi_{t+1} - X_t \vert h_t].
$$
Suppose we are on history $h_t \in \mathbb{H}_t$ and we are given $\bar{\Phi}_{t+1} \in \bar{\mathcal L}_{t+1}$.
Acceptable $(X_t, \varphi_t)$ must then satisfy the inequality $\varphi_t \geq \mathbb E[\bar{\Phi}_{t+1}(f_t(s_t, \xi_t)) - X_t(s_t, \xi_t)]$, which only depends on $s_t$.

\paragraph*{Step three: Attainment.}
Using $\Phi_{T+1} \equiv 0$, we have $\Phi_t(h_t) \geq - \mathbb E[\sum_{k=t}^T X_k \vert h_t] \geq - \sum_{k=t}^T \bar X_k$ for all $h_t \in \mathbb{H}_t$ and $t \in [T]$.
It follows that $\Phi_{0:T}$ is lower bounded.

Now suppose $\varphi_0 \leq \tau$. For period $t = 1$, we have $\mathbb E[\Phi_1 \vert h_0] \leq \tau + \mathbb E[X_0 \vert h_0] \leq \tau + \bar{X}_0$ for all $h_0 \in \mathbb{H}_0$. Taking unconditional expectations gives $\mathbb E[\Phi_1] \leq \tau + \bar{X}_0$. Similarly, for period $t+1$ we have $\mathbb E[\Phi_{t+1} \vert h_t] \leq \Phi_t(h_t) + \mathbb E[X_t \vert h_t] \leq \Phi_t(h_t) + \bar{X}_t$ for all $h_t \in \mathbb{H}_t$. Taking unconditional expectations gives $\mathbb E[\Phi_{t+1}] \leq \mathbb E[\Phi_t] + \bar{X}_t$ and by induction we obtain $\mathbb E[\Phi_{t+1}] \leq \sum_{k=0}^t \mathbb E[X_t] + \tau \leq \sum_{k=0}^t \bar{X}_t + \tau$. It follows that $\mathbb E[\Phi_t]$ for all $t \in [T]$ are upper bounded.
Finally, $\Phi_{0:T}$ is upper bounded by Lemma~\ref{lem:upper_bound}.



\subsubsection{Entropic risk measure}

\paragraph*{Step one: Formulate risk frontier.}
The risk frontier for $\varrho_{{\rm Ent}}$ is determined by the constraints:
\begin{subequations}
\label{eq:frontier_entropic}
\begin{align}
\Phi_t(h_t) \geq & \mathbb E[e^{- \gamma\, X_t} \Phi_{t+1} \vert h_t],\, \forall h_t \in \mathbb{H}_t,\, t \in [T],\\
\Phi_{T+1} \equiv & 1.
\end{align}
\end{subequations}
We check that Definition~\ref{defn:risk_frontier}(i) (monotonicity in payoffs) is satisfied for Eq.~\eqref{eq:frontier_entropic}. For $t \in [T]$, we have
$$
\varphi_t + \varphi_t' \geq \mathbb E[e^{-\gamma X_t}\Phi_{t+1}] + \varphi_t' \geq \mathbb E[e^{-\gamma X_t}(\Phi_{t+1} + \varphi_t')] \geq \mathbb E[e^{-\gamma X_t}(\Phi_{t+1} + \Phi_{t+1}')],
$$
using $X_t \geq 0$ by assumption and $\varphi_t' \geq \Phi_{t+1}'$.

\paragraph*{Step two: Information state.}
For period $t \in [T]$, the tail risk frontier $\mathcal U_t(h_t, \Phi_{t+1})$ for $(X_{t}, \varphi_{t})$ is determined by the constraints:
$$
\varphi_t \geq \mathbb E[e^{- \gamma\, X_t} \Phi_{t+1} \vert h_t].
$$
For each period $t \in [T]$, suppose we are on history $h_t \in \mathbb{H}_t$ and we are given $\bar{\Phi}_{t+1} \in \bar{\mathcal L}_{t+1}$.
Acceptable $(X_t, \varphi_t)$ must then satisfy the inequality $\varphi_t \geq \mathbb E[e^{- \gamma\, X_t(s_t, \xi_t)} \bar{\Phi}_{t+1}(f_t(s_t, \xi_t))]$, which only depends on $s_t$.

\paragraph*{Step three: Attainment.}
The period $t \in [T]$ premium satisfies $\Phi_t(h_t) \geq \mathbb E[e^{- \gamma \sum_{k=t}^{T} X_k } \vert h_t] \geq e^{- \gamma \sum_{k=t}^{T} \bar X_k }$ for all $h_t \in \mathbb{H}_t$.
In particular, the period $t = 0$ premium satisfies $\varphi_0 \geq \mathbb E[e^{- \gamma \sum_{t = 0}^{T} X_t }] \geq e^{- \gamma \sum_{t = 0}^{T} \bar X_t }$.
It follows that $\Phi_{0:T}$ is lower bounded.

Now suppose $\varphi_0 \leq \tau$. For period $t \in [T]$, we have $\Phi_t(h_t) \geq \mathbb E[e^{- \gamma\, X_t } \Phi_{t+1} \vert h_t] \geq \mathbb E[e^{- \gamma\, \bar{X}_t } \Phi_{t+1} \vert h_t]$ for all $h_t \in \mathbb{H}_t$. Taking unconditional expectations gives $\mathbb E[\Phi_{t+1}] \leq e^{\gamma\, \bar{X}_t} \mathbb E[ \Phi_{t}]$.
Using $\varphi_0 \leq \tau$, by induction we obtain $\mathbb E[\Phi_{t+1}] \leq \tau\,e^{\gamma \sum_{k=0}^{t} \bar{X}_k }$ for all $t \in [T]$.
It follows that $\Phi_{0:T}$ is upper bounded by Lemma~\ref{lem:upper_bound}.

\subsubsection{Nested risk measures}

\paragraph*{Step one: Formulate risk frontier.}
The risk frontier for $\varrho_{{\rm N}}$ is determined by the constraints:
\begin{subequations}
\label{eq:frontier_nested}
    \begin{align}
        \Phi_t(h_t) \geq & Z_t(h_t),\, \forall h_t \in \mathbb{H}_t,\, t \in [T],\\
        \Phi_{T+1} \equiv & 0,\\
        [X_t - \Phi_{t+1}\vert h_t] + Z_t(h_t) \in & \mathcal A_t,\, \forall h_t \in \mathbb{H}_t, t \in [T].
    \end{align}
\end{subequations}
%

\paragraph*{Step two: Information state.}
For period $t \in [T]$, the tail risk frontier $\mathcal U_t(h_t, \Phi_{t+1})$ for $(X_{t}, z_{t}, \varphi_{t})$ is determined by the constraints:
\begin{equation*}
\varphi_t \geq z_t,\, [X_t - \Phi_{t+1} \vert h_t] + z_t \in \mathcal A_t. 
\end{equation*}
For each period $t \in [T]$, suppose we are on history $h_t \in \mathbb{H}_t$ and we are given $\bar{\Phi}_{t+1} \in \bar{\mathcal L}_{t+1}$.
Acceptable $(X_t, z_t, \varphi_t)$ must then satisfy:
\begin{equation*}
\varphi_t \geq z_t, [X_t - \bar{\Phi}_{t+1} \vert s_t] + z_t \in \mathcal A_t,
\end{equation*}
which only depend on $s_t$.

\paragraph*{Step three: Attainment.}
For period $T$, for all $h_T \in \mathbb{H}_T$ we have $X_T(h_T) + Z_T(h_T) \in \mathcal A_T$ and so $Z_T(h_T)$ is lower bounded by $\rho_{\mathcal A_T}(X_T(h_T))$. Hence, $\Phi_T(h_T)$ is also lower bounded by $\rho_{\mathcal A_T}(X_T(h_T))$. We then get the uniform lower bound $\underline{\varphi}_T \equiv \rho_{\mathcal A_T}(\bar{X}_T)$ for $Z_T$ and $\Phi_T$.
For period $T-1$, we have
\begin{align*}
    Z_{T-1}(h_{T-1}) \geq & \min_{z \in \mathbb R} \{z : [X_{T-1} \vert h_{T-1}] - [\Phi_T \vert h_{T-1}] + z \in \mathcal A_{T-1}\}\\
    \geq & \min_{z \in \mathbb R} \{z : \bar{X}_{T-1} - \underline{\varphi}_T + z \in \mathcal A_{T-1}\}\\
    = & \rho_{\mathcal A_{T-1}}(\bar{X}_{T-1} - \underline{\varphi}_T).
\end{align*}
We then get the uniform lower bound $\underline{\varphi}_{T-1} \equiv \rho_{\mathcal A_{T-1}}(\bar{X}_{T-1} - \underline{\varphi}_T)$ for $Z_{T-1}$ and $\Phi_{T-1}$.
Proceeding by induction, for period $t$, we have:
\begin{align*}
    z_{t} \geq & \min_{z \in \mathbb R} \{z : [X_{t} \vert h_{t}] - [\Phi_{t+1} \vert h_{t}] + z \in \mathcal A_{t}\}\\
    \geq & \min_{z \in \mathbb R} \{z : \bar{X}_{t} - \underline{\varphi}_{t+1} + z \in \mathcal A_{t}\}\\
    = & \rho_{\mathcal A_{t}}(\bar{X}_{t} - \underline{\varphi}_{t+1}).
\end{align*}
We then get the uniform lower bound $\underline{\varphi}_{t} \equiv \rho_{\mathcal A_{t}}(\bar{X}_{t} - \underline{\varphi}_{t+1})$ for both $Z_t$ and $\Phi_{t}$, for all $t \in [T]$.

Now suppose $\varphi_0 \leq \tau$.
For period $t = 0$, we have $X_0 - \Phi_{1} + z_0 \in \mathcal A_0$ and $z_0 \leq \varphi_0 \leq \tau$. We then see that $\Phi_1$ is upper bounded by assumption that $\mathcal A_0$ is pointed.
By induction, for period $t$ suppose $\Phi_t$ is upper bounded. We have $X_t - \Phi_{t+1} + Z_t \in \mathcal A_t$ and $Z_t \leq \Phi_t$. Then $Z_t$ is upper bounded by $\Phi_t$, and $\Phi_{t+1}$ is again upper bounded by assumption that $\mathcal A_t$ is pointed.

\subsection{Wealth-based}

\subsubsection{Standard capital requirement}

\paragraph*{Step one: Formulate risk frontier.}
The risk frontier for $\varrho_{\mathcal A_{\times}}$ is determined by the constraints:
\begin{subequations}
\label{eq:frontier_standard}
    \begin{align}
        \varphi_0 \geq & w_0 + \mathbb E[\Phi_1],\\
        \Phi_t(h_t) \geq & \mathbb E[\Phi_{t+1} \vert h_t],\, \forall h_t \in \mathbb{H}_t,\, t \in [1, T],\\
        \Phi_{T+1} \equiv & \delta_{\{W_{T+1} \geq 0\}},\\
        [X_t \vert h_t] + Z_t(h_t) \in & \mathcal A_t,\, \forall h_t \in \mathbb{H}_t,\, t \in [T],\\
        w_0 \in & \Re,\, W_{t+1} = W_t - Z_t,\, \forall t \in [T],\, W_{T+1} \geq 0.
    \end{align}
\end{subequations}

\paragraph*{Step two: Information state.}
The period $T$ tail risk frontier $\mathcal U_{T}(h_T, \Phi_{T+1})$ for triples $(X_{T}, z_{T}, \varphi_{T})$ is determined by the constraints: $[X_T \vert h_T] + z_T \in \mathcal A_T$ and $W_T(h_T) - z_T \geq 0$.
In period $T$ on history $h_T \in \mathbb{H}_T$ (where $\Phi_{t+1}$ is fixed), acceptable $(X_T, z_T, \varphi_T)$ must satisfy: $[X_T \vert s_T] + z_T \in \mathcal A_T$ and $w_T - z_T \geq 0$, which only depend on $y_T$.

The period $t \in [1, T-1]$ tail risk frontier $\mathcal U_{t}(h_t, \Phi_{t+1})$ for triples $(X_{t}, z_{t}, \varphi_{t})$ is determined by the constraints: $[X_t \vert h_t] + z_t \in \mathcal A_t$ and $\varphi_t \geq \mathbb E[\Phi_{t+1} \vert h_t]$.
For period $t \in [1, T-1]$ on history $h_t \in \mathbb{H}_t$, suppose we are given $\bar{\Phi}_{t+1} \in \bar{\mathcal L}_{t+1}$. Then, acceptable $(X_t, z_t, \varphi_t)$ must satisfy:
\begin{equation*}
    \varphi_t \geq \mathbb E[\bar{\Phi}_{t+1}(g_t(y_t, z_t, \xi_t))],\, [X_t \vert s_t] + z_t \in \mathcal A_t,
\end{equation*}
which only depend on $y_t$. 

The period $t = 0$ tail risk frontier $\mathcal U_{0}(h_0, \Phi_{1})$ for triples $(X_{0}, z_{0}, \varphi_{0})$ is determined by the constraints: $\varphi_0 \geq w_0 + \mathbb E[\Phi_1]$, $X_0 + z_0 \in \mathcal A_0$, and $w_0 \in \Re$.
The information state for period $t=0$ follows similarly.

\paragraph*{Step three: Attainment.}
We have $\Phi_{T+1} \geq 0$ since $\Phi_{T+1}$ is an indicator function. For period $T$, we have $\Phi_T(h_T) \geq \mathbb E[\Phi_{T+1} \vert h_T] \geq 0$ for all $h_T \in \mathbb{H}_T$. By induction, for period $t \in [1, T]$ suppose $\Phi_{t+1} \geq 0$, then we have $\Phi_t(h_t) \geq \mathbb E[\Phi_{t+1} \vert h_t] \geq 0$ for all $h_t \in \mathbb{H}_t$.  For each period $t \in [T]$, we must have
$$
Z_t(h_t) \geq \min_{z \in \Re} \{z : [X_t \vert h_t] + z \in \mathcal A_t\} \geq \min_{z \in \Re} \{z : \bar{X}_t + z \in \mathcal A_t\},\, \forall h_t \in \mathbb{H}_t,
$$
where the RHS is independent of $h_t$ (as it uses the maximum payoff $\bar{X}_t$). It follows that $Z_{0:T}$ is lower bounded.
For period $t = 0$, we have the lower bound $w_0 \geq - \sum_{t=0}^T \bar X_t$ or else it is not possible to satisfy the terminal wealth constraint $W_{T+1} \geq 0$ (even if we receive the maximum payoff in every period).

Now suppose $\varphi_0 \leq \tau$. Since $\mathbb E[\Phi_1] \geq 0$, we have $w_0 \leq \tau$ and so $w_0$ is upper bounded.
For each $t \in [1, T]$, we have $\mathbb E[\Phi_{t+1} \vert h_t] \leq \Phi_t(h_t)$ for all $h_t \in \mathbb{H}_t$. Taking unconditional expectations gives $\mathbb E[\Phi_{t+1}] \leq \mathbb E[\Phi_t]$ for all $t \in [T]$, and so $\mathbb E[\Phi_t] \leq \tau$ for all $t \in [T]$. It then follows by Lemma~\ref{lem:upper_bound} that $\Phi_{0:T}$ is upper bounded.
Since we have $\mathbb E[\Phi_1] \leq \tau - w_0 < \infty$, $W_{T+1} \geq 0$ must hold almost surely (or else $\mathbb E[\Phi_1] = \infty$).
Since $W_{T+1} \geq 0$, we cannot disburse more in any period than our initial endowment plus total maximum payoff.
It then follows that
$$
Z_t(h_t) \leq w_0 + \sum_{k=0}^T \bar X_k \leq \tau + \sum_{k=0}^T \bar X_k,\, \forall h_t \in \mathbb{H}_t,\, t \in [T],
$$
and so $Z_{0:T}$ is upper bounded.



\subsubsection{Consumption}

\paragraph*{Step one: Formulate risk frontier.}
The risk frontier for $\varrho_{{\rm C}}$ is determined by the constraints: 
\begin{subequations}
\label{eq:frontier_consumption}
    \begin{align}
        \Phi_t(h_t) \geq & \mathbb E[\psi_t(W_t, X_t, Z_t) + \Phi_{t+1} \vert h_t],\, \forall h_t \in \mathbb{H}_t,\, t \in [T],\\ \Phi_{T+1} \equiv & 0,\\
        Z_{0:T} \geq & 0,\\
        W_{t+1} = & \gamma(W_t, X_t, Z_t),\, \forall t \in [T-1].
    \end{align}
\end{subequations}
This risk frontier is non-convex in its current form due to the nonlinear wealth state equation $\gamma$, but it has a convex reformulation through the introduction of dummy shortfall variables as in \cite{pflug2001risk,pflug2005measuring}. Let $M_t \in \mathcal L_{t+1}$ where $M_t \sim \max \{Z_t - W_t - X_t, 0 \}$ is a dummy variable for the period $t \in [T]$ consumption shortfall. Then, we can reformulate Eq.~\eqref{eq:frontier_consumption} as:
\begin{subequations}
\label{eq:frontier_consumption-1}
    \begin{align}
        \Phi_t(h_t) \geq & \mathbb E[\alpha_t Z_t - c_t M_t + \Phi_{t+1} \vert h_t],\, \forall h_t \in \mathbb{H}_t,\, t \in [T],\\ \Phi_{T+1} \equiv & 0,\\
        Z_{0:T} \geq & 0,\\
        W_{t+1} = & W_t + X_t - Z_t + M_t,\, \forall t \in [T-1],\\
        M_{0:T} \geq & 0.
    \end{align}
\end{subequations}
Eq.~\eqref{eq:frontier_consumption-1} only consists of linear constraints, and so is convex.

\paragraph*{Step two: Information state.}
The tail risk frontier $\mathcal U_{t}(h_t, \Phi_{t+1})$ for triples $(X_{t}, z_{t}, \varphi_{t})$ is determined by the constraints:
$$
\varphi_t \geq \mathbb E[\psi_t(W_t, X_t, Z_t) + \Phi_{t+1} \vert h_t],\, z_t \geq 0.
$$
For each period $t \in [T]$, suppose we are on history $h_t \in \mathbb{H}_t$ and we are given $\bar{\Phi}_{t+1} \in \bar{\mathcal L}_{t+1}$.
Then, acceptable $(X_t, z_t, \varphi_t)$ must satisfy:
\begin{equation*}
\varphi_t \geq \mathbb E[\psi_t(w_t, X_t(s_t, \xi_t), z_t + \bar{\Phi}_{t+1}(g_t(y_t, z_t, \xi_t))],\, z_t \geq 0,
\end{equation*}
which only depend on $y_t$.

\paragraph*{Step three: Attainment.}
First we have $\Phi_t(h_t) \geq \mathbb E[\sum_{k=t}^T \psi_k(W_k, X_k, Z_k) \vert h_t]$ for all $h_t \in \mathbb{H}_t$ and $t \in [T]$.
The initial wealth $w_0$ is fixed, and initial wealth plus total payoff is upper bounded by $w_0 + \sum_{t=0}^T \bar X_t$. Then in periods $[t, T]$, the total utility from consumption is upper bounded by $(\max_{t \in [T]} \alpha_t)(w_0 + \sum_{t=0}^T \bar X_t)$, as we cannot consume more than initial wealth plus maximum total payoff.
We then have the lower bounds
$$
\Phi_t(h_t) \geq \mathbb E \left[\sum_{k=t}^T \psi_k(W_k, X_k, Z_k) \vert h_t \right] \geq - \left( \max_{t \in [T]} \alpha_t \right) \left( w_0 + \sum_{t=0}^T \bar X_t \right)
$$
for all $t \in [T]$, where there is no shortfall cost.

Now suppose $\varphi_0 \leq \tau$. Taking unconditional expectations, we have the upper bounds $\mathbb E[\Phi_{t+1}] \leq \tau - \mathbb E[\sum_{k=0}^t \psi_k(W_k, X_k, Z_k)]$ for all $t \in [T]$. Using the fact that $\mathbb E[\sum_{k=0}^t \psi_k(W_k, X_k, Z_k)]$ is lower bounded, we have
$$
\mathbb E[\Phi_{t}] \leq \tau + \left( \max_{t \in [T]} \alpha_t \right) \left( w_0 + \sum_{t=0}^T \bar X_t \right),\, \forall t \in [T].
$$
It follows that $\Phi_{0:T}$ is upper bounded by Lemma~\ref{lem:upper_bound}. 
Next, we have that $Z_{0:T} \geq 0$ are automatically lower bounded.
To conclude, we have $\mathbb E[\sum_{t=0}^T \psi_t(W_t, X_t, Z_t)] \leq \varphi_0 \leq \tau$.
For any $t \in [T]$ and fixed $w_t$, we have $\lim_{z_t \rightarrow \infty} \mathbb E[\psi_t(w_t, X_t, z_t)] = \infty$. It follows that $Z_{0:T}$ is also upper bounded.

\subsubsection{Consumption excess}

\paragraph*{Step one: Formulate risk frontier.}
The risk frontier for $\varrho_{{\rm CE}}$ is determined by the constraints:
\begin{subequations}
\label{eq:frontier_excess}
    \begin{align}
        \Phi_t(h_t) \geq & \alpha_t + \max_{\xi_t \in \Xi} \Phi_{t+1}(h_t, \xi_t),\, \forall h_t \in \mathbb{H}_t,\, t \in [T],\\
        \Phi_{T+1} \equiv & 0,\\
        \psi_t(\alpha_t, [X_t \vert h_t], Z_t(h_t), s_t) \geq & 0,\, \forall h_t \in \mathbb{H}_t,\, t \in [T],\\
        W_{t+1} = & \gamma(W_t, Z_t),\, \forall t \in [T],\\
        W_{T+1} \geq & 0,\\
        \alpha_t \geq & \epsilon,\, \forall t \in [T].
    \end{align}
\end{subequations}
The functions $(\alpha_t, z_t) \rightarrow \psi_t(\alpha_t, X_t, z_t, s_t)$ are concave for fixed $X_t$ as the perspective of the concave function $u$, for fixed $s_t \in \mathbb{S}$ for all $t \in [T]$.
We can write the wealth update as an inequality to obtain the risk frontier in convex form.

For period $T$, we have $\Phi_T(h_T) \geq \epsilon$ for all $h_T \in \mathbb{H}_T$ since $\alpha_T \geq \epsilon$ and $\Phi_{T+1} \equiv 0$. By induction, for period $t \in [T]$ we have $\Phi_t(h_t) \geq \epsilon$ for all $h_t \in \mathbb{H}_t$ since $\alpha_t \geq \epsilon$ and $\Phi_{t+1} \geq 0$ by the induction step.
It follows that $\Phi_{0:T}$ is lower bounded, and $\alpha_{0:T}$ is automatically lower bounded since $\alpha_t \geq \epsilon$ for all $t \in [T]$.

\paragraph*{Step two: Information state.}
The period $T$ tail risk frontier $\mathcal U_T(h_t, \Phi_{T+1})$ for triples $(X_T, z_T, \varphi_T)$ is determined by the constraints:
$$
\varphi_T \geq \alpha_T, \psi_T(\alpha_T, [X_T \vert h_T], z_T; s_T) \geq 0, \gamma_T(w_T, z_T) \geq 0, \alpha_T \geq \epsilon.
$$
The period $t \in [T-1]$ tail risk frontier $\mathcal U_{t}(h_t, \Phi_{t+1})$ for triples $(X_{t}, z_{t}, \varphi_{t})$ is determined by the constraints:
$$
\varphi_t \geq \alpha_t + \max_{\xi_t \in \Xi} \Phi_{t+1}(h_t, \xi_t), \psi_t(\alpha_t, [X_t \vert h_t], z_t; s_t) \geq 0, \alpha_t \geq \epsilon.
$$

For each period $t \in [T]$, suppose we are on history $h_t \in \mathbb{H}_t$ and we are given $\bar{\Phi}_{t+1} \in \bar{\mathcal L}_{t+1}$. Then, acceptable $(X_t, z_t, \varphi_t)$ must satisfy:
\begin{equation*}
    \psi_t(\alpha_t, [X_t \vert h_t], z_t; s_t) \geq 0,\, \varphi_t \geq \alpha_t + \max_{\xi_t \in \Xi} \bar{\Phi}_{t+1}(g_t(y_t, z_t, \xi_t)),\, \alpha_t \geq \epsilon,
\end{equation*}
which only depend on $y_t$.

\paragraph*{Step three: Attainment.}
Now suppose $\varphi_0 \leq \tau$.
Then $\alpha_{0:T}$ is immediately upper bounded since $\sum_{t=0}^T \alpha_t \leq \varphi_0 \leq \tau$.
For all $h_t \in \mathbb{H}_t$ and $t \in [T]$, we also have
$$
\Phi_t(h_t) \geq \Phi_{t+1}(h_t, \xi_t, f_t(s_t, \xi_t)),\, \forall \xi_t \in \Xi.
$$
Then, it follows that $\Phi_t(h_t) \leq \tau$ for all $t \in [T]$ and so $\Phi_{0:T}$ is upper bounded.

Next, we see that $Z_t$ is lower bounded by feasibility of the consumption excess constraints via:
\begin{align*}
& \min_{z_t \in \Re, \alpha_t \geq \epsilon} \{ z_t : \alpha_t \mathbb E[u((X_t + z_t - \zeta_t(s_t)) / \alpha_t) \vert s_t] \geq 0 \}\\
\geq & \min_{z_t \in \Re, \alpha_t \geq \epsilon} \{z_t : \alpha_t \mathbb E[u((\bar{X}_t + z_t - \min_{s_t \in \mathbb S} \zeta_t(s_t)) / \alpha_t)] \geq 0\},
\end{align*}
where $\min_{s_t \in \mathbb S} \zeta_t(s_t)$ is the smallest period $t$ consumption target.
Let $\bar X = \max_{t \in [T]} \bar{X}_t$ be the maximum of the payoff upper bounds $\bar{X}_t$.
Then, given initial wealth $w_0$, total cumulative wealth is upper bounded by $(w_0 + \bar X) \sum_{t=0}^T(1+\beta^S)^{t}$ which accounts for the savings return rate (it assumes all cash is invested and there are no withdrawals).
Since $W_{T+1} \geq 0$ must hold, we must have
$$
Z_t(h_t) \leq (w_0 + \bar X) \sum_{t=0}^T(1+\beta^S)^{t},\, \forall h_t \in \mathbb{H}_t,\, t \in [T],
$$
i.e., withdrawals cannot exceed an upper bound on cumulative wealth.
It follows that $Z_{0:T}$ is also upper bounded.

\subsection{Target-based}

\subsubsection{Expected utility}

\paragraph*{Step one: Formulate risk frontier.}
The risk frontier for $\varrho_{{\rm EU}}$ is determined by the constraints:
\begin{subequations}
\label{eq:frontier_utility}
    \begin{align}
        \Phi_t(h_t) \geq & \mathbb E[\Phi_{t+1} \vert h_t],\, \forall h_t \in \mathbb{H}_t,\, t \in [T],\\
        \Phi_{T+1}(h_{T+1}) = & u \left( \sum_{t=0}^T Z_t(h_{T+1}) \right),\, \forall h_{T+1} \in \mathbb{H}_{T+1},\\
        [X_t \vert h_t] + Z_t(h_t) \in & \mathcal A_t,\, \forall h_t \in \mathbb{H}_t,\, t \in [T],\\
        Z_{0:T} \geq & 0.
    \end{align}
\end{subequations}
%

\paragraph*{Step two: Information state.}
The period $t$ tail risk frontier $\mathcal U_{t}(h_t, \Phi_{t+1})$ for triples $(X_{t}, z_{t}, \varphi_{t})$ is determined by the constraints: $\varphi_t \geq \mathbb E[\Phi_{t+1} \vert h_t]$; $[X_t \vert h_t] + z_t \in \mathcal A_t$; and $z_t \geq 0$.
For each period $t \in [T]$, suppose we are on history $h_t \in \mathbb{H}_t$ and we are given $\bar{\Phi}_{t+1} \in \bar{\mathcal L}_{t+1}$.
Then, acceptable $(X_t, z_t, \varphi_t)$ must satisfy:
\begin{align*}
\varphi_t \geq \mathbb E[\bar{\Phi}_{t+1}(g_t(y_t, z_t, \xi_t))], [X_t \vert s_t] + z_t \in \mathcal A_t,\, z_t \geq 0,
\end{align*}
which only depend on $y_t$.

\paragraph*{Step three: Attainment.}
First $Z_{0:T} \geq 0$ is automatically lower bounded.
Then, it follows that $\Phi_{T+1} \geq u(0)$ since $\sum_{t=0}^T Z_t \geq 0$ and $u$ is non-decreasing.
For earlier periods $t \in [T]$, we have $\Phi_t(h_t) \geq \mathbb E[\Phi_{t+1} \vert h_t] \geq u(0)$ for all $h_t \in \mathbb{H}_t$. It follows that all $\Phi_t(h_t)$ are lower bounded.

Now suppose $\varphi_0 \leq \tau$.
Taking unconditional expectations gives $\mathbb E[\Phi_t] \geq \mathbb E[\Phi_{t+1}]$ for all $t \in [T]$.
It follows that $\mathbb E[\Phi_t] \leq \tau$ for all $t \in [T]$, and so $\Phi_{0:T}$ is upper bounded by Lemma~\ref{lem:upper_bound}.
For any $h_t \in \mathbb{H}_t$, we have $\mathbb E [ u (\sum_{t=0}^T Z_t) ] \geq \mathbb P(h_t) u(Z_t(h_t))$
for $\mathbb P(h_t) > 0$. We then have $\mathbb P(h_t) u(Z_t(h_t)) \leq \tau$ which implies $u(Z_t(h_t)) \leq \tau/(\min_{h_t \in \mathbb{H}_t} \mathbb P(h_t))$. Since $u$ is increasing, it follows that $Z_{0:T}$ is upper bounded.

\subsubsection{Worst-case expectation}

\paragraph*{Step one: Formulate risk frontier.}
Feasible adversary strategies $M_{0:T} \in \mathcal Q$ must satisfy:
\begin{subequations}
\label{eq:frontier_robust_adversary}
\begin{align}
    M_t(h_t) & \in \Delta_t(\eta_t(h_t)),\, \forall h_t \in \mathbb{H}_t,\, t \in [T],\\
    \eta_{t+1}(h_t, \xi_t) & = \eta_t(h_t) - d(M_t(h_t) P_t, P_t),\, \forall h_t \in \mathbb{H}_t,\, \xi_t \in \Xi,\, t \in [T].
\end{align}
\end{subequations}
By definition, $\varrho_{{\rm R}}$ can be reformulated as:
\begin{align*}
    \varrho_{{\rm R}}(X_{0:T}) = \min_{Z_{0:T}, \vartheta \in \Re} \quad & \vartheta\\
    \text{s.t.} \quad & \vartheta \geq \mathbb E_{M_{0:T}} \left[ \sum_{t=0}^T Z_t \right],\, \forall M_{0:T} \in \mathcal{Q},\\
    & X_{0:T} + Z_{0:T} \in \mathcal A_{\times},\\ & Z_{0:T} \in \mathcal{C}_{\geq 0}.
\end{align*}
For any fixed $M_{0:T} \in \mathcal Q$, we define the corresponding system $\mathcal{U}(M_{0:T})$:
\begin{subequations}
\label{eq:frontier_robust}
\begin{align}
    \Phi_t(h_t) \geq & Z_t(h_t) + \mathbb E[\Phi_{t+1}(h_t, \xi_t) [M_t(h_t)](\xi_t)],\, \forall h_t \in \mathbb{H}_t,\, t \in [T],\\
    \Phi_{T+1} \equiv & 0,\\
    [X_{t} \vert h_t] + Z_{t}(h_t) \in & \mathcal A_{t},\, \forall h_t \in \mathbb{H}_t,\, t \in [T],\\
    Z_{0:T} \geq & 0.
\end{align}
\end{subequations}
The risk frontier for $\varrho_{{\rm R}}$ then satisfies:
$$
\mathcal U = \bigcap_{M_{0:T} \in \mathcal Q} \mathcal{U}(M_{0:T}).
$$

\paragraph*{Step two: Information state.} 
The period $t \in [T]$ tail risk frontier $\mathcal U_{t}(h_t, \Phi_{t+1})$ for triples $(X_{t}, z_{t}, \varphi_{t})$ is determined as follows.
In period $t \in [T]$, for any fixed $m_{t} \in \Delta_{t}(\eta_t(h_t))$ we define the system $\mathcal U_{t}(h_t, \Phi_{t+1}, m_{t})$:
\begin{align*}
\varphi_t \geq & z_t + \mathbb E[\Phi_{t+1}(h_t, \xi_t)m_t(\xi_t)],\\
[X_t \vert h_t] + z_t \in & \mathcal A_t,\, z_t \geq 0.
\end{align*}
Then, we have
$$
\mathcal U_{t}(h_t, \Phi_{t+1}) = \bigcap_{m_{t} \in \Delta_t(\eta_t(h_t))} \mathcal U_{t}(h_t, \Phi_{t+1}, m_{t}).
$$

For each period $t \in [T]$, suppose we are on history $h_t \in \mathbb{H}_t$ and we are given $\bar{\Phi}_{t+1} \in \bar{\mathcal L}_{t+1}$.
Then, acceptable $(X_t, z_t, \varphi_t)$ must satisfy the conditions:
\begin{align*}
[X_t \vert s_t] + z_t \in & \mathcal A_t,\\
\varphi_t \geq & z_t + \mathbb E[\bar{\Phi}_{t+1}(g_t(s_t, \xi_t, \eta_t - d(m_t P_t, P_t)))],\, \forall m_t \in \Delta_t(\eta_t),\\
z_t \geq & 0,
\end{align*}
which only depend on $y_t$.

\paragraph*{Step three: Attainment.}
We have $Z_{0:T} \geq 0$ and so $Z_{0:T}$ is automatically lower bounded. Since $Z_{0:T} \geq 0$, we have $\Phi_t(h_t) \geq 0$ for all $h_t \in \mathbb{H}_t$ and $t \in [T]$ by induction.

Now suppose $\varphi_0 \leq \tau$. Eq.~\eqref{eq:frontier_robust} must be feasible for all $M_{0:T} \in \mathcal Q$.
In particular, we can take $M_t \equiv 1$ for all $h_t \in \mathbb{H}_t$ and $t \in [T]$ which belongs to $\mathcal Q$ (and corresponds to no perturbation in each period) to obtain the system:
\begin{subequations}
\label{eq:frontier_robust-2}
\begin{align}
    \Phi_t(h_t) \geq & Z_t(h_t) + \mathbb E[\Phi_{t+1} \vert h_t],\, \forall h_t \in \mathbb{H}_t,\, t \in [T],\\
    \Phi_{T+1} \equiv & 0,\\
    [X_{t} \vert h_t] + Z_{t}(h_t) \in & \mathcal A_{t},\, \forall h_t \in \mathbb{H}_t,\, t \in [T],\\
    Z_{0:T} \geq & 0.
\end{align}
\end{subequations}
Taking unconditional expectations in Eq.~\eqref{eq:frontier_robust-2} gives $\mathbb E[\Phi_t] \geq \mathbb E[\Phi_{t+1}]$ for all $t \in [T]$. It then follows that $\Phi_{0:T}$ is upper bounded, and consequently $Z_{0:T}$ is upper bounded.

\subsubsection{Conditional value-at-risk}

\paragraph*{Step one: Formulate risk frontier.} 
CVaR is equivalent to a worst-case expectation ${\rm CVaR}_{\alpha} ( \sum_{t=0}^T Z_t ) = \sup_{M_{0:T} \in \mathcal Q} \mathbb E_{M_{0:T}} [ \sum_{t=0}^T Z_t ]$ for appropriate $\mathcal Q$.
The set of feasible adversary strategies $M_{0:T} \in \mathcal Q$ explicitly satisfies:
\begin{subequations}
\label{eq:frontier_CVaR_adversary}
\begin{align}
    M_t(h_t) & \in \Delta_t(\eta_t(h_t)),\, \forall h_t \in \mathbb{H}_t,\, t \in [T],\\
    \eta_{t+1}(h_t, \xi_t) & = \eta_t(h_t)/[M_t(h_t)](\xi_t),\, \forall h_t \in \mathbb{H}_t,\, \forall \xi_t \in \Xi, t \in [T].
\end{align}
\end{subequations}
Then, $\varrho_{{\rm CVaR}}$ can be reformulated as:
\begin{align*}
    \varrho_{{\rm CVaR}}(X_{0:T}) = \min_{Z_{0:T}, \vartheta \in \Re} \quad & \vartheta\\
    \text{s.t.} \quad & \vartheta \geq \mathbb E_{M_{0:T}} \left[ \sum_{t=0}^T Z_t \right],\, \forall M_{0:T} \in \mathcal{Q},\\
    & X_{0:T} + Z_{0:T} \in \mathcal A_{\times},\\ & Z_{0:T} \in \mathcal{C}_{\geq 0}.
\end{align*}
For any fixed $M_{0:T} \in \mathcal Q$, we define the system $\mathcal U(M_{0:T})$:
\begin{subequations}
\label{eq:frontier_CVaR}
\begin{align}
    \Phi_t(h_t) \geq & Z_t(h_t) + \mathbb E \left[\Phi_{t+1} M_t \vert h_t \right],\, \forall h_t \in \mathbb{H}_t,\, t \in [T],\\
    \Phi_{T+1} \equiv & 0,\\
    [X_{t} \vert h_t] + Z_{t}(h_t) \in & \mathcal A_{t},\, \forall h_t \in \mathbb{H}_t,\, t \in [T],\\
    Z_{0:T} \geq & 0.
\end{align}
\end{subequations}
Then, the risk frontier for $\varrho_{{\rm CVaR}}$ satisfies:
$$
\mathcal U = \bigcap_{m_{0:T} \in \mathcal Q} \mathcal{U}(M_{0:T}).
$$

\paragraph*{Step two: Information state.}
The period $t \in [T]$ tail risk frontier $\mathcal U_{t}(h_t, \Phi_{t+1})$ for triples $(X_{t}, z_{t}, \varphi_{t})$ is determined as follows.
Given $m_{t} \in \Delta(\eta_t(h_t))$, we define $\mathcal U_{t}(h_t, \Phi_{t+1}, m_t)$ via:
\begin{align*}
\varphi_t \geq & z_t + \mathbb E \left[\Phi_{t+1}m_t \vert h_t \right],\\
[X_t \vert s_t] + z_t \in & \mathcal A_t,\, z_t \geq 0.
\end{align*}
Then the tail risk frontier satisfies
$$
\mathcal U_{t}(h_t, \Phi_{t+1}) = \bigcap_{m_{t} \in \Delta(\eta_t(h_t))} \mathcal U_{t}(h_t, \Phi_{t+1}, m_{t}).
$$

For each period $t \in [T]$, suppose we are on history $h_t \in \mathbb{H}_t$ and we are given $\bar{\Phi}_{t+1} \in \bar{\mathcal L}_{t+1}$.
Then, acceptable $(X_t, z_t, \varphi_t)$ must satisfy:
\begin{align*}
    \varphi_t \geq & z_t + \mathbb E \left[\bar{\Phi}_{t+1}(g_t(s_t, \xi_t, \eta_t/m_t(\xi_t)))m_t(\xi_t) \right],\, \forall m_t \in \Delta_t(\eta_t),\\
    [X_t \vert s_t] + z_t \in & \mathcal A_t,\\
    z_t \geq & 0,
\end{align*}
which only depend on $y_t$.

\paragraph*{Step three: Attainment.}
Now we verify Assumption~\ref{assu:compact}. We have $Z_{0:T} \geq 0$ and so $Z_{0:T}$ is automatically lower bounded. Since $\Phi_{T+1} \equiv 0$ and $Z_{0:T} \geq 0$, we have $\Phi_t(h_t) \geq 0$ for all $h_t \in \mathbb{H}_t$ and $t \in [T]$. It follows that $\Phi_{0:T}$ is lower bounded.

Now suppose $\varphi_0 \leq \tau$. We must have that Eq.~\eqref{eq:frontier_CVaR} is feasible for any $M_{0:T} \in \mathcal Q$. In particular, we see that $M_{0:T}$ with $M_t \equiv 1$ for all $h_t \in \mathbb{H}_t$ and $t \in [T]$ belongs to $\mathcal Q$. We then obtain:
\begin{subequations}
\label{eq:frontier_CVaR-1}
\begin{align}
    \Phi_t(h_t) \geq & Z_t(h_t) + \mathbb E \left[\Phi_{t+1} \vert h_t \right],\, \forall h_t \in \mathbb{H}_t,\, t \in [T],\\
    \Phi_{T+1} \equiv & 0,\\
    [X_{t} \vert h_t] + Z_{t}(h_t) \in & \mathcal A_{t},\, \forall h_t \in \mathbb{H}_t,\, t \in [T],\\
    Z_{0:T} \geq & 0.
\end{align}
\end{subequations}
Taking unconditional expectations gives $\mathbb E[\Phi_t] \geq \mathbb E[\Phi_{t+1}]$ for all $t \in [T]$. It then follows that $\Phi_{0:T}$ is upper bounded, and consequently $Z_{0:T}$ is upper bounded.


\subsubsection{Quantile risk measure}

\paragraph*{Step one: Formulate risk frontier.}
Following \cite{li2017quantile} and treating $-Z_t$ as the period $t$ reward, the risk frontier for $\varrho_{{\rm Q}}$ is determined by the constraints:
\begin{subequations}
\label{eq:frontier_quantile}
\begin{align}
    \Phi_t(h_t) \geq & Z_t(h_t) + \max_{\{\xi_t \in \Xi,\, [M_t(h_t)](\xi_t) \ne 1\}} \left\{ \Phi_{t+1}(h_t, \xi_t) \right\},\, \forall t \in [T],\\
    \Phi_{T+1} \equiv & 0,\\
    [X_{t} \vert h_t] + Z_{t}(h_t) \in & \mathcal A_{t},\, \forall h_t \in \mathbb{H}_t,\, t \in [T],\\
    Z_{0:T} \geq & 0,\\
    M_t(h_t) \in & \Delta_t(\eta_t(h_t)), \forall h_t \in \mathbb{H}_t,\, t \in [T],\\
    \eta_{t+1}(h_t, \xi_t) = & [M_t(h_t)](\xi_t), \forall h_t \in \mathbb{H}_t,\, t \in [T].
\end{align}
\end{subequations}
Notice that the strategy $M_{0:T} \in \mathcal Q$ is now a ``primal'' variable in Eq.~\eqref{eq:frontier_quantile} and is part of DM's optimization (rather than a ``dual'' adversarial strategy as in, e.g., worst-case expectation).

\paragraph*{Step two: Information state.}
The period $t \in [T]$ tail risk frontier $\mathcal U_{t}(h_t, \Phi_{t+1})$ for triples $(X_{t}, z_{t}, \varphi_{t})$ is determined by the constraints:
\begin{align*}
\varphi_t \leq & z_t + \min_{\{\xi_t \in \Xi,\, m_t(\xi_t) \ne 1\}} \left\{ \Phi_{t+1}(h_t, \xi_t) \right\},\\
m_t \in & \Delta_t(\eta_t),\\
[X_t \vert s_t] + z_t \in & \mathcal A_t,\\
z_t \geq & 0.
\end{align*}

For each period $t \in [T]$, suppose we are on history $h_t \in \mathbb{H}_t$ and we are given $\bar{\Phi}_{t+1} \in \bar{\mathcal L}_{t+1}$.
Then, acceptable $(X_t, z_t, \varphi_t)$ must satisfy:
\begin{align*}
\Phi_t(h_t) \leq & z_t + \min_{\{\xi_t \in \Xi,\, m_t(\xi_t) \ne 1\}} \left\{ \bar{\Phi}_{t+1}(g_t(s_t, m_t(\xi_t), \xi_t)) \right\},\\
m_t \in & \Delta_t(\eta_t),\\
[X_t \vert s_t] + z_t \in & \mathcal A_t,
\end{align*}
which only depends on $y_t$.

\paragraph*{Step three: Attainment.}
We have $Z_{0:T} \geq 0$ so $Z_{0:T}$ is automatically lower bounded. Since $\Phi_{T+1} \equiv 0$ and $Z_{0:T} \geq 0$, by induction we have $\Phi_t(h_t) \geq 0$ for all $h_t \in \mathbb{H}_t$ and $t \in [T]$. The sequence $M_{0:T}$ is automatically lower and upper bounded (i.e., we have $M_{0:T} \geq 0$ and $M_t \leq 1$ for all $t \in [T]$). Similarly, $\eta_{0:T}$ is automatically lower and upper bounded.

There is a discontinuity in the premium constraint of Eq.~\eqref{eq:frontier_quantile}, so we cannot establish an upper bound by the same methods.

\subsubsection{Growth profile}

\paragraph*{Step one: Formulate risk frontier.}
The risk frontier for $\varrho_{{\rm G}}$ is determined by the constraints:
\begin{subequations}
\label{eq:frontier_growth}
\begin{align}
    \Phi_t(h_t) \geq & \max_{\xi_t \in \Xi} \Phi_{t+1}(h_t, \xi_t),\, \forall h_t \in \mathbb{H}_t,\, t \in [T],\\
    \Phi_{t+1}(h_{T+1}) = & \max\{Z_0(h_{T+1}), \ldots, Z_T(h_{T+1})\},\, \forall h_{T+1} \in \mathbb{H}_{T+1},\\
    [W_{t+1} \vert h_t] + Z_{t}(h_t) \in & \mathcal A_t(W_{t}(h_t)),\, \forall h_t \in \mathbb{H}_t,\, t \in [T],\\
    W_{t+1} = & W_{t} + X_t,\, \forall t \in [T],\\
    Z_{0:T} \geq & 0.
\end{align}
\end{subequations}

\paragraph*{Step two: Information state.}
The period $t \in [T]$ tail risk frontier $\mathcal U_{t}(h_t, \Phi_{t+1})$ for triples $(X_{t}, z_{t}, \varphi_{t})$ is determined by the constraints:
$$
\varphi_t \geq \max_{\xi_t \in \Xi} \Phi_{t+1}(h_t, \xi_t), w_t + X_t + z_t \in \mathcal A_t(w_t), z_t \geq 0.
$$

For each period $t \in [T]$, suppose we are on history $h_t \in \mathbb{H}_t$ and we are given $\bar{\Phi}_{t+1} \in \bar{\mathcal L}_{t+1}$.
Then, acceptable $(X_t, z_t, \varphi_t)$ must satisfy:
\begin{equation*}
    \varphi_t \geq \mathbb E[\bar{\Phi}_{t+1}(g_t(y_t, z_t, \xi_t))], w_{t} + [X_t \vert s_t] + z_t \in \mathcal A_t(w_t),\, z_t \geq 0,
\end{equation*}
which only depend on $y_t$.

\paragraph*{Step three: Attainment.}
Since $Z_{0:T} \geq 0$, we automatically have that $Z_{0:T}$ is lower bounded and so $\Phi_{t+1} \geq 0$ as well.
We then have $\Phi_t(h_t) \geq \max_{\xi_t \in \Xi} \Phi_{t+1}(h_t, \xi_t)$, and so $\Phi_t(h_t) \geq 0$ for all $h_t \in \mathbb{H}_t$ and $t \in [T]$ by induction.

Now suppose $\varphi_0 \leq \tau$.
The period $t = 0$ premium satisfies $\varphi_0 \geq \max\{Z_0, \ldots, Z_T\}$ (where the maximum is almost sure). Then, we have $\max\{Z_0, \ldots, Z_T\} \leq \tau$ and $Z_t(h_t) \leq \tau$ for all $h_t \in \mathbb{H}_t$ and $t \in [T]$, and so $Z_{0:T}$ is upper bounded. Since $\Phi_t(h_t) \geq \Phi_{t+1}(h_t, \xi_t)$ for all $\xi_t \in \Xi$, we have $\Phi_t(h_t) \leq \tau$ for all $h_t \in \mathbb{H}_t$ and $t \in [T]$ by induction, and so $\Phi_{0:T}$ is upper bounded.

\section{Additional Material for Section~\ref{sec:simulations}}

\subsection{Newsvendor GCR Models}

We let $V_{T+1} \equiv 0$ be the terminal value function in all of the following DP decompositions.

\paragraph*{Standard capital requirement.}
The transition functions on the information state are $g_{t}(y_t, a_t, z_t, \xi_t) = (f_t(s_t, a_t, \xi_t), \gamma(w_t - z_t))$ for all $t \in [T]$.
The DP decomposition for Eq.~\eqref{newsvendor_standard} is:
\begin{subequations}
\label{newsvendor:standard-DP}
\begin{align}
    V_0(s_0, 0) = & \min_{w_0 \in \mathbb W, a_0 \in \mathbb A, z_0 \geq 0} \left\{ w_0 + \mathbb E[V_1(g_0(y_0, a_0, z_0, \xi_0))] : R_0(s_0,a_0) + z_0 \in \mathcal A_0 \right\},\, \forall s_0 \in \mathbb{S},\\
    V_t(s_t, w_t) = & \min_{a_t \in \mathbb A, z_t \geq 0} \left\{ \mathbb E[V_{t+1} (g_t(y_t, a_t, z_t, \xi_t))] : R_t(s_t,a_t) + z_t \in \mathcal A_t \right\},\\
    & \forall y_t \in \mathbb Y, \forall t \in [1, T-1],\\
    V_T(s_T, w_T) = & \min_{a_T \in \mathbb A, z_T \geq 0} \left\{ - \beta \mathbb E[\min\{W_{T+1}, 0\}] : R_T(s_T, a_T) + z_T \in \mathcal A_T,\, w_T + \mathbb E[R_T(s_T, a_T)] - z_T \geq 0 \right\},\\
    & \forall y_T \in \mathbb Y.
\end{align}
\end{subequations}
\noindent
Note that there is only an endowment decision $w_0$ in period $t = 0$. In period $T$, the objective penalizes negative terminal wealth.

\paragraph*{Wealth reserve.}
The transition functions on the information state are
$$
g_{t}(y_t, a_t, \xi_t) = (f_t(s_t, a_t, \xi_t), \gamma(w_t + r_t(s_t, a_t, \xi_t))),\, \forall t \in [T-1].
$$
The DP decomposition for Eq.~\eqref{newsvendor_wealth} is:
\begin{align}
    V_t(s_t, w_t) = & \min_{a_t \in \mathbb A, z_t \geq 0} \left\{ z_t + \mathbb E[V_{t+1} (g_{t}(y_t, a_t, \xi_t))] : R_t(s_t, a_t) + z_t \in \mathcal A_t(w_t) \right\},\\
    & \forall y_t \in \mathbb Y,\,\forall t \in [T]. \nonumber
\end{align}

\paragraph*{Cash orders only.}
The transition functions on the information state are
$$
g_{t}(y_t, a_t, \xi_t) = (f_t(s_t, a_t, \xi_t), \gamma(w_t + r_t(s_t, a_t, \xi_t))),\, \forall t \in [T-1].
$$
The DP decomposition for Eq.~\eqref{newsvendor_order} is then:
\begin{equation}
    V_t(s_t, w_t) = \max_{a_t \in \mathbb A} \left\{ \mathbb E[R_t(s_t, a_t) + V_{t+1} (g_{t}(y_t, a_t, \xi_t))] : a_t \in \mathcal A_t(w_t) \right\}, \forall y_t \in \mathbb Y, \forall t \in [T].
\end{equation}

\paragraph*{Risk-neutral.}

The DP decomposition for Eq.~\eqref{Newsvendor_classical} is:
\begin{subequations}
\begin{align}
    V_t(s_t) = & \max_{a_t \in \mathbb A} \mathbb E[R_t(s_t, a_t) + V_{t+1} (f_{t}(s_t, a_t, \xi_t))], \forall s_t \in \mathbb{S}, \forall t \in [T-1],\\
    V_T(s_T) = & \max_{a_T \in \mathbb A} \mathbb E[R_T(s_T, a_T)], \forall s_T \in \mathbb{S}.
\end{align}
\end{subequations}

\paragraph*{Nested risk measure.}

The DP decomposition for Eq.~\eqref{Newsvendor_nested} is:
\begin{subequations}
\begin{align}
    V_t(s_t) = & \max_{a_t \in \mathbb A} \mathbb \rho_t(R_t(s_t, a_t) + V_{t+1} (f_{t}(s_t, a_t, \xi_t))), \forall s_t \in \mathbb{S}, \forall t \in [T-1],\\
    V_T(s_T) = & \max_{a_T \in \mathbb A} \mathbb \rho_T(R_T(s_T, a_T)), \forall s_T \in \mathbb{S}.
\end{align}
\end{subequations}
To evaluate the one-step risk measures $\{\rho_t\}_{t \in [T]}$, we use the variational form of CVaR for payoffs for $\alpha \in (0, 1)$:
$$
{\rm CVaR}_{\alpha}(X) = \max_{\eta \in \Re} \left\{\eta - \frac{1}{\alpha} \sum_{x} \pr{X=x} \max\{\eta - x, 0\} \right\}.
$$


\end{document}